\documentclass[multi]{cambridge7A}
\usepackage{amsmath,amsthm,amssymb,amsfonts,amscd,graphicx,enumerate}
\usepackage{pinlabel,xcolor,overpic}

\usepackage[utf8]{inputenc}
\usepackage[T1]{fontenc}

\usepackage{tikz}
\usepackage{tikz-cd}
\usepackage{fp}
\usetikzlibrary{calc}
\usetikzlibrary{shapes,arrows}
\usetikzlibrary{fit,positioning}
\usetikzlibrary{patterns,decorations.pathreplacing}

\author[Fanny Kassel]{Fanny Kassel\footnote{CNRS and Laboratoire Alexander Grothendieck, Institut des Hautes \'Etudes Scientifiques, Universit\'e Paris-Saclay, 35 route de Chartres, 91440 Bures-sur-Yvette, France. Email address: kassel@ihes.fr}}

\newtheorem{thm}{Theorem}[section]

\newtheorem{cor}[thm]{Corollary}
\newtheorem{lem}[thm]{Lemma}
\newtheorem{fact}[thm]{Fact}
\newtheorem{claim}[thm]{Claim}
\newtheorem{condition}[thm]{Condition}
\newtheorem*{LocalRigid}{Local rigidity}
\newtheorem*{MostowRigid}{Mostow rigidity}
\newtheorem*{MargulisRigid}{Margulis superrigidity}

\theoremstyle{definition}
\newtheorem{defn}[thm]{Definition}
\newtheorem{rem}[thm]{Remark}
\newtheorem{rems}[thm]{Remarks}
\newtheorem{ex}[thm]{Example}
\newtheorem{examples}[thm]{Examples}
\newtheorem{notation}[thm]{Notation}

\newcommand{\N}{\mathbb{N}}
\newcommand{\Z}{\mathbb{Z}}
\newcommand{\Q}{\mathbb{Q}}
\newcommand{\R}{\mathbb{R}}
\newcommand{\C}{\mathbb{C}}
\newcommand{\HH}{\mathbb{H}}
\newcommand{\PP}{\mathbb{P}}

\newcommand{\SL}{\mathrm{SL}}
\newcommand{\GL}{\mathrm{GL}}
\newcommand{\SO}{\mathrm{SO}}
\newcommand{\OO}{\mathrm{O}}
\newcommand{\PO}{\mathrm{PO}}

\newcommand{\PSL}{\mathrm{PSL}}
\newcommand{\PGL}{\mathrm{PGL}}
\newcommand{\Sp}{\mathrm{Sp}}
\newcommand{\U}{\mathrm{U}}
\newcommand{\SU}{\mathrm{SU}}

\newcommand{\Rrank}{\mathrm{rank}_{\R}}
\newcommand{\Hom}{\mathrm{Hom}}
\newcommand{\ie}{i.e.\ }
\newcommand{\eg}{e.g.\ }
\newcommand{\resp}{resp.\ }
\newcommand{\di}{\partial_{\infty}}
\newcommand{\Int}{\mathrm{Int}}
\newcommand{\Gr}{\mathrm{Gr}}

\newcommand{\Lambdao}{\Lambda^{\mathsf{orb}}}
\newcommand{\Ccore}{\mathcal{C}^{\mathsf{cor}}}

\newcommand*{\longhookrightarrow}{\ensuremath{\lhook\joinrel\relbar\joinrel\rightarrow}}

\begin{document}

\setcounter{chapter}{3}
\setcounter{page}{118}

\chapter{Discrete subgroups of semisimple Lie groups, beyond lattices}

\begin{abstract}
Discrete subgroups of $\SL(2,\R)$ are well understood, and classified by the geometry of the corresponding hyperbolic surfaces.
Discrete subgroups of higher-rank semisimple Lie groups, such as $\SL(n,\R)$ for $n>2$, remain more mysterious.
While lattices in this setting are rigid, there also exist more flexible, ``thinner'' discrete subgroups, which may have large and interesting deformation spaces, giving rise in particular to so-called higher Teichm\"uller theory.
We survey recent progress in constructing and understanding such discrete subgroups from a geometric and dynamical viewpoint.
\end{abstract}

\section{Introduction} \label{sec:intro}

Recall that a Lie group is a group which is also a differentiable manifold.
All Lie groups considered in these notes will be assumed to be real linear Lie groups, \ie closed subgroups of $\GL(N,\R)$ for some $N\in\N$, with finitely many connected components.
We will be specifically interested in such Lie groups which are \emph{noncompact}, since our goal is to study their \emph{infinite} discrete subgroups.

We say that a Lie group $G$ is \emph{simple} if its Lie algebra is simple, \ie nonabelian with no nonzero proper ideals; equivalently, all infinite closed normal subgroups of~$G$ have finite index in~$G$ and are nonabelian.
Simple Lie algebras have been completely classified by \'E.~Cartan, leading to a classification of simple Lie groups up to local isomorphism.
(Recall that two Lie groups $G_1$ and~$G_2$ are said to be \emph{locally isomorphic} if they have the same Lie algebra; equivalently, some finite cover of the identity component of~$G_1$ is isomorphic to some finite cover of the identity component of $G_2$.)
Noncompact simple Lie groups come in several infinite families, given in Table~\ref{table1}, and 17 (up to local isomorphism) additional groups, called \emph{exceptional} (see \eg \cite[Ch.\,X]{hel01}).
\begin{table}[h!]
\centering
\begin{tabular}{|p{1.8cm}|p{2cm}|p{2.4cm}|p{2.6cm}|}
\hline
& & \tabularnewline [-0.4cm]
& \centering Noncompact simple Lie group $G$ & \centering Maximal compact subgroup $K$ & \centering $\Rrank(G)$ \tabularnewline
\hline
\centering A & \centering $\SL(n,\C)$ & \centering $\SU(n)$ & \centering $n-1$\tabularnewline
\centering B & \centering $\SO(2n+1,\C)$ & \centering $\SO(2n+1)$ & \centering $n$\tabularnewline
\centering C & \centering $\Sp(2n,\C)$ & \centering $\Sp(n)$ & \centering $n$\tabularnewline
\centering D & \centering $\SO(2n,\C)$ & \centering $\SO(2n)$ & \centering $n$\tabularnewline
\centering A\,I & \centering $\SL(n,\R)$ & \centering $\SO(n)$ & \centering $n-1$\tabularnewline
\centering A\,II & \centering $\SU^*(2n)$ & \centering $\Sp(n)$ & \centering $n-1$\tabularnewline
\centering A\,III & \centering $\SU(p,q)$ & \centering $\mathrm{S}(\U(p)\times\U(q))$ & \centering $\min(p,q)$\tabularnewline
\centering BD\,I & \centering $\SO(p,q)_0$ & \centering $\SO(p)\times\SO(q)$ & \centering $\min(p,q)$\tabularnewline
\centering D\,III & \centering $\SO^*(2n)$ & \centering $\U(n)$ & \centering $\lfloor n/2 \rfloor$\tabularnewline
\centering C\,I & \centering $\Sp(2n,\R)$ & \centering $\U(n)$ & \centering $n$\tabularnewline
\centering C\,II & \centering $\Sp(p,q)$ & \centering $\Sp(p)\times\Sp(q)$ & \centering $\min(p,q)$\tabularnewline
\hline
\end{tabular}
\caption{List of classical noncompact simple real linear Lie groups, up to local isomorphism. Here $n,p,q\geq 1$ are integers. For types A, A\,I, and A\,II we assume $n\geq 2$, for types D and~D\,III we assume $n\geq 3$, and for type~BD\,I we assume $(p,q)\notin\{ (1,1), (2,2)\}$.}
\label{table1}
\end{table}

We say that a Lie group $G$ is \emph{semisimple} if it is locally isomorphic to a direct product $G_1\times\dots\times G_{\ell}$ of simple Lie groups~$G_i$, called the \emph{simple factors} of~$G$; in that case, if $G$ is connected and simply connected, then it is actually isomorphic to such a direct product $G_1\times\dots\times G_{\ell}$.
For instance, $\SO(2,2)$ and $\SO(4,\C)$ are semisimple (they are locally isomorphic to $\PSL(2,\R)\times\PSL(2,\R)$ and $\PSL(2,\C)\times\PSL(2,\C)$, respectively).
Any connected semisimple Lie group is the identity component (for the~real topology) of the real points of some $\R$-algebraic group (see \cite[\S\,2.14]{che51}).

Infinite discrete subgroups of semisimple Lie groups are important objects that appear in various areas of mathematics, such as geometry, complex analysis, differential equations, number theory, mathematical physics, ergodic theory, representation theory, etc.
There are many motivations for studying these discrete subgroups.
Let us mention three:
\begin{enumerate}[(1)]
  \item \emph{Historical importance:} The study of second-order linear differential equations over~$\C$, in particular by Fuchs, naturally led to the study of discrete subgroups of $\PSL(2,\C)$, in particular by Poincar\'e, and to the celebrated Uniformisation Theorem: any closed Riemann surface of genus $\geq 2$ is a quotient of the hyperbolic plane $\HH^2$ by a discrete subgroup $\Gamma$ of $\PSL(2,\R)$.
  See \eg \cite{saint-gervais-unif} for details.
  \item \emph{Locally symmetric spaces:} Any discrete subgroup $\Gamma$ of a noncompact semisimple Lie group~$G$ defines a Riemannian locally symmetric space $\Gamma\backslash G/K$, where $K$ is a maximal compact subgroup of~$G$.
  These locally symmetric spaces, which include real hyperbolic manifolds $\Gamma\backslash\HH^n$ for $G=\PO(n,1)=\OO(n,1)/\{\pm\mathrm{I}\}$, are geometrically important.
  They naturally appear in representation theory and harmonic analysis, where symmetric spaces $G/K$ play a central role (see \eg \cite{bfs97}).
  \item \emph{Geometric structures on manifolds:} A modern point of view on geometry, which dates back to Klein's 1872 Erlangen program and which has been much developed in the twentieth century especially through the work of Ehresmann and Thurston, is to study manifolds that ``locally look like'' some ``model spaces'' with large ``symmetry groups''.
  Model spaces are typically homogeneous spaces $X=G/H$ where $G$ is a real Lie group (often semisimple).
  Important examples include $X=G/K$ as above, but also $(X,G) = (\mathbb{RP}^n,\PGL(n+1,\R))$ (real projective geometry), $(\mathbb{CP}^n,\PGL(n+1,\C))$ (complex projective geometry), or $(\HH^{p,q},\PO(p,q+1))$ (pseudo-Riemannian hyperbolic geometry in signature $(p,q)$).
  See \cite{gol22-book} for details.
\end{enumerate}

An important class of discrete subgroups of noncompact semisimple Lie groups is the class of \emph{lattices}, namely discrete subgroups of finite covolume for the Haar measure (see Section~\ref{sec:lattices} below).
They play an important role in several fields of mathematics, in addition to the above, such as:
\begin{itemize}
  \item geometric group theory (lattices are finitely presented groups with many desirable properties --- \eg lattices of $\SL(n,\R)$ for $n\geq 3$ have Kazhdan's property (T)),
  \item combinatorics (construction of expander graphs),
  \item number theory (arithmetic groups),
  \item ergodic theory (flows on $\Gamma\backslash G$) and homogeneous dynamics.
\end{itemize}
See \eg \cite{wit15} and references therein.
In some of these settings (in particular ergodic theory and homogeneous dynamics), there is currently active research aiming to extend, to classes of discrete subgroups of infinite covolume, classical results involving lattices.
Infinite-index subgroups of arithmetic groups (and particularly those that are still Zariski-dense, named \emph{thin groups} by Sarnak) have also attracted considerable interest recently, see \eg \cite{kllr19}.

In these notes, we will review a few properties of lattices, and then~fo\-cus on the problem of finding other large classes of infinite discrete subgroups $\Gamma$ of semisimple Lie groups~$G$ with desirable properties, including:
\begin{enumerate}[(1)]
  \item the existence of examples with interesting geometric interpretations,
  \item a good control of the subgroups' behaviour under deformation,
  \item interesting dynamics of $\Gamma$ on certain homogeneous spaces of~$G$.
\end{enumerate}
These properties are typically invariant under replacing $\Gamma$ by a finite-index subgroup.
This will allow us to sometimes reduce to \emph{torsion-free} $\Gamma$: indeed, the Selberg lemma \cite[Lem.\,8]{sel60} states that any finitely generated subgroup of~$G$ admits a finite-index subgroup which is torsion-free.

\section{Lattices} \label{sec:lattices}

Let $G$ be a noncompact semisimple Lie group.
It admits a \emph{Haar measure}, \ie a nonzero Radon measure which is invariant under left and right multiplication; this measure is unique up to scaling.

\begin{defn}
A \emph{lattice} of~$G$ is a discrete subgroup $\Gamma$ of~$G$ such that the quotient $\Gamma\backslash G$ has finite volume for the measure induced by the Haar measure of~$G$.
\end{defn}

If $\Gamma$ is a lattice of~$G$, then the quotient $\Gamma\backslash G$ can be compact (in which case we say that $\Gamma$ is a \emph{cocompact} or \emph{uniform} lattice) or not.

A fundamental result of Borel and Harish-Chandra \cite{bor63,bhc62} states that $G$ always admits both cocompact lattices and noncocompact lattices.

Borel's Density Theorem \cite{bor60} states that lattices are Zariski-dense in~$G$ as soon as $G$ is connected and has no compact simple factors.
This means that if the set of real points of some $\R$-algebraic group contains a lattice of~$G$, then it actually contains the whole of~$G$.

We say that a lattice $\Gamma$ of~$G$ is \emph{irreducible} if for any noncompact, infinite-index, closed normal subgroup $G'$ of~$G$, the projection of $\Gamma$ to $G/G'$ is nondiscrete.
(This is automatically satisfied if $G$ is simple.)

\subsection{Geometric interpretation} \label{subsec:G/K}

Lattices of~$G$ can be characterised by their action on the \emph{Riemannian symmetric space} of~$G$.
Let us recall what this fundamental object is (see \eg \cite{ebe96,hel01} for details).

As mentioned in the introduction, $G$ admits a maximal compact subgroup~$K$.
It is unique up to conjugation, and so the quotient $G/K$ is uniquely defined.
For instance, if $G=\SL(n,\R)$, then $K=\SO(n)$ up to conjugation, and $G/K$ identifies with the space of ellipsoids of~$\R^n$ of volume~$1$; if $G=\SL(n,\C)$, then $K=\SU(n)$ up to conjugation.
See Table~\ref{table1} for further examples.

The group $K$ is the set of fixed points of some involution $\theta$ of~$G$, called a Cartan involution.
This yields a splitting of the Lie algebra $\mathfrak{g}$ of~$G$ as the direct sum of two linear subspaces, namely the subspace $\mathfrak{g}^{\mathrm{d}\theta}$ of fixed points of $\mathrm{d}\theta$ (which is the Lie algebra of~$K$) and the subspace $\mathfrak{g}^{-\mathrm{d}\theta}$ of anti-fixed points of $\mathrm{d}\theta$.
The tangent space $T_{eK}(G/K)$ to $G/K$ at the origin identifies with $\mathfrak{g}^{-\mathrm{d}\theta}$, on which there is a natural $K$-invariant positive definite symmetric bilinear form, the Killing form.
Pushing forward this bilinear form by elements of~$G$ yields a $G$-invariant Riemannian metric on $G/K$.
With this metric, $G/K$ has nonpositive sectional curvature and is a \emph{symmetric space}: at every point, the geodesic symmetry sending $\exp(tv)$ to $\exp(-tv)$ (where $v$ is a tangent vector) is an isometry.

Since $K$ is compact, any discrete subgroup $\Gamma$ of~$G$ acts properly discontinuously on $G/K$.
The subgroup $\Gamma$ is a lattice if and only if the quotient $\Gamma\backslash G/K$ has finite volume, which is equivalent to the action of $\Gamma$ on $G/K$ admitting a fundamental domain of finite volume.

\subsection{Examples} \label{subsec:examples}

The following fundamental example goes back to Minkowski.

\begin{ex} \label{ex:SLnZ}
The group $\Gamma = \SL(n,\Z)$ is a noncocompact lattice in $G = \SL(n,\R)$.
\end{ex}

Let us briefly explain how to see this, starting with the case $n=2$.

For $n=2$, the Riemannian symmetric space $G/K$ is the hyperbolic plane $\HH^2 \simeq \{ z=x+iy\in\C ~|~ y=\mathrm{Im}(z)>0\}$ with its $G$-invariant metric $\mathrm{d}s^2 = (\mathrm{d}x^2 + \mathrm{d}y^2)/y^2$, on which $G = \SL(2,\R)$ acts by M\"obius transformations: $\big(\begin{smallmatrix} \mathtt{a} & \mathtt{b}\\ \mathtt{c} & \mathtt{d}\end{smallmatrix}\big) \cdot z = \frac{\mathtt{a}z+\mathtt{b}}{\mathtt{c}z+\mathtt{d}}$.
It is an easy exercise to check that
$$\mathcal{D} := \Big\{ z\in\HH^2 ~\Big|~ |\mathrm{Re}(z)| \leq \frac{1}{2} \quad\mathrm{and}\quad |z|\geq 1\Big\}$$
(see Figure~\ref{fig:SL2Z}) is a finite-volume fundamental domain for the action of $\Gamma=\SL(2,\Z)$ on~$\HH^2$.
(Use that $\Gamma$ is generated by $\big(\begin{smallmatrix} 1 & 1\\ 0 & 1\end{smallmatrix}\big)$ and $\big(\begin{smallmatrix} 0 & 1\\ -1 & 0\end{smallmatrix}\big)$ and that the $G$-invariant volume form on~$\HH^2$ is given by $\mathrm{d}\mathrm{vol} = \mathrm{d}x \, \mathrm{d}y/(4y^2)$.)
Therefore $\Gamma$ is a lattice in $G = \SL(2,\R)$.
This lattice is not cocompact since for any $\gamma = \big(\begin{smallmatrix} \mathtt{a} & \mathtt{b}\\ \mathtt{c} & \mathtt{d}\end{smallmatrix}\big) \in \Gamma$ we have $\mathrm{Im}(\gamma\cdot i) = 1/(\mathtt{c}^2+\mathtt{d}^2) \leq 1$, hence there exist points of~$\HH^2$ (\eg $ti$ with $t>0$ large) that are arbitrarily far away from any point of the $\Gamma$-orbit of $i$ in~$\HH^2$.

For general $n\geq 2$, we can use the classical \emph{Iwasawa decomposition} $G = NAK$, where  $N$ (\resp $A$) is the subgroup of $G = \SL(n,\R)$ consisting of upper triangular unipotent (\resp positive diagonal) matrices and $K = \SO(n)$.
This means that any element $g\in G$ can be written in a unique way as $g = nak$ where $n\in N$, $a\in A$, and $k\in K$.
A finite-volume fundamental domain for the action of $\Gamma$ on $G/K$ is given by the \emph{Siegel set} $\mathcal{S}$ consisting of those elements of $G/K$ of the form $naK$ with $n\in N$ having all entries above the diagonal in $[-1/2,1/2]$ and $a = \mathrm{diag}(a_1,\dots,a_n) \in A$ satisfying $|a_i/a_{i+1}|\geq\sqrt{3}/2$ for all $1\leq i\leq n-1$.

\begin{figure}[ht!]
\begin{overpic}[scale=1,percent]{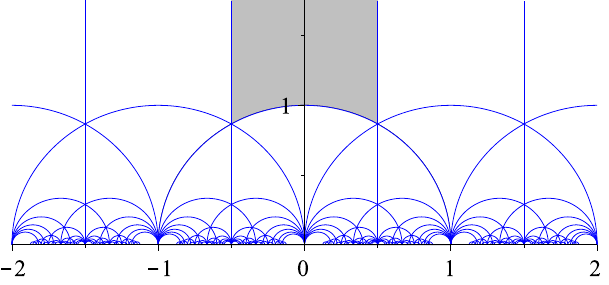}
\put(55,38){$\mathcal{D}$}
\end{overpic}
\caption{Fundamental domains for the action of $\SL(2,\Z)$ on the upper half plane model of~$\HH^2$}
\label{fig:SL2Z}
\end{figure}

Generalising Example~\ref{ex:SLnZ}, a fundamental result of Borel and Harish-Chandra \cite{bhc62} states that if $\mathbf{G}$ is a semisimple $\Q$-algebraic group, then $\mathbf{G}_{\Z}$ is a lattice in $\mathbf{G}_{\R}$.
Godement's cocompactness criterion (see \eg \cite[\S\,2.8]{ben09}) states that this lattice is cocompact if and only if it does not contain any nontrivial unipotent elements.

We now give a concrete example of a cocompact lattice (see \cite[\S\,2]{ben09}).

\begin{ex} \label{ex:unif-lattice}
For $p,q\geq 1$ with $p+q=n\geq 3$, consider the block diagonal matrix
$$J_{p,q} := \begin{pmatrix} \mathrm{I}_p & 0 \\ 0 & -\sqrt{2}\,\mathrm{I}_q\end{pmatrix}$$
and the Lie group $G := \SO(J_{p,q},\R) \simeq \SO(p,q)$.
Then $\Gamma :=\linebreak G \cap \SL(n,\Z[\sqrt{2}])$ is a cocompact lattice in~$G$.
\end{ex}

In order to see this, we can apply Weil's trick of ``restriction of scalars''.
Namely, consider the automorphism $\sigma$ of $\SL(n,\Z[\sqrt{2}])$ obtained by applying the Galois conjugation $x+\sqrt{2}y \mapsto x-\sqrt{2}y$ of $\Q[\sqrt{2}]$ to each entry.
Let $J_{p,q}^{\sigma}$ be the image of $J_{p,q}$ under~$\sigma$, and $\mathbf{H}$ the semisimple algebraic subgroup of $\mathbf{GL}_{2n}$ whose set $\mathbf{H}_{\C}$ of complex points consists of those block matrices of the form
$$h := \begin{pmatrix} a & 2b\\ b & a\end{pmatrix} \in \GL(2n,\C)$$
with $\varphi_+(h) := a + \sqrt{2}b \in \SO(J_{p,q},\C)$ and $\varphi_-(h) := a - \sqrt{2}b \in \SO(J_{p,q}^{\sigma},\C)$.
An elementary computation (or more abstractly the fact that the family of polynomial equations defining~$\mathbf{H}$ is invariant under~$\sigma$) shows that $\mathbf{H}$ is a $\Q$-algebraic group.
We have isomorphisms
$$\left\{\hspace{-0.2cm}\begin{array}{l}
\mathbf{H}_{\R} \overset{(\varphi_+,\varphi_-)}{\simeq} \SO(J_{p,q},\R) \times \SO(J_{p,q}^{\sigma},\R) = G \times \SO(J_{p,q}^{\sigma},\R),\\
\mathbf{H}_{\Z} \overset{\varphi_+}{\simeq} \Gamma,
\end{array}\right.$$
where $\SO(J_{p,q}^{\sigma},\R) \simeq \SO(n)$ is compact.
The group $\mathbf{H}_{\Z}$ is a lattice in $\mathbf{H}_{\R}$, hence $\Gamma$ is a lattice in~$G$.
Moreover, $\mathbf{H}_{\Z}$ does not contain any nontrivial unipotent elements since $\varphi_-$ takes $\Gamma$ to a subgroup of a \emph{compact} group, hence without nontrivial unipotent elements, and a homomorphism of algebraic groups takes unipotent elements to unipotent elements.
Godement's criterion then ensures that $\mathbf{H}_{\Z}\backslash\mathbf{H}_{\R}$ is compact, and so $\Gamma\backslash G$ is compact too.

\medskip

In both Examples \ref{ex:SLnZ} and~\ref{ex:unif-lattice}, the group $\Gamma$ is \emph{arithmetic} in~$G$, \ie there is a homomorphism $\pi : \mathbf{H}\to\mathbf{G}$ of semisimple $\Q$-algebraic groups such that $G = \mathbf{G}_{\R}$, such that the kernel of $\pi$ in $\mathbf{H}_{\R}$ is compact, and such that $\Gamma$ is commensurable to $\pi(\mathbf{H}_{\Z})$ (see \cite{wit15}).

Nonarithmetic lattices are known to exist in $G=\SO(n,1)$ for any $n\geq 2$: examples were constructed by Vinberg \cite{vin68} for small~$n$ using reflection groups, then by Gromov and Piatetski-Shapiro \cite{gps88} for any~$n$.
Later, different examples were constructed by Agol \cite{ago06} and Belolipetsky--Thomson \cite{bt10} (see also the very recent work \cite{dou23}) in the form of lattices of $\SO(n,1)$ whose systole (\ie length of the shortest closed geodesic) is arbitrarily small.
(Due to a separability property later established in \cite[Cor.\,1.12]{bhw11}, Agol's construction \cite{ago06} actually works for any~$n$.)
Finitely many commensurability classes of nonarithmetic lattices are also known in $\SU(2,1)$ and $\SU(3,1)$ by Deligne--Mostow \cite{dm86,mos80} and Deraux--Parker--Paupert \cite{der20,dpp16}.
It is an open question whether nonarithmetic lattices exist in $\SU(n,1)$ for $n>3$.

On the other hand, in noncompact simple Lie groups which are not locally isomorphic to $\SO(n,1)$ or $\SU(n,1)$, all lattices are arithmetic (as a consequence of superrigidity, see Section~\ref{subsec:lattice-deform}).

\subsection{Rank one versus higher rank} \label{subsec:rank}

The \emph{real rank} of a semisimple Lie group is an integer defined as follows.

\begin{defn}
The \emph{real rank} of~$G$, denoted $\Rrank(G)$, is the maximum dimension of a closed connected subgroup of~$G$ which is diagonalisable over~$\R$; equivalently, for noncompact~$G$, it is the maximum dimension of a totally geodesic subspace of the Riemannian symmetric space $G/K$ which is flat (\ie of constant zero sectional curvature).
\end{defn}

The real rank is invariant under local isomorphism, and the real rank of a product is the sum of the real ranks of the factors.
We refer to Table~\ref{table1} for the real ranks of the classical noncompact simple Lie groups.
A compact Lie group has real rank~$0$.

The simple Lie groups of real rank~$1$ are, up to local isomorphism, $\SO(n,1)$, $\SU(n,1)$, $\Sp(n,1)$ for $n\geq 2$, and the exceptional group $F_{4(-20)}$.
(Note that $\PSL(2,\R)\simeq\SO(2,1)_0$ and $\PSL(2,\C)\simeq\SO(3,1)_0$, where the subscript $0$ denotes the identity components.)

Semisimple Lie groups $G$ of real rank~$1$ are characterised by the fact that the sectional curvature of the corresponding Riemannian symmetric space $G/K$ is everywhere $<0$.
(In fact, the curvature is then \emph{pinched}, \ie contained in an interval of the form $[\alpha,\beta]$ where $\alpha\leq\beta <0$.)
This implies that the geodesic metric space $G/K$ is \emph{Gromov hyperbolic}, meaning that there exists $\delta\geq 0$ such that all geodesic triangles $(a,b,c)$ of $G/K$ are $\delta$-thin: the side $[a,b]$ lies in the uniform $\delta$-neighbourhood of the union $[b,c]\cup [c,a]$ of the other two sides (see Figure~\ref{fig:thin-triangle}).
On the other hand, when $r := \Rrank(G) \geq 2$, the Riemannian symmetric space $G/K$ is only nonpositively curved, and not Gromov hyperbolic; its geometry is somewhat more complicated due to the presence of \emph{flats} (\ie isometric copies of Euclidean~$\R^r$, where the curvature vanishes).

\begin{figure}[ht!]
\begin{overpic}[scale=0.35,percent]{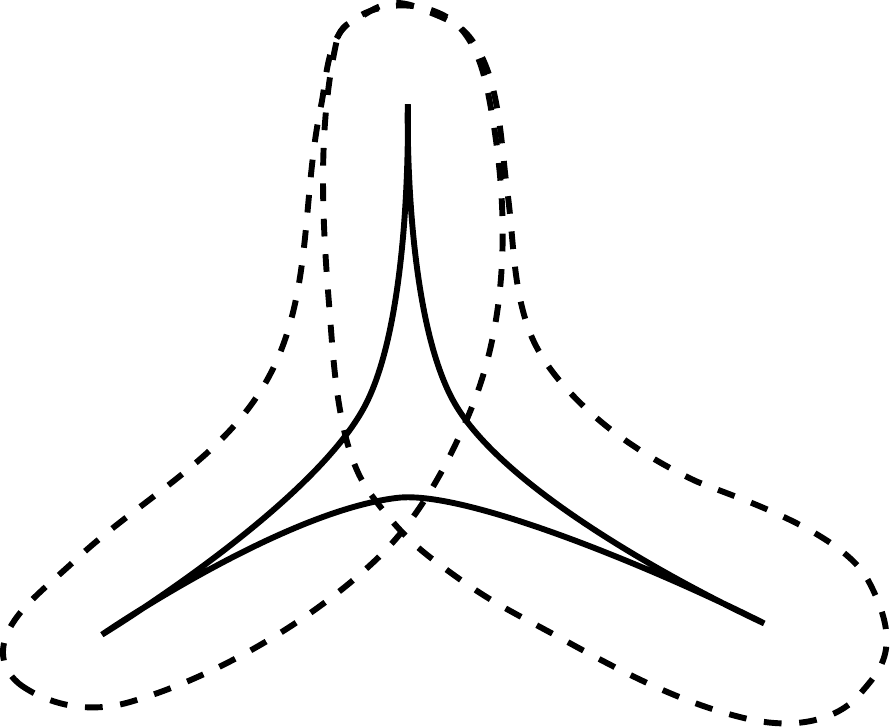}
\put(7,7){$a$}
\put(88,9){$b$}
\put(45,73){$c$}
\end{overpic}
\caption{A $\delta$-thin triangle in a geodesic metric space. The side $[a,b]$ is contained in the union of the uniform $\delta$-neighbourhoods (indicated by dashes) of the sides $[b,c]$ and $[c,a]$.}
\label{fig:thin-triangle}
\end{figure}

There are a number of differences between lattices in real rank one and lattices in higher real rank.

One difference concerns hyperbolicity.
Namely, if $\Rrank(G)=1$, then
\begin{itemize}
  \item any cocompact lattice $\Gamma$ of~$G$ is \emph{Gromov hyperbolic}, \ie $\Gamma$ acts properly discontinuously, by isometries, with compact quotient, on a Gromov hyperbolic proper geodesic metric space~$X$ (\eg $X=G/K$);
  \item any noncocompact lattice $\Gamma$ of~$G$ is \emph{relatively hyperbolic} with respect to some collection $\mathcal{P}$ of subgroups which are virtually (\ie up to finite index) nilpotent: this means that $\Gamma$ acts properly discontinuously by isometries on some visual Gromov hyperbolic proper metric space $X$ (\eg $X=G/K$), and with compact quotient on some closed subset of~$X$ of the form $X \smallsetminus \bigcup_{P\in\mathcal{P}} B_P$ where each $B_P$ is a $P$-invariant open horoball of~$X$ and $B_P \cap B_{P'} = \emptyset$ for $P\neq P'$ (see \cite[\S\,4]{hh20} for details).
\end{itemize}
On the other hand, if $\Rrank(G)\geq\nolinebreak 2$, then lattices of~$G$ are never Gromov hyperbolic, nor relatively hyperbolic with respect to any collection of subgroups \cite{bdm09}.
This follows from the fact that these groups are \emph{metrically thick} in the sense of \cite{bdm09} (see \cite{kl97} for cocompact lattices).
In fact, if $\Rrank(G)\geq 2$, then any isometric action of a lattice $\Gamma$ of~$G$ on a Gromov hyperbolic metric space $X$ is ``trivial'' (\ie admits a global fixed point in $X$ or its boundary), unless it is obtained by projecting $\Gamma$ to a rank-one factor of~$G$ \cite{bcfs22,hae20}.

More generally, lattices $\Gamma$ in simple Lie groups $G$ with $\Rrank(G)\geq\nolinebreak 2$ tend to have global fixed points when they act on various classes of spaces.
For instance, any continuous action by affine isometries of $\Gamma$ on a Hilbert space has a global fixed point.
This property, which is equivalent to Kazhdan's property~(T) (see \cite[Ch.\,13]{wit15}), is also satisfied by lattices in the rank-one Lie groups $\Sp(n,1)$ with $n\geq 2$ or $F_{4(-20)}$.
However, other fixed point properties actually distinguish higher rank from rank one.
For instance, for any simple Lie group $G$ with $\Rrank(G)\geq 2$ and any $\sigma$-finite positive measure $\nu$ on a standard Borel space, any continuous affine isometric action of a lattice of $G$ on $L^p(\nu)$ for $1<p<+\infty$ has a global fixed point, by Bader--Furman--Gelander--Monod; on the other hand, by Pansu~and Bourdon--Pajot, any cocompact lattice $\Gamma$ in a simple Lie group $G$ with $\Rrank(G)=\nolinebreak 1$ (and more generally, any Gromov hyperbolic group $\Gamma$) admits fixed-point-free affine isometric actions on $L^p(\Gamma)$ whose linear part is the regular representation, for any $p>1$ large enough.
See \cite{bfgm07}.

Another difference between real rank one and higher real rank concerns normal subgroups.
Namely, if $\Rrank(G)=1$, then lattices of~$G$ have many normal subgroups (see Gromov \cite{gro87}); in fact, if $\Gamma$ is a lattice of~$G$, then any countable group can be embedded into a quotient of $\Gamma$ by some normal subgroup (this ``universality'' property holds for all relatively hyperbolic groups \cite{amo07}).
On the other hand, if $\Rrank(G)\geq 2$, then all normal subgroups of an irreducible lattice $\Gamma$ of~$G$ are finite or finite-index in~$\Gamma$ (this is Margulis's Normal Subgroups Theorem, see~\cite{mar91}).

Note that for an irreducible lattice $\Gamma$ of~$G$, the finite normal subgroups of~$\Gamma$ are easy to describe: for connected~$G$, they are the subgroups of the finite abelian group $\Gamma\cap Z(G)$ (using Borel's Density Theorem \cite{bor60}).
On the other hand, much more effort is required to understand the finite-index normal subgroups of~$\Gamma$.
By \cite{bms67}, for $\Gamma=\SL(n,\Z)$ with $n\geq 3$, any finite-index normal subgroup of~$\Gamma$ is a \emph{congruence subgroup}, \ie contains the kernel of the natural projection $\SL(n,\Z)\to\SL(n,\Z/m\Z)$ for some $m\geq 1$; this is false for $\Gamma = \SL(2,\Z)$.
In general, it is conjectured that lattices of~$G$ have a slightly weaker form of this ``Congruence Subgroup Property'' if and only if $\Rrank(G)\geq 2$: see \cite{sur03}.

We now discuss some rigidity results for representations of lattices inside~$G$, which hold in particular for $\Rrank(G)\geq 2$.

\subsection{Deformations and rigidity} \label{subsec:lattice-deform}

Let $\Gamma$ be a discrete subgroup of~$G$.
We denote by $\Hom(\Gamma,G)$ the space of representations of $\Gamma$ to~$G$, endowed with the compact-open topology (if $\Gamma$ admits a finite generating subset~$F$, then this coincides with the topology of pointwise convergence on~$F$).

By a \emph{continuous deformation} of $\Gamma$ in~$G$ we mean a continuous path $(\rho_t)_{t\in [0,1)}$ in $\Hom(\Gamma,G)$ where $\rho_0$ is the natural inclusion of $\Gamma$ in~$G$.
Certain continuous deformations of $\Gamma$ in~$G$ are considered \emph{trivial}: namely, those of the form $\rho_t = g_t\,\rho_0(\cdot)\,g_t^{-1}$ where $(g_t)_{t\in [0,1)}$ is a continuous path in~$G$ (and $g_0$ is the identity element).
In other words, if $\Hom(\Gamma,G)/G$ denotes the quotient of $\Hom(\Gamma,G)$ by the natural action of $G$ by conjugation at the target, then the trivial deformations are those whose image in $\Hom(\Gamma,G)/G$ are constant.

For $G = \PSL(2,\R) \simeq \SO(2,1)_0$, torsion-free lattices $\Gamma$ of~$G$ admit many nontrivial continuous deformations.
Indeed, if $\Gamma$ is noncocompact in~$G$, then $\Gamma$ is a nonabelian free group on finitely many generators $\gamma_1,\dots,\gamma_m$, and the natural inclusion $\rho_0 : \Gamma\hookrightarrow G$ can be continuously deformed by deforming independently the image of each~$\gamma_i$; the map $\rho \mapsto (\rho(\gamma_1),\dots,\rho(\gamma_m))$ yields an isomorphism $\Hom(\Gamma,G) \simeq G^m$.
If $\Gamma$ is cocompact in~$G$, then $\Gamma$ identifies with the fundamental group of the closed hyperbolic surface $S := \Gamma\backslash\HH^2$; the connected component of the natural inclusion $\rho_0$ in $\Hom(\Gamma,G)$ consists entirely of injective and discrete representations \cite{gol-PhD}, and its image in $\Hom(\Gamma,G)/G$ is homeomorphic to $\R^{6g-6}$: it is the \emph{Teichm\"uller space} of~$S$.

On the other hand, a number of rigidity results have been proved for lattices in other noncompact semisimple Lie groups~$G$, including local rigidity, Mostow rigidity, and Margulis superrigidity, which we now briefly state and comment on.
See \cite{fis-survey-margulis,pan-bourbaki} for details and references.

\begin{LocalRigid}[Selberg, Calabi, Weil, Garland--Raghunathan]
Let $G$ be a semisimple Lie group with no simple factors that are compact or locally isomorphic to $\PSL(2,\R)$ (\resp $\PSL(2,\mathbb{K})$ with $\mathbb{K}=\R$ or~$\C$).
If $\Gamma$ is a cocompact (\resp noncocompact) irreducible lattice of~$G$, then any continuous deformation of $\Gamma$ in~$G$ is trivial.
\end{LocalRigid}

Note that noncocompact lattices of $G=\PSL(2,\C)$ are not locally rigid: they can be deformed using Thurston's \emph{hyperbolic Dehn surgery} theory.
However, they do not admit nontrivial deformations sending unipotent elements to unipotent elements.

Local rigidity is an important ingredient in the proof of Wang's finiteness theorem, which states that if $G$ is simple and not locally isomorphic to $\PSL(2,\mathbb{K})$ with $\mathbb{K}=\R$ or~$\C$, then for any $v>0$ there are only finitely many conjugacy classes of lattices of~$G$ with covolume $\leq v$.

\begin{MostowRigid}[Mostow, Prasad, Margulis]
Let $G,G'$ be connected semisimple Lie groups, with trivial centre, and with no simple factors that are compact or locally isomorphic to $\PSL(2,\R)$.
If $\Gamma$ and~$\Gamma'$ are irreducible lattices of $G$ and~$G'$, respectively, then any isomorphism between $\Gamma$ and~$\Gamma'$ extends to a continuous isomorphism between $G$ and~$G'$.
\end{MostowRigid}

This implies (see Section~\ref{subsec:G/K}) that the fundamental group of any locally symmetric space $\Gamma\backslash G/K$ completely determines its geometry.

\begin{MargulisRigid}[{Margulis, Corlette, Gromov--Schoen, see \eg \cite[Th.\,16.1.4]{wit15}}]
Let $G$ be a noncompact semisimple Lie group which is connected, algebraically simply connected, and not locally isomorphic to the product of $\SO(n,1)$ or $\SU(n,1)$ with a compact Lie group.
Then any irreducible lattice $\Gamma$ of~$G$ is \emph{superrigid}, in the sense that any representation $\rho : \Gamma\to\GL(d,\R)$ (for any $d\geq 2$) continuously extends to~$G$ up to finite index and to bounded error.
\end{MargulisRigid}

Here ``$\rho$ continuously extends to~$G$ up to finite index and to bounded error'' means that there exist a finite-index subgroup $\Gamma'$ of~$\Gamma$, a continuous homomorphism $\rho_G : G\to\GL(d,\R)$, and a compact subgroup $C$ of $\GL(d,\R)$ centralising $\rho_G(G)$ such that $\rho(\gamma) \in \rho_G(\gamma) C$ for all $\gamma\in\Gamma'$.
Under an appropriate assumption on the image of~$\rho$, we can take $C$ to be trivial.
``Algebraically simply connected'' is a technical assumption which is always satisfied up to passing to a finite cover: see \cite[\S\,16.1]{wit15}.

Margulis used his superrigidity (over $\R$ as above, but also over non-Archimedean local fields) to prove that if $G$ is semisimple with no compact simple factors and if $\Rrank(G)\geq 2$, then all irreducible lattices $\Gamma$ of~$G$ are arithmetic in the sense of Section~\ref{subsec:examples}.
The same conclusion holds when $G$ is locally isomorphic to $\Sp(n,1)$ with $n\geq 2$ or $F_{4(-20)}$, as superrigidity holds for these rank-one groups as well.

Margulis superrigidity was further extended by Zimmer into a rigidity result for cocycles, see \cite{fm03}.
This was the starting point of important new directions of research at the intersection of group theory and dynamics (see \eg \cite{fur11}), including the so-called \emph{Zimmer program} (see \cite{can-bourbaki,fis-icm}).
The idea of this program is the following: for a lattice $\Gamma$ in a simple Lie group~$G$ with $\Rrank(G)\geq 2$, Margulis superrigidity states that any linear representation of~$\Gamma$ essentially comes from a linear representation of~$G$; in particular, the minimal dimension of a finite-kernel linear representation of~$\Gamma$ is equal to the minimal dimension of a finite-kernel linear representation of~$G$.
Zimmer asked whether this last property has a nonlinear analogue, for actions by diffeomorphisms of $\Gamma$ on closed manifolds: namely, is the minimal dimension of a closed manifold on which $\Gamma$ acts faithfully by diffeomorphisms equal to the minimal dimension of a closed manifold on which $G$ (or a compact form of the complexification of~$G$) acts faithfully by diffeomorphisms?
Brown, Fisher, and Hurtado have recently answered this question positively in many cases, building on new developments in dynamics and on recent strengthenings of Kazhdan's property~(T): see \cite{can-bourbaki,fis-icm}.
This has led to intense research activity around rigidity questions for actions by diffeomorphisms of higher-rank lattices on manifolds.

\section{A change of paradigm} \label{sec:flex-examples}

We just saw that many important rigidity results have been established for lattices since the 1960s, particularly in higher real rank, and that this topic is still very active.
On the other hand, since the 1990s and early 2000s, there has been growing interest in \emph{flexibility}: namely, there has been increasing effort to find and study infinite discrete subgroups of semisimple Lie groups which are more flexible than lattices, and which in certain cases can have large deformation spaces.
Such discrete subgroups have been known to exist for a long time in real rank one, whereas the investigation of their analogues in higher real rank has gathered momentum only much more recently.
We present a few examples below.

To be more precise, we are interested in infinite discrete subgroups $\Gamma$ of semisimple Lie groups~$G$ that admit continuous deformations $(\rho_t)_{t\in [0,1)}\linebreak\subset\Hom(\Gamma,G)$ as in Section~\ref{subsec:lattice-deform} which, not only are nontrivial, but also satisfy that each $\rho_t$ is injective with discrete image, so that the $\rho_t(\Gamma)$ for $t>0$ are still discrete subgroups of~$G$ isomorphic (but not conjugate) to~$\Gamma$.
An ideal situation is when the natural inclusion $\rho_0 : \Gamma\hookrightarrow G$ admits a full open neighbourhood in $\Hom(\Gamma,G)$ consisting entirely of injective and discrete representations, with a nonconstant image in $\Hom(\Gamma,G)/G$.

We are thus led, for given discrete subgroups $\Gamma$ of~$G$, to study subsets of $\Hom(\Gamma,G)$ consisting of injective and discrete representations, and their images in the corresponding character varieties.
In this framework, we discuss so-called \emph{higher Teichm\"uller theory} in Section~\ref{subsec:higher-Teich} below.

\begin{rem} \label{rem:two-points-of-view}
In the sequel, we go back and forth between two equivalent points of view: studying discrete subgroups $\Gamma$ of~$G$, or fixing an abstract group~$\Gamma_0$ and studying the injective and discrete representations of $\Gamma_0$ into~$G$ (corresponding to the various ways of realising $\Gamma_0$ as a discrete subgroup of~$G$).
We sometimes allow ourselves to weaken ``injective'' into ``finite-kernel''.
\end{rem}

\subsection{Examples in real rank one} \label{subsec:ex-rank-1}

Examples of flexible discrete subgroups in real rank one include classical Schottky groups (which are nonabelian free groups), quasi-Fuchsian groups (which are closed surface groups), as well as other discrete subgroups which are fundamental groups of higher-dimensional manifolds.
We briefly review such examples, referring to \cite{kap07,mas88} for more details.

\subsubsection{Schottky groups}

For $n\geq 2$, let $X = \HH^n$ be the real hyperbolic space of dimension~$n$, with visual boundary $\di X \simeq \mathbb{S}^{n-1}$.
Concretely, choosing a symmetric bilinear form $\langle\cdot,\cdot\rangle_{n,1}$ of signature $(n,1)$ on~$\R^{n+1}$, we can realise $X$ as the open subset
\begin{equation} \label{eqn:proj-model-Hn}
\HH^n = \{ [v]\in\PP(\R^{n+1}) ~|~ \langle v,v\rangle_{n,1}<0\}
\end{equation}
of the real projective space $\PP(\R^{n+1})$ and $\di X$ as the boundary of $X$ in $\PP(\R^{n+1})$.
The geodesics of~$X$ are then the nonempty intersections of $X$ with projective lines of $\PP(\R^{n+1})$, the geodesic copies of $\HH^{n-1}$ in~$X$ are the nonempty intersections of $X$ with projective hyperplanes of $\PP(\R^{n+1})$, and the isometry group $G = \mathrm{Isom}(X)$ of~$X$ is $\PO(n,1) = \OO(n,1)/\{\pm\mathrm{I}\}$.

An \emph{open disk} in $\di X$ is the boundary at infinity of an open half-space of~$X$, bounded by a geodesic copy of~$\HH^{n-1}$.
(For $n=2$, open disks are just open intervals in $\di X \simeq \mathbb{S}^1$.)

For $m\geq 2$, choose $2m$ pairwise disjoint open disks $B_1^{\pm}, \dots, B_m^{\pm}$ in $\di X$, such that $\di X \smallsetminus \bigcup_{i=1}^m (B_i^- \cup B_i^+)$ has nonempty interior, and elements $\gamma_1,\dots,\gamma_m\in G$ such that $\gamma_i\cdot\mathrm{Int}(\di X\smallsetminus B_i^-) = B_i^+$ for all~$i$.
Let $\Gamma$ be the subgroup of~$G$ generated by $\gamma_1,\dots,\gamma_m$.

\begin{claim} \label{claim:ping-pong-rank-one}
The group $\Gamma$ is a nonabelian free group with free generating subset $\{\gamma_1,\dots,\gamma_m\}$.
It is discrete in~$G$.
\end{claim}

\begin{proof}
Consider any reduced word $\gamma = \gamma_{i_1}^{\sigma_1} \dots \gamma_{i_N}^{\sigma_N}$ in the alphabet $\{ \gamma_1^{\pm 1},\dots,\gamma_m^{\pm 1}\}$, where $1\leq i_j\leq m$ and $\sigma_j\in\{\pm 1\}$ for all $1\leq\nolinebreak j\leq\nolinebreak N$.
Since $\gamma_{i_j}^{\sigma_j}\cdot\mathrm{Int}(\di X\smallsetminus B_{i_j}^{-\mathrm{sign}(\sigma_j)}) = B_{i_j}^{\mathrm{sign}(\sigma_j)}$ for all~$j$ and since $B_{i_j}^{\mathrm{sign}(\sigma_j)} \subset \mathrm{Int}(\di X\smallsetminus B_{i_{j-1}}^{-\mathrm{sign}(\sigma_{j-1})})$ for $j\geq 2$, we see that the element of~$\Gamma$ corresponding to~$\gamma$ sends $\di X \smallsetminus \bigcup_{i=1}^m (B_i^- \cup B_i^+)$ into the closure of $B_{i_1}^{\mathrm{sign}(\sigma_1)}$ in $\di X$.
On the other hand, the set of elements $g\in G$ sending\linebreak $\di X \smallsetminus \bigcup_{i=1}^m (B_i^- \cup B_i^+)$ into the closure of $\bigcup_{i=1}^m (B_i^- \cup B_i^+)$ in $\di X$ is a closed subset of~$G$ that does not contain the identity element.
\end{proof}

Such a group $\Gamma$ is called a \emph{Schottky group}.
The proof of Claim~\ref{claim:ping-pong-rank-one} is based on the so-called \emph{ping pong dynamics} of $\Gamma$ on $\di X$: imagine the ping pong players are the generators $\gamma_1,\gamma_1^{-1},\dots,\gamma_m,\gamma_m^{-1}$; the ping pong table is $\di X$, which is divided into several open regions, namely the $B_i^{\pm}$ and the ``central region'' $\Int(\di X \smallsetminus \bigcup_i (B_i^- \cup B_i^+))$; the rules of the game are that each player $\gamma_i^{\pm 1}$ sends all regions but one (namely $B_i^{\mp}$) into a single region (namely $B_i^{\pm}$).
The ping pong ball is a point which is initially in the central region.
For any reduced word in the generators, we successively apply the corresponding ping pong players; the ball ends up in one of the $B_i^{\pm}$.
We deduce that the element of $\Gamma$ corresponding to this reduced word is nontrivial in~$\Gamma$, and not too close to the identity~in~$G$.

\begin{rem} \label{rem:DoD-rank-one}
Let $\mathcal{D} := \di X \smallsetminus \bigcup_{i=1}^m (B_i^- \cup B_i^+)$ and $\Omega :=\linebreak \mathrm{Int}(\bigcup_{\gamma\in\Gamma} \gamma\cdot\nolinebreak\mathcal{D})$.
Then $\Omega$ is an open subset of $\di X$ on which $\Gamma$ acts properly discontinuously with fundamental domain~$\mathcal{D}$.
\end{rem}

(Here we have assumed that $\mathcal{D}$ has nonempty interior; therefore $\Omega \neq \emptyset$ and $\Gamma$ is \emph{not} a lattice in~$G$: it has infinite covolume for the Haar measure.)
See \eg \cite{msw02} for beautiful illustrations in dimension two, for $X = \HH^3$.

Since Schottky groups $\Gamma$ are nonabelian free groups, they admit, as in Section~\ref{subsec:lattice-deform}, many nontrivial continuous deformations $(\rho_t)_{t\in [0,1)}\subset\Hom(\Gamma,G)$, obtained by independently deforming the image of each generator~$\gamma_i$.
Some of these deformations $(\rho_t)_{t\in [0,1)}$ are ``good'' in the sense that for every $t\in [0,1)$, the group $\rho_t(\Gamma)$ still has a ping pong configuration analogous to that of~$\Gamma$, hence $\rho_t$ is injective with discrete image by arguing as in Claim~\ref{claim:ping-pong-rank-one}.
If the open disks $B_1^{\pm},\dots,B_m^{\pm}$ in the initial configuration have pairwise disjoint \emph{closures} (\ie $\Gamma$ is a ``strong'' Schottky group), then all small deformations are ``good'': the natural inclusion $\rho_0 : \Gamma\hookrightarrow G$ admits an open neighbourhood in $\Hom(\Gamma,G)$ consisting entirely of injective and discrete representations, with a ping pong configuration analogous to that of~$\Gamma$.

\subsubsection{Quasi-Fuchsian groups}

Quasi-Fuchsian groups are important infinite discrete subgroups of\linebreak $\PSL(2,\C)$ which have been much studied (see \cite{mas88}), and which are \emph{not} lattices in $\PSL(2,\C)$.
They are by definition the images of quasi-Fuchsian representations.
Let us briefly recall what these are.

Let $S$ be a closed orientable surface of genus $g\geq 2$.
By the Uniformisation Theorem (see Section~\ref{sec:intro}), there exist injective and discrete representations from the fundamental group $\pi_1(S)$ to $\PSL(2,\R)$.
These representations form two connected components of $\Hom(\pi_1(S),\PSL(2,\R))$ \cite{gol-PhD}, switched by conjugation by elements of $\PGL(2,\R)\smallsetminus\PSL(2,\R)$ (\ie by orientation-reversing isometries of~$\HH^2$).
The image of either of these connected components in $\Hom(\pi_1(S),\PSL(2,\R))/\PSL(2,\R)$ identifies with the \emph{Teichm\"uller space} of~$S$, which is homeomorphic to $\R^{6g-6}$.

Now view $\PSL(2,\R)$ as a subgroup of $\PSL(2,\C)$.
Recall that $\PSL(2,\C)\linebreak\simeq\PO(3,1)_0$ acts by isometries on the hyperbolic space~$\HH^3$; the subgroup $\PSL(2,\R)\simeq\PO(2,1)_0$ preserves an isometric copy of $\HH^2$ inside~$\HH^3$.
We see the injective and discrete representations $\rho : \pi_1(S)\to\PSL(2,\R)$ as representations with values in $\PSL(2,\C)$, called \emph{Fuchsian}.
They preserve a circle in $\di\HH^3$, namely the boundary $\di\HH^2$ of the isometric copy of~$\HH^2$ preserved by $\PSL(2,\R)$.

The Fuchsian representations admit an open neighbourhood in\linebreak $\Hom(\pi_1(S),\PSL(2,\C))$ consisting entirely of injective and discrete representations, called \emph{quasi-Fuchsian}.
Each quasi-Fuchsian representation preserves a topological circle in $\di\HH^3$, but which may now be ``wiggly'' as in Figure~\ref{fig:QFLimSet}.
Quasi-Fuchsian representations form an open subset of $\Hom(\pi_1(S),\PSL(2,\C))$ which is dense in the set of injective and discrete representations; its image in $\Hom(\pi_1(S),\PSL(2,\C))/\PSL(2,\C)$ admits a natural parametrisation (due to Bers) by two copies of the Teichm\"uller space of~$S$ (hence by $\R^{12g-12}$).
See \eg \cite{ser05} for details and references.

\begin{figure}[ht!]
\includegraphics[scale=0.4]{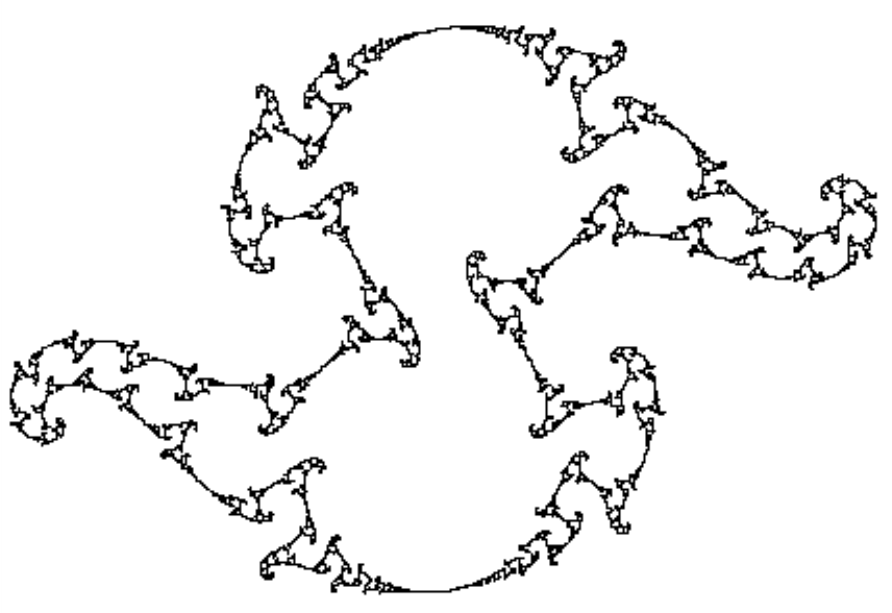}
\caption{The limit set (an invariant topological circle) of a quasi-Fuchsian group in $\di\HH^3 \simeq \C\cup\nolinebreak\{\infty\}$}
\label{fig:QFLimSet}
\end{figure}

\subsubsection{Deformations of Fuchsian representations for higher-dimensional groups} 

Recall that $\PSL(2,\R) \simeq \SO(2,1)_0$ and $\PSL(2,\C) \simeq \SO(3,1)_0$.
We now consider any integer $n\geq 2$ and let $\Gamma$ be a cocompact lattice of $\SO(n,1)_0$.
As above, we can see $\Gamma$ as a discrete subgroup of $\SO(n+1,1)$ (which is not a lattice anymore).
Interestingly, although all continuous deformations of $\Gamma$ in $\SO(n,1)$ are trivial for $n\geq 3$ (by Mostow rigidity, see Section~\ref{subsec:lattice-deform}), there can exist nontrivial continuous deformations of $\Gamma$ in $\SO(n+1,1)$.
Such deformations were constructed in \cite{jm87,kou85} based on a construction of Thurston called \emph{bending}.

The idea is the following.
The cocompact lattice $\Gamma$ of $\SO(n,1)_0$ defines a closed hyperbolic manifold $M = \Gamma\backslash\HH^n$ whose fundamental group $\pi_1(M)$ identifies with $\Gamma$.
Suppose that $M$ admits a closed totally geodesic embedded hypersurface~$N$.
Its fundamental group $\pi_1(N)$ is a subgroup of $\Gamma$ contained in a copy of $\SO(n-1,1)$ inside $\SO(n,1)$.
In particular, the centraliser of $\pi_1(N)$ in $\SO(n+1,1)$ contains a one-parameter subgroup $(g_t)_{t\in\R}$ which is not contained in $\SO(n,1)$.

If $N$ separates $M$ into two submanifolds $M_1$ and~$M_2$, then by van Kam\-pen's theorem $\pi_1(M)$ is the amalgamated free product\linebreak $\pi_1(M_1) *_{\pi_1(N)}\nolinebreak\pi_1(M_2)$ of $\pi_1(M_1)$ and $\pi_1(M_2)$ over $\pi_1(N)$.
Let $\rho_0 : \Gamma\to\SO(n+1,1)$ be the natural inclusion.
A continuous deformation $(\rho_t)_{t\in [0,1)} \subset \Hom(\Gamma,\SO(n+1,1))$ is obtained by defining $\rho_t$ to be $\rho_0$ when restricted to $\pi_1(M_1)$ and $g_t\rho_0(\cdot)g_t^{-1}$ when restricted to $\pi_1(M_2)$ (these two representations coincide on $\pi_1(N)$).

Otherwise, $M' := M\smallsetminus N$ is connected and $\pi_1(M)$ is an HNN extension of $\pi_1(M')$: it is generated by $\pi_1(M')$ and some element $\nu$ with the relations $\nu\,j_1(\gamma)\,\nu^{-1} = j_2(\gamma)$ for all $\gamma\in\pi_1(N)$, where $j_1 : \pi_1(N)\to\pi_1(M)$ and $j_2 : \pi_1(N)\to\pi_1(M)$ are the inclusions in~$\pi_1(M)$ of the fundamental groups of the two sides of~$N$.
Let $\rho_0 : \Gamma\to\SO(n+1,1)$ be the natural inclusion.
A continuous deformation $(\rho_t)_{t\in [0,1)} \subset \Hom(\Gamma,\SO(n+1,1))$ is obtained by defining $\rho_t$ to be $\rho_0$ when restricted to $\pi_1(M')$ and setting $\rho_t(\nu) := \nu g_t$ (the relations $\nu\,j_1(\gamma)\,\nu^{-1} = j_2(\gamma)$ for $\gamma\in\pi_1(N)$ are preserved since $g_t$ centralises $\pi_1(N)$).

In either case, Johnson and Millson \cite{jm87} observed that for small enough $t>0$ the representation $\rho_t$ has Zariski-dense image in $\SO(n+1,1)$; moreover, $\rho_t$ is still injective and discrete for small~$t$ (see Section~\ref{subsec:cc-open-rank-1}).

\begin{rems} \label{rem:bending}
\begin{enumerate}[(1)]
  \item\label{item:bending-1} In this construction, $\rho_t$ is not injective and discrete for all $t\in\R$.
Indeed, the one-parameter subgroup $(g_t)_{t\in\R}$ takes values in a copy of $\SO(2)$ in $\SO(n+1,1)$, which centralises $\rho_0(\pi_1(N))$.
For $t\in\R$ such that $g_t = -\mathrm{I}$ in $\SO(2)$, the representation $\rho_t$ takes values in $\SO(n,1)$ but is not injective and discrete.
  \item\label{item:bending-2} The fact that $\rho_t$ is injective and discrete for small~$t$ also follows from Maskit's combination theorems, which generalise the idea of ping pong to amalgamated free products and HNN extensions (see \cite[\S\,VIII.E.3]{mas88}).
\end{enumerate}
\end{rems}

\subsection{Ping pong in higher real rank} \label{subsec:ping-pong-higher-rank}

Examples of ``flexible'' discrete subgroups of higher-rank semisimple Lie groups~$G$ which are nonabelian free groups can be constructed by generalising the classical Schottky groups of Section~\ref{subsec:ex-rank-1} in various ways.
Let us mention three geometric constructions.

\subsubsection{Ping pong in projective space}

The idea of the following construction goes back to Tits \cite{tit72} in his proof of the Tits alternative.
The construction was later studied in a more quantitative way by Benoist \cite{ben97}.
It works in any flag variety $G/P$ where $G$ is a noncompact semisimple Lie group and $P$ a proper parabolic subgroup of~$G$, but for simplicity we consider the projective space $\PP(\R^d)$ which is a flag variety of $G = \SL(d,\R)$, for $d\geq 3$.
We fix a Riemannian metric $\mathtt{d}_{\PP(\R^d)}$ on $\PP(\R^d)$.

An element $g\in G$ is said to be \emph{biproximal} in $\PP(\R^d)$ if it admits a unique complex eigenvalue of highest modulus and a unique complex eigenvalue of lowest modulus, and if these two eigenvalues (which are then necessarily real) have multiplicity~$1$; equivalently, $g$ is conjugate to a block-diagonal matrix $\mathrm{diag}(t,A,s^{-1})$ where $t,s>1$ and $A\in\GL(d-2,\R)$ is such that the spectral radii of $A$ and~$A^{-1}$ are $<t$ and $<s$, respectively (for instance, $A$ could be the identity matrix).
In this case, $g$ has a unique attracting fixed point $x_g^+$ and a unique repelling fixed point $x_g^-$ in $\PP(\R^d)$, corresponding to the eigenspaces for the highest and lowest eigenvalues.
More precisely, $g$ has the following ``North-South dynamics'' on $\PP(\R^d)$:
\begin{itemize}
  \item it preserves a unique projective hyperplane $X_g^+$ (\resp $X_g^-$) of $\PP(\R^d)$ containing $x_g^+$ (\resp $x_g^-$), corresponding to the sum of the generalised eigenspaces for the eigenvalues of nonminimal (\resp nonmaximal) modulus,
  \item for any $x \in \PP(\R^d)\smallsetminus X_g^{\mp}$ we have $g^{\pm k}\cdot x\to x_g^{\pm}$ as $k\to +\infty$, uniformly on compact sets.
\end{itemize}
In particular, if we fix $\varepsilon>0$, then any large power of~$g$ sends the complement of the open uniform $\varepsilon$-neighbourhood $B_{g^{-1}}^{\varepsilon}$  of $X_g^-$ into the closure of the open ball $b_g^{\varepsilon}$ of radius $\varepsilon$ centred at~$x_g^+$ for $\mathtt{d}_{\PP(\R^d)}$, and $g^{-1}$ has a similar behaviour after replacing $(B_{g^{-1}}^{\varepsilon},X_g^-,b_g^{\varepsilon},x_g^+)$ by $(B_g^{\varepsilon},X_g^+,b_{g^{-1}}^{\varepsilon},x_g^-)$ (see Figure~\ref{fig:ping-pong-proj}).
\begin{figure}
\vspace{0.5cm}
\begin{overpic}[scale=0.4,percent]{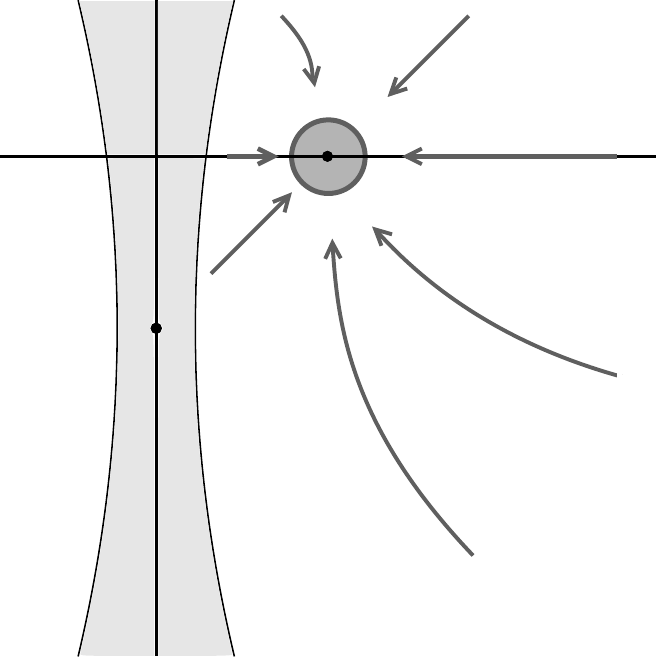}
\put(51,85){$b_g^{\varepsilon}$}
\put(-12,75){$X_g^+$}
\put(13,47){$x_g^-$}
\put(14,104){$B_{g^{-1}}^{\varepsilon}$}
\end{overpic}
\hspace{0.7cm}
\begin{overpic}[scale=0.4,percent]{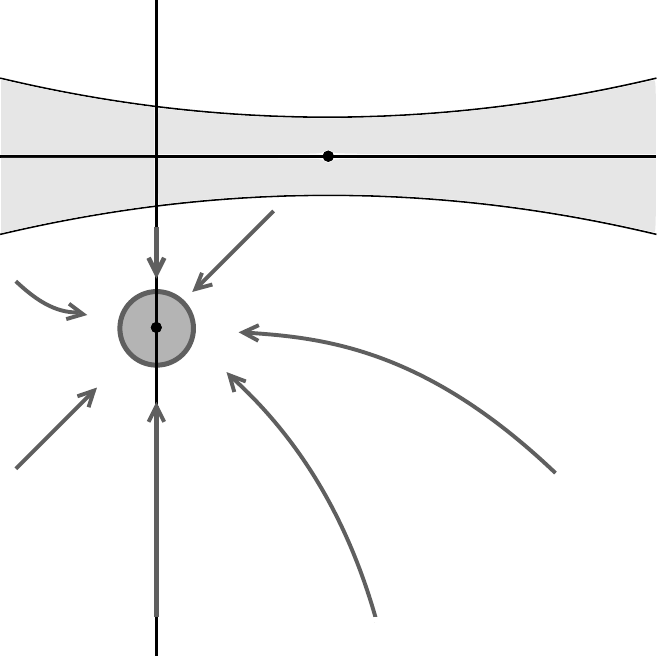}
\put(52,80){$x_g^+$}
\put(102,80){$B_g^{\varepsilon}$}
\put(5,44.5){$b_{g^{-1}}^{\varepsilon}$}
\put(20,103){$X_g^-$}
\end{overpic}
\caption{Left (\resp right) panel: the dynamics of a large power of $g$ (\resp $g^{-1}$) on $\PP(\R^d)$ for a biproximal element $g\in\SL(d,\R)$}
\label{fig:ping-pong-proj}
\end{figure}
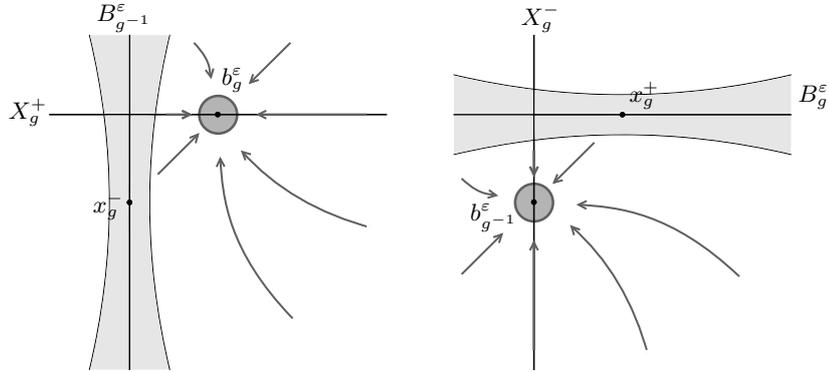

Let $\gamma_1,\dots,\gamma_m\in G$ be biproximal elements which are ``transverse'' in the sense that $x_{\gamma_i}^+, x_{\gamma_i}^- \notin X_{\gamma_j}^+\cup X_{\gamma_j}^-$ for all $1\leq i\neq j\leq m$ (in other words, the configuration of pairs $(x_{\gamma_i}^{\bullet},X_{\gamma_i}^{\bullet})_{1\leq i\leq m,\ \bullet\in\{+,-\}}$ with $x_{\gamma_i}^{\bullet}\in X_{\gamma_i}^{\bullet}$ is generic).
Up to replacing each $\gamma_i$ by a large power, we may assume that there exists $\varepsilon>0$ such that $\PP(\R^d) \smallsetminus \bigcup_{i=1}^m (B_{\gamma_i}^{\varepsilon} \cup B_{\gamma_i^{-1}}^{\varepsilon})$ has nonempty interior and such that for any $\alpha\neq\beta$ in $\{\gamma_1,\gamma_1^{-1},\dots,\gamma_m,\gamma_m^{-1}\}$, the sets $b_{\alpha}^{\varepsilon}$ and $B_{\beta}^{\varepsilon}$ have disjoint closures in $\PP(\R^d)$ and $\alpha$ sends the interior of $\PP(\R^d)\smallsetminus B_{\alpha^{-1}}^{\varepsilon}$ into $b_{\alpha}^{\varepsilon}$ (see Figure~\ref{fig:ping-pong-config-proj} for $m=2$).
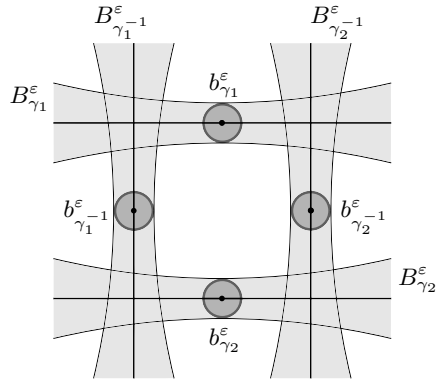
\begin{figure}[ht!]
\vspace{0.4cm}
\begin{overpic}[scale=0.4,percent]{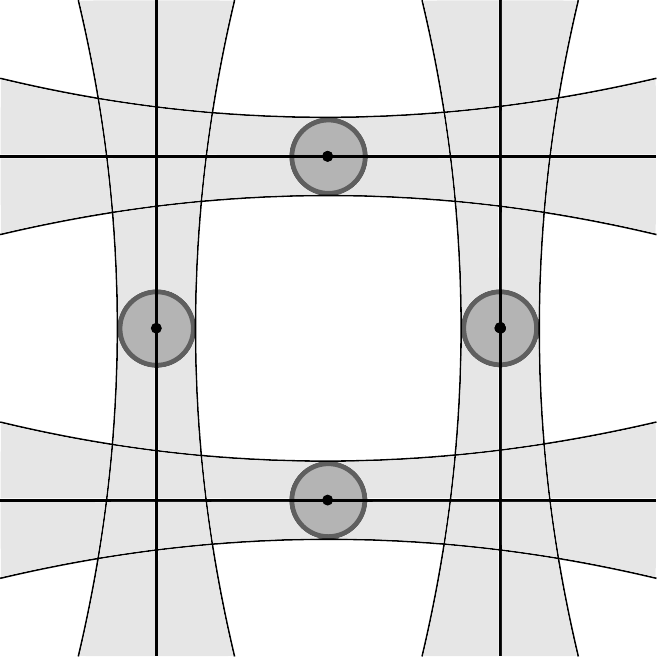}
\put(46,86){$b_{\gamma_1}^{\varepsilon}$}
\put(12,106){$B_{\gamma_1^{-1}}^{\varepsilon}$}
\put(3,48){$b_{\gamma_1^{-1}}^{\varepsilon}$}
\put(-13,82){$B_{\gamma_1}^{\varepsilon}$}
\put(46,10){$b_{\gamma_2}^{\varepsilon}$}
\put(76,106){$B_{\gamma_2^{-1}}^{\varepsilon}$}
\put(85,48){$b_{\gamma_2^{-1}}^{\varepsilon}$}
\put(102,28){$B_{\gamma_2}^{\varepsilon}$}
\end{overpic}
\caption{A ping pong configuration as in Claim~\ref{claim:ping-pong-proj}}
\label{fig:ping-pong-config-proj}
\end{figure}
Let $\Gamma$ be the subgroup of~$G$ generated by $\gamma_1,\dots,\gamma_m$.
The following is analogous to Claim~\ref{claim:ping-pong-rank-one}.

\begin{claim} \label{claim:ping-pong-proj}
The group $\Gamma$ is a nonabelian free group with free generating subset $\{\gamma_1,\dots,\gamma_m\}$.
It is discrete in~$G$.
\end{claim}

\begin{proof}
Consider any reduced word $\gamma = \gamma_{i_1}^{\sigma_1} \dots \gamma_{i_N}^{\sigma_N}$ in the alphabet $\{ \gamma_1^{\pm 1},\dots,\gamma_m^{\pm 1}\}$, where $1\leq i_j\leq m$ and $\sigma_j\in\{\pm 1\}$ for all $1\leq j\leq\nolinebreak N$.
Using the inclusions $\alpha\cdot\mathrm{Int}\big(\PP(\R^d)\smallsetminus B_{\alpha^{-1}}^{\varepsilon}\big) \subset b_{\alpha}^{\varepsilon}$ for $\alpha = \gamma_{i_j}^{\sigma_j}$ and $b_{\alpha}^{\varepsilon} \subset \mathrm{Int}\big(\PP(\R^d)\smallsetminus B_{\beta^{-1}}^{\varepsilon}\big)$ for $(\alpha,\beta) = (\gamma_{i_j}^{\sigma_j},\gamma_{i_{j-1}}^{\sigma_{j-1}})$ with $j\geq 2$, we see that the element of~$\Gamma$ corresponding to $\gamma$ sends $\PP(\R^d) \smallsetminus \bigcup_{i=1}^m (B_{\gamma_i}^{\varepsilon} \cup B_{\gamma_i^{-1}}^{\varepsilon})$ into the closure of $b_{\gamma_{i_1}^{\sigma_1}}^{\varepsilon}$ (hence of $B_{\gamma_{i_1}^{\sigma_1}}^{\varepsilon}$) in $\PP(\R^d)$.
On the other hand, the set of elements $g\in G$ sending $\PP(\R^d) \smallsetminus \bigcup_{i=1}^m (B_{\gamma_i}^{\varepsilon} \cup B_{\gamma_i^{-1}}^{\varepsilon})$ into the closure of $\bigcup_{i=1}^m (B_{\gamma_i}^{\varepsilon} \cup B_{\gamma_i^{-1}}^{\varepsilon})$ in $\PP(\R^d)$ is a closed subset of~$G$ that does not contain the identity element.
\end{proof}

Similarly to the classical strong Schottky groups of Section~\ref{subsec:ex-rank-1}, the group $\Gamma$ admits nontrivial continuous deformations $(\rho_t)_{t\in [0,1)}\subset\Hom(\Gamma,G)$, obtained by independently deforming the image of each generator~$\gamma_i$; moreover, there is a neighbourhood of the natural inclusion $\rho_0 : \Gamma\hookrightarrow G$ consisting entirely of injective and discrete representations.

\subsubsection{Schottky groups with disjoint ping pong domains}

In certain situations it is possible to construct discrete subgroups of~$G$, with ping pong dynamics, for which the ping pong domains are pairwise disjoint, as in the case of the classical rank-one Schottky groups of Section~\ref{subsec:ex-rank-1}.
Achieving this disjointness may require using a slightly modified ping pong compared to Figures \ref{fig:ping-pong-proj} and~\ref{fig:ping-pong-config-proj}, allowing the attracting and repelling subsets of the generators to be larger than points.

Such a construction has been made in $G = \PGL(2n,\mathbb{K})$, acting on the projective space $\PP(\mathbb{K}^{2n})$, for $\mathbb{K}=\R$ or~$\C$: the first examples were constructed by Nori in the 1980s, for $\mathbb{K}=\C$, then generalised by Seade and Verjovsky (see \cite{sv03}); it was observed in \cite{hv08} that the construction also works for $\mathbb{K}=\R$.
The idea is to consider pairwise disjoint $(n-1)$-dimensional projective subspaces $X_1^+,X_1^-,\dots,X_m^+,X_m^-$ of $\PP(\mathbb{K}^{2n})$ and elements $\gamma_1,\dots,\gamma_m\in G$ such that for any $1\leq i\leq m$ we have $\gamma_i^{\pm k} \cdot x \to X_i^{\pm}$ for all $x \in \PP(\mathbb{K}^{2n})\smallsetminus X_i^{\mp}$ as $k\to +\infty$, uniformly on compact sets.
Up to replacing each $\gamma_i$ by a large power, we may assume that there exist tubular neighbourhoods $B_i^{\pm}$ of $X_i^{\pm}$ such that $B_1^+,B_1^-,\dots,B_m^+,B_m^-$ are pairwise disjoint, $\PP(\mathbb{K}^{2n}) \smallsetminus \bigcup_{i=1}^m (B_i^- \cup B_i^+)$ has nonempty interior, and $\gamma_i \cdot \mathrm{Int}(\PP(\mathbb{K}^{2n})\smallsetminus B_i^-) = B_i^+$ for all~$i$.
Then the subgroup $\Gamma$ of~$G$ generated by $\gamma_1,\dots,\gamma_m$ is a nonabelian free group with free generating subset $\{\gamma_1,\dots,\gamma_m\}$.
It is discrete in~$G$, and it acts properly discontinuously on $\Omega := \mathrm{Int}(\bigcup_{\gamma\in\Gamma} \gamma\cdot\nolinebreak\mathcal{D})$ with fundamental domain~$\mathcal{D}$, as in Remark~\ref{rem:DoD-rank-one}.
As for the classical strong Schottky groups of Section~\ref{subsec:ex-rank-1}, there is a neighbourhood of the natural inclusion $\rho_0 : \Gamma\hookrightarrow G$ consisting entirely of injective and discrete representations.

\subsubsection{Crooked Schottky groups}

Here is another ping pong construction, introduced and studied in~\cite{bk-Schottky}.

Let $G = \Sp(2n,\R)$ be the group of elements of $\GL(2n,\R)$ that preserve the skew-symmetric bilinear form $\omega(v,v') = \sum_{i=1}^{2n} (-1)^i v_i v'_{2n+1-i}$ on~$\R^{2n}$.
A \emph{symplectic basis} of~$\R^{2n}$ is a basis in which the matrix of $\omega$ is antidiagonal with entries $1,-1,\dots,1,-1$; for instance, the canonical basis is a symplectic basis.
To any symplectic basis $(e_1,\dots,e_{2n})$ of~$\R^{2n}$ we as\-sociate an open simplex $B = \PP(\R^{>0}\text{-span}(e_1,\dots,e_{2n}))$ in $\PP(\R^{2n})$, which we call a \emph{symplectic simplex}.
Its \emph{dual} $B^* := \{[v]\in\nolinebreak\PP(\R^{2n}) \,|\, \PP(v^{\perp})\cap\nolinebreak\overline{B}=\nolinebreak\emptyset\}$ (where $v^{\perp}$ denotes the orthogonal of $v$ with respect to~$\omega$ and $\overline{B}$ the closure of $B$ in $\PP(\R^{2n})$) is still a symplectic simplex, associated to the symplectic basis $(e_{2n},-e_{2n-1},\dots,e_2,-e_1)$.
Note that $B$ and~$B^*$ are two of the $2^{2n-1}$ connected components of $\PP(\R^{2n})\smallsetminus\bigcup_{i=1}^{2n} \PP(e_i^{\perp})$ (see Figure~\ref{fig:tetrahedra}).
We also make the elementary observation that for any symplectic simplices $B_1$ and~$B_2$, we have $B_1 \subset B_2^*$ if and only if $B_2 \subset B_1^*$.

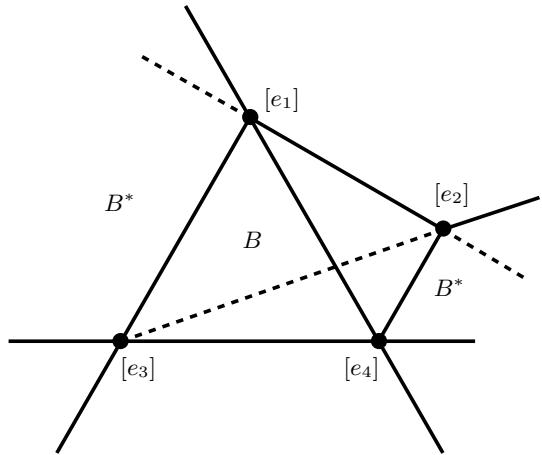
\begin{figure}[ht!]
\begin{overpic}[scale=0.4,percent]{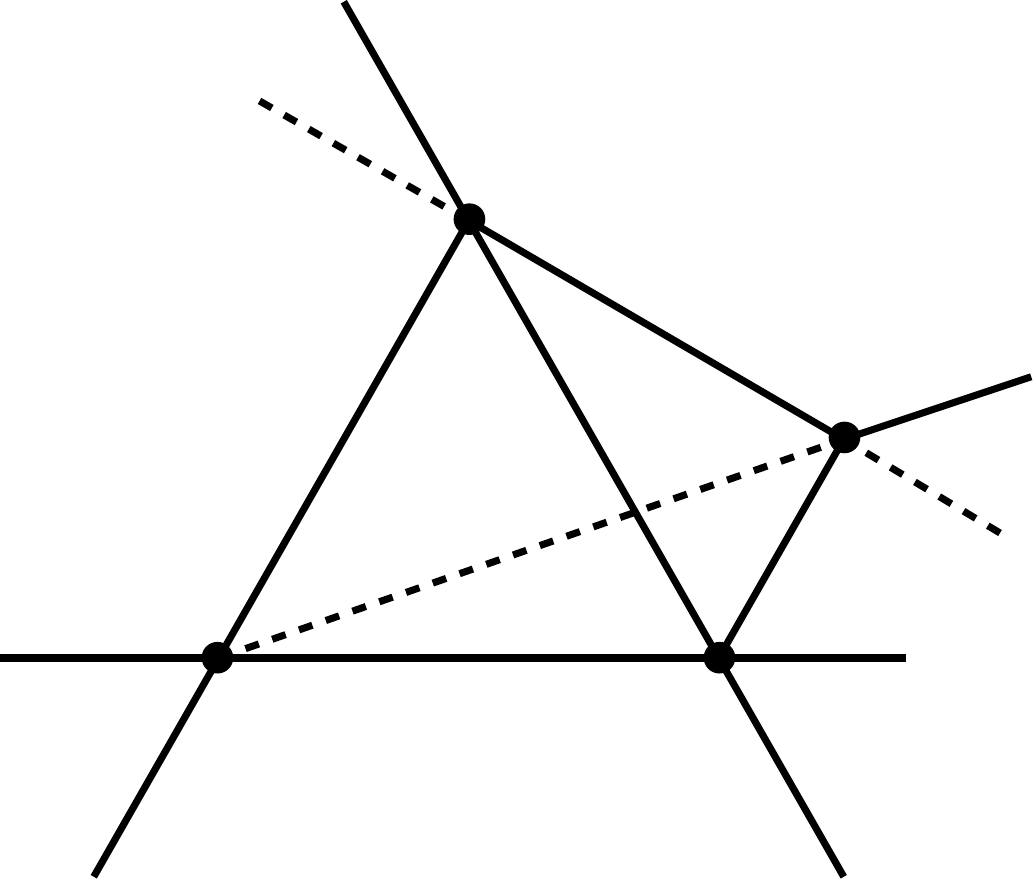}
\put(48,66){$[e_1]$}
\put(80,48){$[e_2]$}
\put(21,15){$[e_3]$}
\put(63,15){$[e_4]$}
\put(44,39){$B$}
\put(18,46){$B^*$}
\put(80,30){$B^*$}
\end{overpic}
\caption{A symplectic simplex $B$ of $\PP(\R^4)$, associated to a symplectic basis $(e_1,e_2,e_3,e_4)$ of~$\R^4$, and its dual $B^*$}
\label{fig:tetrahedra}
\end{figure}

\begin{lem} \label{lem:construct-simplices}
For any $m\geq 2$, there exist $2m$ symplectic simplices $B_1^{\pm}$, $\dots, B_m^{\pm}$ in $\PP(\R^{2n})$ such that $B\subset {B'}^*$ for all $B\neq B'$ in $\{B_1^{\pm}, \dots, B_m^{\pm}\}$.
\end{lem}

\begin{proof}
Choose any symplectic simplex~$B_1$.
Note that any nonempty open subset of $\PP(\R^{2n})$ contains a symplectic simplex; therefore, we can find a symplectic simplex $B_2$ such that $\overline{B_2} \subset B_1^*$.
Moreover, any neighbourhood of the closure of a symplectic simplex meets the dual of the simplex; therefore $B_1^* \cap B_2^*$ is nonempty.
By induction, we construct symplectic simplices $B_1,\dots,B_{2m}$ such that $\overline{B_j} \subset \bigcap_{i=1}^{j-1} B_i^*$ for all $2\leq j\leq 2m$.
We then have $B_j \subset B_i^*$ for all $1\leq i<j\leq 2m$.
By the elementary observation above, we also have $B_i \subset B_j^*$ for all $1\leq i<j\leq 2m$.
We can then take $(B_i^+,B_i^-) := (B_i,B_{m+i})$ for all $1\leq i\leq m$.
\end{proof}

For $m\geq 2$, choose symplectic simplices $B_1^{\pm}, \dots, B_m^{\pm}$ as in Lemma~\ref{lem:construct-simplices}, and elements $\gamma_1,\dots,\gamma_m\in G$ such that $\gamma_i\cdot (B_i^-)^* = B_i^+$ for all $1\leq i\leq m$ (these exist since $G$ acts transitively on the set of symplectic simplices).
A ping pong argument as in Claim~\ref{claim:ping-pong-rank-one} shows that the subgroup $\Gamma$ of~$G$ generated by $\gamma_1,\dots,\gamma_m$ is a nonabelian free group with free generating subset $\{\gamma_1,\dots,\gamma_m\}$, and that it is discrete in~$G$.

Moreover, there is an interesting counterpart of Remark~\ref{rem:DoD-rank-one}, not in the projective space $\PP(\R^{2n})$, but in the space $\mathrm{Lag}(\R^{2n})$ of \emph{Lagrangians} of $(\R^{2n},\omega)$, \ie of $n$-dimensional linear subspaces of~$\R^{2n}$ which are totally isotropic for~$\omega$.
For this, we associate to any symplectic simplex $B = \PP(\R^{>0}\text{-span}(e_1,\dots,e_{2n}))$ of $\PP(\R^{2n})$ an open subset of $\mathrm{Lag}(\R^{2n})$, namely
$\mathcal{H}(B) := \{ L\in\mathrm{Lag}(\R^{2n}) ~|~ L\cap B \neq \emptyset\}$
where we see each~$L\in\mathrm{Lag}(\R^{2n})$ as an $(n-1)$-dimensional projective subspace of $\PP(\R^{2n})$.
In \cite{bk-Schottky} we prove the following remarkable property: for any symplectic simplex~$B$, we have
$\mathrm{Lag}(\R^{2n}) = \overline{\mathcal{H}(B)} \sqcup \mathcal{H}(B^*).$
In other words, $\mathcal{H}(B)$ and $\mathcal{H}(B^*)$ are two open ``half-spaces'' of $\mathrm{Lag}(\R^{2n})$, bounded by their common boundary $\partial\mathcal{H}(B) = \overline{\mathcal{H}(B)}\smallsetminus\mathcal{H}(B) = \overline{\mathcal{H}(B^*)}\smallsetminus\mathcal{H}(B^*)$.
We observe \cite{bk-Schottky} that these boundaries $\partial\mathcal{H}(B)$ are nice geometric objects which for $n=2$ coincide with the \emph{crooked surfaces} of Frances \cite{fra03} in the Einstein universe $\mathrm{Ein}^3 \simeq \mathrm{Lag}(\R^4)$.
Remark~\ref{rem:DoD-rank-one} generalises as follows: consider symplectic simplices $B_1^{\pm}, \dots, B_m^{\pm}$ and elements $\gamma_1,\dots,\gamma_m\in G$ as above.
If we set $\mathcal{D} := \mathrm{Lag}(\R^{2n}) \smallsetminus \bigcup_{i=1}^m (\mathcal{H}(B_i^-) \cup \mathcal{H}(B_i^+))$, then the group $\Gamma$ generated by $\gamma_1,\dots,\gamma_m$ acts properly discontinuously on $\Omega := \mathrm{Int}(\bigcup_{\gamma\in\Gamma} \gamma\cdot\nolinebreak\mathcal{D})$ with fundamental domain~$\mathcal{D}$.
We call $\Gamma$ a \emph{crooked Schottky group}.

As in the classical case of Section~\ref{subsec:ex-rank-1}, continuous deformations of configurations of symplectic simplices yield nontrivial continuous deformations $(\rho_t)_{t\in [0,1)}\subset\Hom(\Gamma,G)$ of crooked Schottky groups $\Gamma$ for which each $\rho_t$ is injective with discrete image.
If the symplectic simplices $B_1^{\pm},\dots,B_m^{\pm}$ in the initial configuration have pairwise disjoint \emph{closures} (\ie $\Gamma$ is a ``strong'' crooked Schottky group), then the natural inclusion $\rho_0 : \Gamma\hookrightarrow G$ admits an open neighbourhood in $\Hom(\Gamma,G)$ consisting entirely of injective and discrete representations.

\subsection{Higher-rank deformations of Fuchsian representations} \label{subsec:deform-Fuchsian-higher-rank}

Inspired by the quasi-Fuchsian representations and their higher-dimen\-sional analogues from Section~\ref{subsec:ex-rank-1}, here is one strategy for constructing ``flexible'' infinite discrete subgroups, beyond nonabelian free groups, in semisimple Lie groups~$G$ of higher real rank.
Consider a finitely generated group~$\Gamma_0$, an injective and discrete representation $\sigma_0$ of $\Gamma_0$ into a simple Lie group~$G'$ of real rank one, and a nontrivial Lie group homomorphism $\tau : G'\to G$.
Consider the composed representation
$$\rho_0 : \Gamma_0 \overset{\sigma_0}{\longhookrightarrow} G' \overset{\tau}{\longhookrightarrow} G.$$
In some important cases (see \eg Facts \ref{fact:Ano-open} and~\ref{fact:cc-embed-Ano}), there will be an open neighbourhood of $\rho_0$ in $\Hom(\Gamma_0,G)$ consisting entirely of injective and discrete representations.
The goal is then to deform $\rho_0$ nontrivially in $\Hom(\Gamma_0,G)$ outside of $\Hom(\Gamma_0,G')$, so as to obtain discrete subgroups of $G$ that are isomorphic to~$\Gamma_0$ but not conjugate to~$\Gamma_0$ or any subgroup of~$G'$ (ideally Zariski-dense discrete subgroups of~$G$).

This strategy works well, for instance, for $\Gamma_0 = \pi_1(S)$ where $S$ is a closed orientable surface of genus $\geq 2$ as in Section~\ref{subsec:ex-rank-1}, and $G' = \SL(2,\R)$ or $\PSL(2,\R)$.
Let us give three examples in this setting.

\subsubsection{Barbot representations}

For $d\geq 2$, consider the standard embedding $\tau : G'=\SL(2,\R)\hookrightarrow G=\SL(d,\R)$, acting trivially on a $(d-2)$-dimensional linear subspace of~$\R^d$.
Then there is a neighbourhood of $\rho_0$ in $\Hom(\Gamma_0,G)$ consisting entirely of injective and discrete representations (see Section~\ref{subsec:Anosov}).
Nontrivial continuous deformations $(\rho_t)_{t\in [0,1)} \subset \Hom(\Gamma_0,G)$ exist; for $d=3$, they were studied by Barbot, who particularly investigated \cite{bar01} the case that the $\rho_t$ take values in $\GL(2,\R)\ltimes\R^2$, seen as the subgroup of~$G$ consisting of lower block-triangular matrices with blocks of size $(2,1)$.

\subsubsection{Hitchin representations}

For $d\geq 2$, consider the irreducible embedding $\tau_d : G' = \PSL(2,\R)\hookrightarrow G = \PSL(d,\R)$.
It is unique modulo conjugation by $\PGL(d,\R)$, and given concretely as follows: identify $\R^d$ with the vector space $\R[X,Y]_{d-1}$ of real polynomials in two variables $X,Y$ which are homogeneous of degree $d-1$.
The group $\SL(2,\R)$ acts on $\R[X,Y]_{d-1}$ by
$$\begin{pmatrix} \mathtt{a} & \mathtt{b}\\ \mathtt{c} & \mathtt{d}\end{pmatrix} \cdot P\!\begin{pmatrix}X\\Y\end{pmatrix} = P\!\left(\!\begin{pmatrix}\mathtt{a} & \mathtt{b}\\ \mathtt{c} & \mathtt{d}\end{pmatrix}^{-1} \! \begin{pmatrix}X\\Y\end{pmatrix}\!\right),$$
and this defines an irreducible representation $\SL(2,\R) \to \SL(\R[X,Y]_{d-1})\linebreak\simeq\SL(d,\R)$, which is injective if $d$ is even, and has kernel $\{\pm\mathrm{I}\}$ if $d$ is odd.
It factors into an embedding $\tau_d : \PSL(2,\R)\hookrightarrow\PSL(d,\R)$.
In this setting, the following result was proved by Choi--Goldman \cite{cg93} for $d=3$, and by Labourie \cite{lab06} and Fock--Goncharov \cite{fg06} for general~$d$ (recall that the case $d=2$ is due to Goldman \cite{gol-PhD}).

\begin{thm} \label{thm:Hitchin-comp}
Let $\Gamma_0 = \pi_1(S)$ be a closed surface group and $\sigma_0 : \Gamma_0\to\PSL(2,\R)$ an injective and discrete representation.
For any $d\geq 2$, the connected component of $\rho_0 := \tau_d\circ\sigma_0$ in $\Hom(\Gamma_0,\PSL(d,\R))$ consists entirely of injective and discrete representations.
\end{thm}

The image of this connected component in the $\PSL(d,\R)$-character variety of $\Gamma_0$ had previously been studied by Hitchin \cite{hit92}, and is now known as the \emph{Hitchin component}.
The corresponding representations are called \emph{Hitchin representations}.

\begin{proof}[Rough sketch of the proofs of Theorem~\ref{thm:Hitchin-comp}]
The proof of Choi--Gold\-man \cite{cg93} for $d=3$ is geometric.
The point is that the group $\tau_3(\PSL(2,\R))\linebreak \simeq \SO(2,1)_0$ preserves a nondegenerate symmetric bilinear form $\langle\cdot,\cdot\rangle_{2,1}$ of signature $(2,1)$ on~$\R^3$; in particular, it preserves the open subset
$$\Omega = \{ [v]\in\PP(\R^3) ~|~ \langle v,v\rangle_{2,1}<0\}$$
of the projective plane $\PP(\R^3)$, which is a model for the hyperbolic plane $\HH^2$ (see \eqref{eqn:proj-model-Hn}).
This set $\Omega$ is \emph{properly convex}: it is convex and bounded in some affine chart of $\PP(\R^3)$ (\eg it is the open unit disk in the affine chart $\{ v_3=1\}$, see Figure~\ref{fig:ConvDiv}, left).
The group $\PSL(2,\R)$ acts properly and transitively on~$\Omega$ via~$\tau_3$, hence $\Gamma_0$ acts properly discontinuously with compact quotient on~$\Omega$ via $\rho_0 = \tau_3\circ\sigma_0$.
By work of Koszul, the set of representations through which $\Gamma_0$ acts properly discontinuously with compact quotient on some nonempty properly convex open subset of $\PP(\R^3)$ is open in $\Hom(\Gamma_0,\PSL(3,\R))$.
Choi and Goldman proved that this set is also closed.
Therefore the entire connected component of~$\rho_0$ consists of such representations, and they are injective and discrete.

\begin{figure}[ht!]
\includegraphics[scale=0.4]{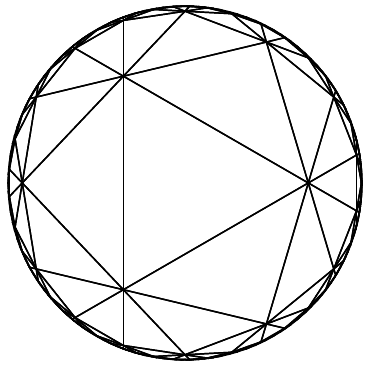}
\includegraphics[scale=0.4]{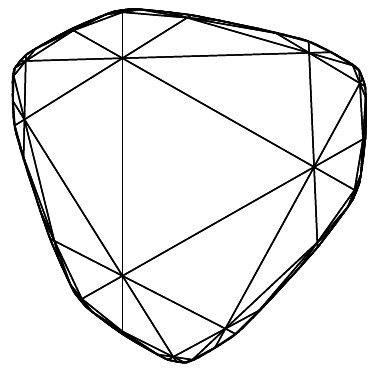}
\includegraphics[scale=0.4]{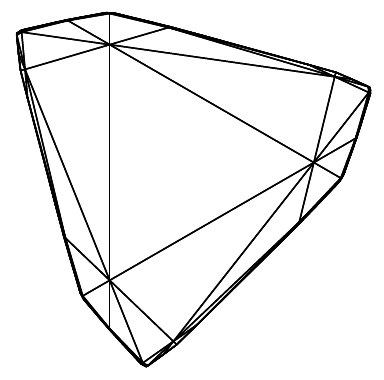}
\includegraphics[scale=0.4]{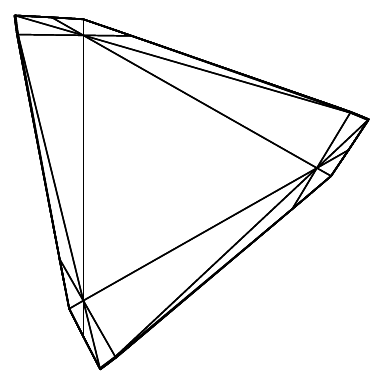}
\caption{The left-most picture shows the projective model of~$\HH^2$ (an open disk in an affine chart of the projective plane $\PP(\R^3)$), tiled by fundamental domains for the action of a triangle group $T$ given by some injective and discrete representation $\rho_0 : T\to\PO(2,1)\subset\PGL(3,\R)$. (Note that $T$ admits a finite-index subgroup which is a closed surface group $\Gamma_0 = \pi_1(S)$ as in the proof of Theorem~\ref{thm:Hitchin-comp}.) The other pictures show the effect of a continuous deformation of $\rho_0$ in $\Hom(T,\PGL(3,\R))$: the open disk deforms into an invariant properly convex open subset of $\PP(\R^3)$ which is not a disk anymore, and the tiling deforms as the action of~$T$ remains properly discontinuous and cocompact. These pictures are taken from \cite{ben04}.}
\label{fig:ConvDiv}
\end{figure}

The proofs of Labourie and Fock--Goncharov for general~$d$ are dynamical.
They involve two key objects.
The first one is the \emph{Gromov boundary} $\di\Gamma_0$ of~$\Gamma_0$: by definition, this is the visual boundary of a proper geodesic metric space on which $\Gamma_0$ acts properly discontinuously, by isometries, with compact quotient; in our situation, $\Gamma_0$ is a closed surface group and $\di\Gamma_0$ is the visual boundary of~$\HH^2$, namely a circle.
The second key object is the space $\mathrm{Flags}(\R^d)$ of full flags $(V_1\subset\dots\subset V_{d-1}\subset\R^d)$ of~$\R^d$ (where each $V_i$ is an $i$-dimensional linear subspace of~$\R^d$); this space $\mathrm{Flags}(\R^d)$ is compact with a transitive action of $G = \PSL(d,\R)$, and may be thought of as a kind of ``boundary'' for $G$ or its symmetric space.
The point of the proof is then to show that for any Hitchin representation $\rho : \Gamma_0\to G$, there exists a continuous, injective, $\rho$-equivariant ``boundary map'' $\xi : \di\Gamma_0\to\mathrm{Flags}(\R^d)$.
See Figure~\ref{fig:Hitchin-bound-map}.
(By \emph{$\rho$-equivariant} we mean that $\xi(\gamma\cdot w) = \rho(\gamma)\cdot\xi(w)$ for all $\gamma\in\Gamma_0$ and all $w\in\di\Gamma_0$.)
The existence of such a boundary map $\xi$ easily implies that $\rho$ is injective and discrete: see the proof of Lemma~\ref{lem:basic-Ano}.\eqref{item:basic-Ano-3}.
Indeed, the idea is that the continuous, injective, equivariant map $\xi$ ``transfers'', to $\mathrm{Flags}(\R^d)$, the dynamics of the intrinsic action of $\Gamma_0$ on $\di\Gamma_0$, which is a so-called \emph{convergence action}: any sequence $(\gamma_k)_{k\in\N}$ of pairwise distinct elements of~$\Gamma_0$ comes with some contraction in $\di\Gamma_0$, hence (using~$\xi$) the sequence $(\rho(\gamma_k))_{k\in\N}$ comes with some contraction in $\mathrm{Flags}(\R^d)$, and this prevents $(\rho(\gamma_k))_{k\in\N}$ from converging to the identity element of~$G$.

We note that the existence of continuous, injective, equivariant boundary maps for Hitchin representations is obtained by an open-and-closed argument, using stronger properties satisfied by these maps (namely, some uniform forms of contraction and transversality for Labourie, and a positivity property for Fock and Goncharov).
\end{proof}

\subsubsection{Maximal representations}

For $n\geq 2$, consider the embedding $\tau : G' = \PSL(2,\R) \simeq \SO(2,1)_0 \hookrightarrow \SO(2,n)$.
Then the entire connected component of $\rho_0$ in $\Hom(\Gamma_0,G)$ consists of injective and discrete representations, as was proved by Burger, Iozzi, and Wienhard \cite{biw10}.
This is an example of a so-called \emph{maximal component}: it consists of representations (called \emph{maximal representations}) that maximise the \emph{Toledo invariant} (a topological invariant generalising the Euler number, see \eg \cite[\S\,5.1]{biw14}).

\subsection{Higher Teichm\"uller theory} \label{subsec:higher-Teich}

We already encountered in Sections \ref{subsec:lattice-deform} and~\ref{subsec:ex-rank-1} the Teichm\"uller space of a closed surface $S$ of genus $\geq 2$.
It is a fundamental object in many areas of mathematics, which can be viewed both as a moduli space for marked complex structures on~$S$ or, via the Uniformisation Theorem, as a moduli space for marked hyperbolic structures on~$S$.
In this second point of view, the holonomy representation of the fundamental group $\Gamma_0 = \pi_1(S)$ naturally realises the Teichm\"uller space of~$S$ as a connected component of the $G$-character variety of $\Gamma_0$ for $G=\PSL(2,\R)$, corresponding to the image, modulo conjugation by~$G$ at the target, of a connected component of $\Hom(\Gamma_0,G)$ consisting entirely of injective and discrete representations.

An interesting and perhaps surprising phenomenon, which has led to a considerable amount of research in the past twenty years, is that for certain semisimple Lie groups $G$ of higher real rank, there also exist connected components of $\Hom(\Gamma_0,G)$ consisting entirely of injective and discrete representations, and which are nontrivial in the sense that they are not reduced to a single representation and its conjugates by~$G$.
The images in the $G$-character variety of these components are now called \emph{higher(-rank) Teichm\"uller spaces}.
We saw two examples in Section~\ref{subsec:deform-Fuchsian-higher-rank}:
\begin{itemize}
  \item \emph{Hitchin components} when $G$ is $\PSL(d,\R)$, or more generally a real split simple Lie group; these are by definition components containing a Fuchsian representation $\rho_0 : \Gamma_0\hookrightarrow\PSL(2,\R)\hookrightarrow G$, where $\PSL(2,\R)\hookrightarrow G$ is the so-called \emph{principal embedding};
  \item \emph{maximal components} when $G$ is $\SO(2,n)$, or more generally a simple Lie group of Hermitian type; these are by definition components of representations that maximise the Toledo invariant.
\end{itemize}
See \cite{biw14,can22,oberwolfach22,poz-bourbaki,wie-icm} for details about these examples.

\subsubsection{Towards a full list of higher Teichm\"uller spaces}

Recently, new higher Teichm\"uller spaces were discovered in \cite{bp,bcggo,glw} when $G$ is $\OO(p,q)$ with $p\neq q$ or an exceptional simple Lie group whose restricted root system is of type $F_4$.
These higher Teichm\"uller spaces consist of so-called \emph{$\Theta$-positive} representations, introduced by Guichard and Wienhard \cite{gw16,gw22}.
Notions of positivity for Hitchin representations and maximal representations had been previously found by Fock--Goncharov \cite{fg06} (based on Lusztig's total positivity \cite{lus94}) and Burger--Iozzi--Wienhard \cite{biw10}; the notion of $\Theta$-positivity encompasses them both.
Together with Hitchin components and maximal components, these new $\Theta$-positive components conjecturally (see \cite{gw16}) form the full list of higher Teichm\"uller spaces.

Without entering into technical details, let us mention briefly the role of Higgs bundles in this conjectural classification.
See \cite{bra24,bgg06,g20} for details.

Let $\Sigma$ be a Riemann surface homeomorphic to~$S$.
By definition, a \emph{$G$-Higgs bundle} over~$\Sigma$ is a pair $(E,\varphi)$ where $E$ is a holomorphic $K_{\C}$-bundle over~$\Sigma$ and $\varphi$ (the \emph{Higgs field}) is a holomorphic section of a certain natural bundle over~$\Sigma$ associated to~$E$.
(Here $K_{\C}$ is the complexification of a maximal compact subgroup $K$ of~$G$.)
The \emph{non-Abelian Hodge correspondence} of Hitchin, Donaldson, Corlette, Simpson, and others (see \cite{ggm09}), gives a homeomorphism between the $G$-character variety of $\pi_1(S)$ and the moduli space $\mathcal{M}_G(\Sigma)$ of so-called \emph{polystable} $G$-Higgs bundles over~$\Sigma$.
This was used by Hitchin to define and study the Hitchin component.

Some of the connected components of the $G$-character variety of $\pi_1(S)$ can be distinguished using topological invariants.
However, such invariants are not sufficient to distinguish them all in general.
One fruitful approach is to use the fact, proved by Hitchin, that $(E,\varphi) \mapsto \Vert\varphi\Vert_{L^2(\Sigma)}^2$ defines a proper Morse function $f$ from $\mathcal{M}_G(\Sigma)$ to $\R_{\geq 0}$; therefore, the connected components of $\mathcal{M}_G(\Sigma)$ can be studied by examining the local minima of~$f$.
The zero locus $f^{-1}(0)$ of~$f$ corresponds, in the $G$-character variety, to representations of $\pi_1(S)$ whose image lies in a compact subgroup of~$G$; in particular, these representations are not injective and discrete, and so connected components for which $f$ has a local minimum of~$0$ cannot be higher Teichm\"uller spaces.
This approach has already been successfully exploited to find and count almost all connected components of the $G$-character variety of $\pi_1(S)$ for simple~$G$, including conjecturally all higher Teichm\"uller spaces: see \cite{bra24,bcggo,bgg06,g20}.

\subsubsection{Similarities with the classical Teichm\"uller space}

The study of higher Teichm\"uller spaces, or \emph{higher Teichm\"uller theory}, has been very active in the past twenty years.
In particular, striking similarities have been found between higher Teichm\"uller spaces and the classical Teichm\"uller space of~$S$, including:
\begin{itemize}
  \item associated notions of positivity (see above);
  \item for Hitchin components: the topology of $\R^{\dim(G) |\chi(S)|}$ (Hitchin);
  \item the proper discontinuity of the action of the mapping class group (Labourie, Wienhard);
  \item good systems of coordinates (Goldman, Fock--Goncharov, Bonahon--Dreyer, Strubel, Zhang);
  \item analytic Riemannian metrics invariant under the mapping class group (Bridgeman--Canary--Labourie--Sambarino, Pollicott--Sharp);
  \item natural maps to the space of geodesic currents on~$S$ (Labourie,\linebreak Bridge\-man--Canary--La\-bourie--Sambarino, Martone--Zhang, Ouyang--Tamburelli);
  \item versions of the collar lemma for the associated locally symmetric spaces (Lee--Zhang, Burger--Pozzetti, Beyrer--Pozzetti, Beyrer--Gui\-chard--Labourie--Pozzetti--Wienhard);
  \item interpretations of higher Teichm\"uller spaces as moduli spaces of geometric structures on~$S$ or on closed manifolds fibering over~$S$ (Choi--Goldman, Guichard--Wienhard, Collier--Tholozan--Toulisse).
\end{itemize}
There are also conjectural interpretations of higher Teichm\"uller spaces as moduli spaces of ``higher complex structures'' on~$S$ (Fock--Thomas), as well as various approaches to see higher Teichm\"uller spaces as mapping-class-group-equivariant fiber bundles over the classical Teichm\"uller space of~$S$ (Labourie, Loftin, Alessandrini--Collier, Collier--Tholozan--Toulisse).
We refer to \cite{biw14,oberwolfach22,poz-bourbaki,wie-icm} for more details and references.

\subsubsection{Higher higher Teichm\"uller spaces}

Phenomena analogous to Theorem~\ref{thm:Hitchin-comp} have also been uncovered for fundamental groups of higher-dimensional manifolds, in two situations.

The first one is in the context of \emph{convex projective geometry}, which is by definition the study of properly convex open subsets $\Omega$ of real projective spaces $\PP(\R^d)$, as in Choi--Goldman's proof of Theorem~\ref{thm:Hitchin-comp} for $d=3$.
Let $\Gamma_0 = \pi_1(M)$ where $M$ is a closed topological manifold of dimension $n\geq 2$.
Generalising Theorem~\ref{thm:Hitchin-comp}, Benoist \cite{ben05} proved that if $\Gamma_0$ does not contain an infinite nilpotent normal subgroup, then the set of representations through which $\Gamma_0$ acts properly discontinuously with compact quotient on some properly convex open subset of $\PP(\R^{n+1})$ is closed in $\Hom(\Gamma_0,G)$ for $G=\PGL(n+1,\R)$.
This set is also open in $\Hom(\Gamma_0,G)$ by Koszul, and so it is a union of connected components of $\Hom(\Gamma_0,G)$.
It consists entirely of injective and discrete representations.
Recent results of Marseglia and Cooper--Tillman extend this to some cases where $M$ and the quotients of the properly convex sets are not necessarily closed (see \cite{ct}).

The second situation is in the context of \emph{pseudo-Riemannian hyperbolic geometry}, which is by definition the study of pseudo-Riemannian manifolds (\ie smooth manifolds with a smooth assignment, to each tangent space, of a nondegenerate quadratic form) which have constant negative sectional curvature.
In signature $(p,q)$, such manifolds are locally modeled on the pseudo-Riemannian symmetric space $\HH^{p,q} = \PO(p,q+1)/\mathrm{P}(\OO(p)\times\OO(q+1))$, which can be realised as an open set in projective space, namely $\{ [v]\in\PP(\R^{p+q+1}) \,|\, \langle v,v\rangle_{p,q+1}<0\}$ where $\langle\cdot,\cdot\rangle_{p,q+1}$ is a symmetric bilinear form of signature $(p,q+1)$ on $\R^{p+q+1}$.
For $q=0$ we recover the real hyperbolic space $\HH^p$, with its projective model \eqref{eqn:proj-model-Hn}, and for $q=1$ the space $\HH^{p,1}$ is the $(p+1)$-dimensional \emph{anti-de Sitter space} (a Lorentzian analogue of the real hyperbolic space).
Let $\Gamma_0 = \pi_1(M)$ where $M$ is a closed hyperbolic $p$-manifold, with holonomy $\sigma_0 : \Gamma_0\to\OO(p,1)$, and let $\tau : \OO(p,1)\hookrightarrow G=\PO(p,q+1)$ be the standard embedding.
For $q=1$, Barbot \cite{bar15} proved that the connected component of $\rho_0 = \tau\circ\sigma_0$ in $\Hom(\Gamma_0,G)$ consists entirely of injective and discrete representations (corresponding to holonomies of so-called \emph{globally hyperbolic spatially compact} anti-de Sitter manifolds, studied in \cite{mes90} for $p=2$). 
This was recently extended in \cite{bk-Hpq-cc} to general $p\geq 2$ and $q\geq 1$.
In fact, the following more general result is proved in \cite{bk-Hpq-cc}: for $\Gamma_0 = \pi_1(M)$ where $M$ is any closed topological manifold of dimension $p\geq 2$, the set of so-called \emph{$\HH^{p,q}$-convex cocompact} representations is a union of connected components in $\Hom(\Gamma_0,G)$.
These $\HH^{p,q}$-convex cocompact representations are injective and discrete representations with a nice geometric behaviour in $\HH^{p,q}$ (see Section~\ref{subsec:Ano-cc}); they include the representations $\tau\circ\sigma_0 : \Gamma_0 = \pi_1(M)\to\OO(p,1)\hookrightarrow G=\PO(p,q+1)$ above where $M$ is a closed hyperbolic manifold, but also other examples where $M$ can be quite ``exotic'' (see \cite{lm19,mst} for $q=1$).
These representations can have Zariski-dense image in~$G$: see \eg \cite{bk-Hpq-cc} for a bending argument as in Section~\ref{subsec:ex-rank-1}.

In these two situations, there are connected components in $\Hom(\Gamma_0,G)$ consisting entirely of injective and discrete representations, where $\Gamma_0$ is the fundamental group of an $n$-dimensional closed manifold with $n>2$ and $G$ is a semisimple Lie group with $\Rrank(G)\geq 2$.
It is natural to call \emph{higher-dimensional higher-rank Teichm\"uller spaces} (or \emph{higher higher Teichm\"uller spaces} for short) the images of these components in the $G$-character variety of $\Gamma_0$.
It would be interesting in the future to investigate whether these higher higher Teichm\"uller spaces have any topological or geometric analogies with classical Teichm\"uller space or its higher-rank counterparts, as above.
See also \cite[\S\,14]{wie-icm} for some further discussion.

\section{Classes of discrete subgroups in real rank one} \label{sec:rank-1}

In Section~\ref{sec:flex-examples} we saw various examples of ``flexible'' infinite discrete subgroups of semisimple Lie groups.
We now present some general theory in which these examples fit, first in real rank one (this section), then in higher real rank (Section~\ref{sec:higher-rank}).

More precisely, throughout this section we consider a semisimple Lie group $G$ with $\Rrank(G) = 1$.
We discuss two important classes of finitely generated discrete subgroups of~$G$ that have received considerable attention, namely convex cocompact subgroups and geometrically finite subgroups.
The inclusion relations between these classes and lattices of~$G$ are shown in Figure~\ref{fig:rank-1}.

\begin{figure}[h!]
\includegraphics[scale=0.8]{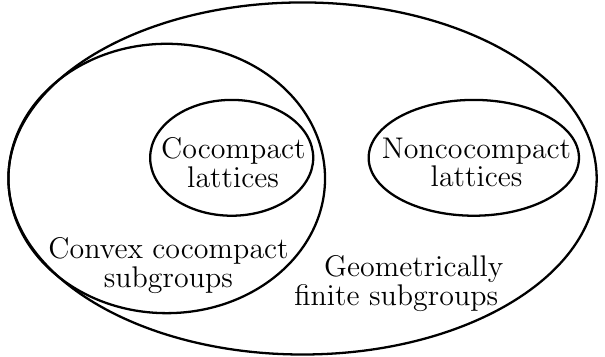}
\caption{Inclusions between four important classes of discrete subgroups of~$G$ for $\Rrank(G) = 1$}
\label{fig:rank-1}
\end{figure}

\subsection{Definitions}

Consider, as in Section~\ref{subsec:G/K}, the Riemannian symmetric space $X = G/K$, where $K$ is a maximal compact subgroup of~$G$.
If $G = \SO(n,1)$ (\resp $\SU(n,1)$, \resp $\Sp(n,1)$), then $X$ is the $n$-dimensional hyperbolic space over $\R$ (\resp $\C$, \resp the quaternions).
If $G$ is the exceptional group $F_{4(-20)}$, then $X$ is the ``hyperbolic plane over the octonions''.

There is a natural notion of convexity in~$X$: any two points $x,y$ of~$X$ are joined by a unique geodesic segment; we say that a subset $\mathcal{C}$ of~$X$ is convex if this segment is contained in~$\mathcal{C}$ for all $x,y\in X$.
See \cite[\S\,1.6]{ebe96} for more details.
For any $\varepsilon>0$ and any subset $\mathcal{C}$ of~$X$, we denote by $\mathcal{U}_{\varepsilon}(\mathcal{C})$ the uniform $\varepsilon$-neighbourhood of $\mathcal{C}$ in~$X$.

\begin{defn} \label{def:cc-gf-rank-1}
Suppose $\Rrank(G)=1$.
A discrete subgroup $\Gamma$ of~$G$ is \emph{convex cocompact} (\resp \emph{geometrically finite}) if it is finitely generated and there is a nonempty $\Gamma$-invariant convex subset $\mathcal{C}$ of~$X$ such that the quotient $\Gamma\backslash\mathcal{C}$ is compact (\resp the quotient $\Gamma\backslash\mathcal{U}_{\varepsilon}(\mathcal{C})$ has finite volume for some $\varepsilon>0$).

Alternatively (see Remark~\ref{rem:two-points-of-view}), given a group $\Gamma_0$, we say that a representation $\rho : \Gamma_0\to G$ is \emph{convex cocompact} (\resp \emph{geometrically finite}) if it has finite kernel and discrete, convex cocompact (\resp geometrically finite) image.
\end{defn}

\begin{rem} \label{rem:cc-gf-finitely-gen}
If there is a nonempty $\Gamma$-invariant convex subset $\mathcal{C}$ of~$X$ such that $\Gamma\backslash\mathcal{C}$ is compact, then $\Gamma$ is automatically finitely generated, by the \v{S}varc--Milnor lemma (see \eg \cite[Th.\,8.37]{dk18}).
Thus the assumption that $\Gamma$ be finitely generated can be omitted in the definition of convex cocompactness.
On the other hand, this assumption cannot be omitted in general in the definition of geometric finiteness: see \cite{ham98}.
\end{rem}

\begin{rem}
Bowditch \cite{bow95} gave several equivalent definitions of geometric finiteness.
Here we use a variation on his definition F5, where the uniform bound on the orders of finite subgroups of~$\Gamma$ is replaced by the assumption that $\Gamma$ be finitely generated.
The two definitions are equivalent by \cite{bow95} and the Selberg lemma \cite[Lem.\,8]{sel60}.
\end{rem}

We now explain how Definition~\ref{def:cc-gf-rank-1} can be rephrased in terms of a specific convex set in $G/K$.
For this, we first recall the important notion of the \emph{limit set} of a discrete subgroup of~$G$.

\subsubsection{Limit sets and convex cores}

Let $\di X$ be the visual boundary of $X = G/K$, \ie the set of equivalence classes of geodesic rays in~$X$ for the equivalence relation ``to remain at bounded distance''.
There is a natural topology on $\overline{X} := X \sqcup \di X$ that extends that of~$X$ and makes $\overline{X}$ compact, and the action of $G$ on~$X$ extends continuously to~$\overline{X}$ (see \eg \cite[\S\,1.7]{ebe96}).
For instance, as in Sections \ref{subsec:ex-rank-1} and~\ref{subsec:higher-Teich}, if $G = \PO(n,1)$ and $X = \HH^n$, then we can realise $X$ as the open subset \eqref{eqn:proj-model-Hn} of $\PP(\R^{n+1})$ where some quadratic form of signature $(n,1)$ is negative, and $\overline{X}$ is then the closed subset of $\PP(\R^{n+1})$ where the quadratic form is nonpositive, endowed with the topology from $\PP(\R^{n+1})$ and the natural action of $G=\PO(n,1)$.

\begin{defn} \label{def:lim-set-rank-1}
Let $\Gamma$ be a discrete subgroup of~$G$.
The \emph{limit set} of~$\Gamma$ is the set $\Lambda_{\Gamma}$ of accumulation points in~$\overline{X}$ of a $\Gamma$-orbit of~$X$; it is contained in $\di X$ and does not depend on the choice of $\Gamma$-orbit.
The \emph{convex core} $\Ccore_{\Gamma} \subset X$ of~$\Gamma$ is the convex hull of $\Lambda_{\Gamma}$ in~$X$ (\ie the smallest closed convex subset of~$X$ whose closure in~$\overline{X}$ contains~$\Lambda_{\Gamma}$).
\end{defn}

Note that $\Lambda_{\Gamma}$ and $\Ccore_{\Gamma}$ are both invariant under the action of $\Gamma$ on~$\overline{X}$.

The limit set $\Lambda_{\Gamma}$ is nonempty if and only if $\Gamma$ is infinite.
This set has either at most two elements (in which case we say $\Gamma$ is \emph{elementary}), or infinitely many.
If $\Gamma$ is not elementary, then the action of $\Gamma$ on $\Lambda_{\Gamma}$ is minimal (all orbits are dense), and any nonempty $\Gamma$-invariant closed subset of $\di X$ contains~$\Lambda_{\Gamma}$ (see \eg \cite[\S\,3.2]{bow95}); in particular, any nonempty $\Gamma$-invariant closed convex subset of~$X$ contains the convex core~$\Ccore_{\Gamma}$.
We deduce~the~following.

\begin{fact} \label{fact:cc-gf-rk-1-Ccore}
Suppose $\Rrank(G)=1$.
A finitely generated infinite discrete subgroup $\Gamma$ of~$G$ is convex cocompact (\resp geometrically finite) if and only if the quotient $\Gamma\backslash\Ccore_{\Gamma}$ is compact and nonempty (\resp the quotient $\Gamma\backslash\mathcal{U}_{\varepsilon}(\Ccore_{\Gamma})$ has finite volume for some $\varepsilon>0$).
\end{fact}

\begin{rem} \label{rem:bound-G/K-rank-1}
In our setting where $\Rrank(G) = 1$, the group $G$ acts transitively on $\di X$.
The stabilisers in~$G$ of points of $\di X$ are the \emph{proper parabolic subgroups} of~$G$.
Thus $\di X$ is $G$-equivariantly homeomorphic to $G/P$ where $P$ is a proper parabolic subgroup of~$G$.
\end{rem}

\subsection{Properties} \label{subsec:cc-open-rank-1}

Let us briefly mention a few useful properties of geometrically finite and convex cocompact representations.

\subsubsection{Domains of discontinuity}

We first observe that any discrete subgroup $\Gamma$ of~$G$ acts properly discontinuously on the open subset $\Omega_{\Gamma} := \di X \smallsetminus \Lambda_{\Gamma}$ of $\di X$, and in fact on $X \cup \Omega_{\Gamma}$.
Indeed, let $\mathcal{C}$ be a nonempty $\Gamma$-invariant closed convex subset of~$X$.
One can check that the closest point projection from $X$ to $\mathcal{C}$ extends to a continuous $\Gamma$-equivariant map from $X \cup \Omega_{\Gamma}$ to $\mathcal{C}$.
The fact that $\Gamma$ acts properly discontinuously on $\mathcal{C} \subset X$ then implies that $\Gamma$ acts properly discontinuously on $X \cup \Omega_{\Gamma}$.

If $\Gamma$ is convex cocompact, then the quotient $\Gamma\backslash\Omega_{\Gamma}$ is compact (possibly empty), and $\Gamma\backslash (X\cup\Omega_{\Gamma})$ is a compactification of $\Gamma\backslash X$.

If $\Gamma$ is geometrically finite, then $\Gamma\backslash (X\cup\Omega_{\Gamma})$ is not necessarily compact, but it has only finitely many topological ends, each of which is a ``pa\-rabolic end''; this actually characterises geometric finiteness: see~\cite{bow95}.

\subsubsection{Deformations}

Convex cocompactness is stable under small deformations:

\begin{fact} \label{fact:cc-open-rank-1}
Suppose $\Rrank(G)=1$.
For any finitely generated group~$\Gamma_0$, the space of convex cocompact representations is open in $\Hom(\Gamma_0,G)$.
\end{fact}

On the other hand, geometric finiteness is in general not stable under small deformations.
If one restricts to small deformations that are \emph{cusp-preserving} (\ie that keep parabolic elements parabolic), then stability holds for $G=\PO(n,1)$ when $n\leq 3$ or when all cusps have rank $\geq n-2$, but not in general.
See \eg \cite[App.\,B]{gk17} for more details and references.

\subsubsection{Homomorphisms}

Convex cocompactness behaves well under Lie group homomorphisms:

\begin{fact} \label{fact:cc-embed-rank-1}
Suppose $\Rrank(G)=1$.
Let $G'$ be another semisimple Lie group with $\Rrank(G')=1$ and let $\tau : G'\to G$ be a Lie group homomorphism with compact kernel.
For any finitely generated group~$\Gamma_0$ and any representation $\sigma_0 : \Gamma_0\to G'$, the composed representation $\tau\circ\sigma_0 : \Gamma_0\to G$ is convex cocompact if and only if $\sigma_0$ is.
\end{fact}

\subsection{Examples} \label{subsec:ex-rk1-cc-gf}

\begin{itemize}
  \item If $\Gamma$ is a lattice in~$G$, then $\Lambda_{\Gamma} = \di X$ and $\Ccore_{\Gamma} = X$, and $\Gamma$ is geometrically finite.
  If $\Gamma$ is cocompact in~$G$, then it is convex cocompact.
  \item Suppose $G = \PSL(2,\R) \simeq \PO(2,1)_0$.
  Then every finitely generated discrete subgroup $\Gamma$ of~$G$ is geometrically finite; $\Gamma$ is convex cocompact if and only if the associated hyperbolic surface $\Gamma\backslash\HH^2$ has no cusps.
  
  \begin{rem}
  On the other hand, for $G = \PO(n,1)$ with $n\geq\nolinebreak 3$, there exist finitely generated discrete subgroups of~$G$ which are \emph{not} geometrically finite.
  The first examples were given by Bers for $n=3$ (``singly degenerate'' Kleinian groups, for which the domain of discontinuity $\Omega_{\Gamma}$ is simply connected): see \cite[\S\,2]{kap07}.
  \end{rem}
  \item Any discrete subgroup of $G = \PO(n,1)$ generated by the orthogonal reflections in the faces of a finite-sided right-angled polyhedron of~$\HH^n$ is geometrically finite; it is convex cocompact if and only if no distinct facets of the polyhedron have closures meeting in $\di\HH^n$ (see \cite[\S\,4]{dh13}).
  \item The Schottky groups of Section~\ref{subsec:ex-rank-1} are geometrically finite; the strong Schottky groups (for which $B_1^{\pm},\dots,B_m^{\pm}$ have pairwise disjoint closures) are convex cocompact.
  Their limit sets are Cantor sets.
  The set $\Omega$ of Remark~\ref{rem:DoD-rank-one} is the domain of discontinuity $\Omega_{\Gamma} = \di X \smallsetminus \Lambda_{\Gamma}$ of $\Gamma$ in $\di X$ from Section~\ref{subsec:cc-open-rank-1}.
  \item Any quasi-Fuchsian group $\Gamma = \rho(\pi_1(S))$ as in Section~\ref{subsec:ex-rank-1} is convex cocompact.
  The limit set $\Lambda_{\Gamma}$ is a topological circle in $\di\HH^3$ (see Figure~\ref{fig:QFLimSet}).
  The quotient $\Gamma\backslash\Ccore_{\Gamma}$ is homeomorphic to $S\times [0,1]$.
  \item The small deformations of cocompact lattices of $G'=\SO(n,1)$ inside $G=\SO(n+1,1)$ from Section~\ref{subsec:ex-rank-1} are convex cocompact by Facts \ref{fact:cc-open-rank-1} and~\ref{fact:cc-embed-rank-1} (see also Remark~\ref{rem:bending}.\eqref{item:bending-2}).
\end{itemize}

\subsection{A few characterisations of convex cocompactness} \label{subsec:charact-cc-rank-1}

\subsubsection{Preliminaries}

Given a finitely generated group $\Gamma_0$, we choose a finite generating subset $F$ of~$\Gamma_0$ and denote by $\mathrm{Cay}(\Gamma_0) = \mathrm{Cay}(\Gamma_0,F)$ the corresponding Cayley graph, with its metric $\mathtt{d}_{\mathrm{Cay}(\Gamma_0)}$.

As in Section~\ref{subsec:rank}, a group $\Gamma_0$ is called \emph{Gromov hyperbolic} if it is finitely generated and acts properly discontinuously, by isometries, with compact quotient, on some Gromov hyperbolic proper geodesic metric space~$Y$; in that case, we can take $Y$ to be $\mathrm{Cay}(\Gamma_0)$.
As in the proof of Theorem~\ref{thm:Hitchin-comp}, the \emph{Gromov boundary} $\di\Gamma_0$ of~$\Gamma_0$ is then the visual boundary of~$Y$, endowed with the action of~$\Gamma_0$ extending that on~$Y$.
The Gromov boundary $\di\Gamma_0$ does not depend on~$Y$ up to $\Gamma_0$-equivariant homeomorphism.
An important property is that the action of $\Gamma_0$ on $\di\Gamma_0$ is a \emph{convergence action}: for any sequence $(\gamma_k)_{k\in\N}$ of pairwise distinct elements of~$\Gamma_0$, up to passing to a subsequence, there exist $w^+,w^-\in\di\Gamma_0$ such that $\gamma_k\cdot w\to w^+$ for all $w\in\di\Gamma_0\smallsetminus\{w^-\}$, uniformly on compact sets.
Moreover, any infinite-order element of~$\Gamma_0$ has two fixed points in $\di\Gamma_0$, one attracting and one repelling.
The group $\Gamma_0$ is called \emph{elementary} if it is finite (in which case $\di\Gamma_0$ is empty) or if it admits a finite-index subgroup which is cyclic (in which case $\di\Gamma_0$ consists of two points).
If $\Gamma_0$ is not elementary, then the set of attracting fixed points of infinite-order elements of~$\Gamma_0$ is infinite and dense in $\di\Gamma_0$.
See \eg \cite{bk02} for details.

\begin{examples} \label{ex:Gromov-hyp-groups}
If $\Gamma_0$ is a nonabelian free group with finite free generating subset~$F$, then $\Gamma_0$ is Gromov hyperbolic, $\mathrm{Cay}(\Gamma_0)$ is a tree, and $\di\Gamma_0$ is a Cantor set.
If $\Gamma_0 = \pi_1(M)$ for some closed negatively-curved manifold~$M$, then $\Gamma_0$ is Gromov hyperbolic, we can take $Y$ to be the universal cover $\widetilde{M}$ of~$M$, and $\di\Gamma_0 = \di\widetilde{M}$.
In particular, if $\Gamma_0 = \pi_1(S)$ for some closed orientable surface of genus $\geq 2$, then $\Gamma_0$ is Gromov hyperbolic and $\di\Gamma_0$ is a circle (as in the proof of Theorem~\ref{thm:Hitchin-comp}).
\end{examples}

\begin{rem} \label{rem:hyp-gp-no-Z2}
A Gromov hyperbolic group can never contain a subgroup isomorphic to $\Z^2$ or to a Baumslag--Solitar group $BS(m,n) := \langle a,t \,|\, t^{-1} a^m t = a^n\rangle$.
Understanding how close this is to characterising Gromov hyperbolic groups is an important question in geometric group theory: see \eg \cite{gkl}.
\end{rem}

For any isometry $g$ of a metric space $(M,\mathtt{d}_M)$, we define the \emph{translation length} of $g$ in~$M$ to be
\begin{equation} \label{eqn:transl-length}
\mathrm{transl}_M(g) := \inf_{m\in M} \mathtt{d}_M(m,g\cdot m) \geq 0.
\end{equation}

Finally, we denote by $\mathtt{d}_X$ the metric on the Riemannian symmetric space $X=G/K$ (see Section~\ref{subsec:G/K}).
We fix a basepoint $x_0\in X$, and a Riemannian metric $\mathtt{d}_{\di X}$ on the visual boundary $\di X$.

\subsubsection{A few classical characterisations}

Many interesting characterisations of convex cocompactness have been found by various authors including Beardon, Bowditch, Maskit, Sullivan, Thurston, Tukia, and others.
We now give a few.
We refer to \cite{kap07,kl-msri} for more details and references, as well as further characterisations (\eg in terms of \emph{conical limit points}).
We also refer to \cite{bow95,kap07} for characterisations of geometric finiteness.

\begin{thm} \label{thm:charact-cc-rank-1}
Suppose $\Rrank(G)=1$.
For any infinite group $\Gamma_0$ and any representation $\rho : \Gamma_0\to G$, the following are equivalent:
\begin{enumerate}[(1)]
  \item\label{item:rk1-cc} $\rho$ is convex cocompact (Definition~\ref{def:cc-gf-rank-1});
  \item\label{item:rk1-qi} $\Gamma_0$ is finitely generated and $\rho$ is a \emph{quasi-isometric embedding}: there exist $c,c'>0$ such that for any $\gamma\in\Gamma$,
  \begin{equation} \label{eqn:qi-embed}
  \mathtt{d}_X(x_0,\rho(\gamma)\cdot x_0) \geq c\,\mathtt{d}_{\mathrm{Cay}(\Gamma_0)}(e,\gamma) - c' ;
  \end{equation}
  \item\label{item:rk1-displac} $\Gamma_0$ is Gromov hyperbolic and $\rho$ is \emph{well-displacing}: there exist $c,c''>\nolinebreak 0$ such that for any $\gamma\in\Gamma$,
  \begin{equation} \label{eqn:displac}
  \mathrm{transl}_X(\rho(\gamma)) \geq c\,\mathrm{transl}_{\mathrm{Cay}(\Gamma_0)}(\gamma) - c'' ;
  \end{equation}
  \item\label{item:rk1-bound-map-dyn-preserv} $\Gamma_0$ is Gromov hyperbolic and there exists a $\rho$-equivariant map
  $$\xi : \di\Gamma_0 \longrightarrow \di X$$
  which is continuous, injective, and \emph{dynamics-preserving} (\ie for any infinite-order element $\gamma\in\Gamma_0$, the image by~$\xi$ of the attracting fixed point of $\gamma$ in $\di\Gamma_0$ is an attracting fixed point of $\rho(\gamma)$ in $\di X$);
  \smallskip
  \item\label{item:rk1-bound-map-strong-dyn-preserv} $\Gamma_0$ is Gromov hyperbolic and there exists a $\rho$-equivariant map
  $$\xi : \di\Gamma_0 \longrightarrow \di X$$
  which is continuous, injective, and \emph{strongly dynamics-preserving} (\ie for any $(\gamma_k)\in\Gamma_0^{\N}$ and any $w^+,w^-\in\di\Gamma_0$, if $\gamma_k\cdot w\to w^+$ for all $w\in\di\Gamma_0\smallsetminus\{w^-\}$, then $\rho(\gamma_k)\cdot z\to\xi(w^+)$ for all $z\in\di X\smallsetminus\{\xi(w^-)\}$);
  \smallskip
  \item\label{item:rk1-expand} $\rho$ has finite kernel, discrete image, and the action of $\Gamma$ on $\di X$ via~$\rho$ is \emph{expanding} at $\Lambda_{\rho(\Gamma_0)}$, \ie for any $z\in\Lambda_{\rho(\Gamma_0)}$, there exist a neighbourhood $\mathcal{U}$ of $z$ in $\di X$ and an element $\gamma\in\Gamma_0$ such that
  \begin{equation} \label{eqn:expand}
  \inf_{z_1\neq z_2\ \mathrm{in}\ \mathcal{U}} \, \frac{\mathtt{d}_{\di X}(\rho(\gamma)\cdot z_1,\rho(\gamma)\cdot z_2)}{\mathtt{d}_{\di X}(z_1,z_2)} > 1.
  \end{equation}
\end{enumerate}
\end{thm}

\begin{rems} \label{rem:charact-cc-rank-1}
\begin{itemize}
  \item Using the triangle inequality, one sees that condition~\eqref{item:rk1-qi} does not depend on the choice of basepoint $x_0\in X$ (changing $x_0$ may change the values of $c,c'$ but not their existence).
  \item One also sees that for $\Gamma_0$ with finite generating subset $F$, the reverse inequality $\mathtt{d}_X(x_0,\rho(\gamma)\cdot x_0) \leq C\,\mathtt{d}_{\mathrm{Cay}(\Gamma_0)}(e,\gamma)$ to \eqref{eqn:qi-embed} holds for any representation $\rho : \Gamma_0\to G$, with $C := \max_{f\in F} \mathtt{d}_X(x_0,\rho(f)\cdot x_0)$.
  \item In condition~\eqref{item:rk1-displac} we cannot remove the assumption that $\Gamma_0$ be Gromov hyperbolic: for instance, there exist finitely generated infinite groups $\Gamma_0$ with only finitely many conjugacy classes \cite{osi10}, and for such~$\Gamma_0$ any representation $\rho : \Gamma_0\to G$ is well-displacing.
  \item In condition~\eqref{item:rk1-bound-map-dyn-preserv}, \emph{dynamics-preserving} implies that for any $\gamma\in\Gamma_0$ of infinite order, $\rho(\gamma)$ is a \emph{hyperbolic} element of~$G$ (\ie an element with two fixed points in $\di X$, one attracting and one repelling).
  In condition~\eqref{item:rk1-bound-map-strong-dyn-preserv}, \emph{strongly dynamics-preserving} means that $\xi$ preserves the convergence action of $\Gamma_0$ on $\di\Gamma_0$ mentioned above.
\end{itemize}
\end{rems}

\subsubsection{Sketches of proofs}

\begin{proof}[Proof of \eqref{item:rk1-cc}~$\Rightarrow$~\eqref{item:rk1-qi}:]
We may assume that the basepoint $x_0$ belongs to the convex core $\Ccore_{\rho(\Gamma_0)}$.
By the \v{S}varc--Milnor lemma (see \eg \cite[Th.\,8.37]{dk18}), if $\Gamma_0$ acts properly discontinuously, by isometries, with compact quotient, on a proper geodesic metric space $M$, then $\Gamma_0$ is finitely generated and any orbital map $\gamma\mapsto\gamma\cdot m$ is a quasi-isometric embedding: there exist $c,c'>0$ such that $\mathtt{d}_M(m,\gamma\cdot m) \geq c\,\mathtt{d}_{\mathrm{Cay}(\Gamma_0)}(e,\gamma) - c'$ for all $\gamma\in\Gamma_0$.
We apply this to the convex core $M = \Ccore_{\rho(\Gamma_0)}$, endowed with the restriction of the metric~$\mathtt{d}_X$.
\end{proof}

\begin{proof}[Proof of \eqref{item:rk1-qi}~$\Rightarrow$~\eqref{item:rk1-cc}:]
Since $\mathtt{d}_X(x_0,\rho(\gamma)\cdot x_0) \to +\infty$ as $\mathtt{d}_{\mathrm{Cay}(\Gamma_0)}(e,\gamma)\to +\infty$, the representation $\rho$ has finite kernel and discrete image.

The orbital map $\gamma\mapsto\rho(\gamma)\cdot x_0$ from $\Gamma_0$ to~$X$ extends to a map from $\mathrm{Cay}(\Gamma_0)$ to~$X$ sending edges of $\mathrm{Cay}(\Gamma_0)$ to geodesic segments of~$X$.
The fact that $\rho$ is a quasi-isometric embedding implies the existence of $c,c'>\nolinebreak 0$ such that any geodesic of $\mathrm{Cay}(\Gamma_0)$ is sent to a $(c,c')$-quasigeodesic in~$X$, and the Morse lemma (see \eg \cite[Th.\,11.40 \& 11.105]{dk18}) states that $(c,c')$-quasigeodesics are uniformly close to actual geodesics in~$X$.
Therefore the orbit $\rho(\Gamma_0)\cdot x_0$ is \emph{quasiconvex}: there exists a uniform neighbourhood $\mathcal{U}$ of $\rho(\Gamma_0)\cdot x_0$ in~$X$ such that any geodesic segment between two points of $\rho(\Gamma_0)\cdot x_0$ is contained in~$\mathcal{U}$.
We conclude using the fact (see \cite[Prop.\,2.5.4]{bow95}) that any quasiconvex subset of~$X$ lies at finite Hausdorff distance from its convex hull in~$X$.
\end{proof}

In order to prove \eqref{item:rk1-qi}~$\Rightarrow$~\eqref{item:rk1-displac}, we consider, for any metric space $(M,\mathtt{d}_M)$ and any isometry $g$ of~$M$, the \emph{stable length}
$$\mathrm{length}^{\infty}_M(g) := \lim_k \frac{1}{k} \, \mathtt{d}_M(m,g^k\cdot m) \geq 0$$
of~$g$.
It is an easy exercise to check, using the triangle inequality, that this limit exists (because the sequence $(d_M(m,g^k\cdot m))_{k\in\N}$ is subadditive) and that it does not depend on the choice of $m\in M$.
Note that
\begin{equation} \label{eqn:transl-stable-length}
\mathrm{length}^{\infty}_M(g) \leq \mathrm{transl}_M(g).
\end{equation}
Indeed, for any $m\in M$ and any $k\geq 1$ we have $\mathtt{d}_M(m,g^k\cdot m) \leq k\,\mathtt{d}_M(m,g\cdot m)$ by the triangle inequality.
Dividing by~$k$ and passing to the limit yields $\mathrm{length}^{\infty}_M(g) \leq \mathtt{d}_M(m,g\cdot m)$, and we conclude by taking an infimum over all $m\in M$ on the right-hand side.

\begin{proof}[Proof of \eqref{item:rk1-qi}~$\Rightarrow$~\eqref{item:rk1-displac}:]
Applying \eqref{eqn:qi-embed} to $\gamma^k$ instead of~$\gamma$, dividing by~$k$, and passing to the limit yields $\mathrm{length}^{\infty}_X(\rho(\gamma)) \geq c\,\mathrm{length}^{\infty}_{\mathrm{Cay}(\Gamma_0)}(\gamma)$ for all $\gamma\in\Gamma$.
In order to obtain \eqref{eqn:displac}, it is sufficient to use \eqref{eqn:transl-stable-length} for $M = X$ and to check that
\begin{enumerate}[(i)]
  \item\label{item:transl-stable-length-Cay} for $M = \mathrm{Cay}(\Gamma_0)$, the inequality \eqref{eqn:transl-stable-length} is ``almost'' an equality: $\mathrm{length}^{\infty}_{\mathrm{Cay}(\Gamma_0)}(g) \geq \mathrm{transl}_{\mathrm{Cay}(\Gamma_0)}(g) - 8\delta$ where $\delta\geq 0$ is a hyperbolicity constant for $\mathrm{Cay}(\Gamma_0)$ (\ie all triangles of $\mathrm{Cay}(\Gamma_0)$ are $\delta$-thin).
\end{enumerate}
Indeed, then \eqref{eqn:displac} will hold with $c'' = 8\delta$.
We note that actually
\begin{enumerate}[(i)] \setcounter{enumi}{1}
  \item\label{item:transl-stable-length-X} for $M = X = G/K$, the inequality \eqref{eqn:transl-stable-length} is an equality.
\end{enumerate}

Indeed, \eqref{item:transl-stable-length-X} is based on the fact that $X$ is a \emph{CAT(0) space}: any geodesic triangle of~$X$ is ``at least as thin'' as a triangle with the same side lengths in the Euclidean plane.
Applying this to a geodesic triangle with vertices $m,g\cdot m,g^2\cdot m$, we see that if $m'$ is the midpoint of the geodesic segment $[m,g\cdot m]$ (so that $g\cdot m'$ is the midpoint of $[g\cdot m,g^2\cdot m]$), then\linebreak $\mathtt{d}_X(m',g\cdot m') \leq \mathtt{d}_X(m,g^2\cdot m)/2$.
By induction on~$k$, we obtain that for any $m\in M$ and any $k\geq 1$, there exists $m_k\in M$ such that $\mathtt{d}_X(m,g^{2^k}\cdot m) \geq 2^k\,\mathtt{d}_X(m_k,g\cdot m_k) \geq 2^k\,\mathrm{transl}_M(g)$.
We conclude by dividing by~$2^k$ and passing to the limit.

\eqref{item:transl-stable-length-Cay} can be proved in a similar way, replacing the CAT(0) inequality $\mathtt{d}_X(m',g\cdot m') \leq \mathtt{d}_X(m,g^2\cdot m)/2$ by the Gromov hyperbolicity inequality $\mathtt{d}_X(m',g\cdot m') \leq \mathtt{d}_X(m,g^2\cdot m)/2 + 4\delta$ (see \cite[Ch.\,10, Prop.\,5.1]{cdp90}).
\end{proof}

\begin{proof}[Proof of \eqref{item:rk1-displac}~$\Rightarrow$~\eqref{item:rk1-qi}:]
The Gromov hyperbolic group $\Gamma_0$ has the following property: there exist a finite subset $S$ of~$\Gamma_0$ and a constant $C'>0$ such that for any $\gamma\in\Gamma_0$ we can find $s\in S$ with $\mathrm{transl}_{\mathrm{Cay}(\Gamma_0)}(s\gamma) \geq \mathtt{d}_{\mathrm{Cay}(\Gamma_0)}(e,\gamma) - C'$.
(If $\Gamma_0$ is nonelementary, then we can take $S = \{ \gamma_1^N,\gamma_1^{-N},\gamma_2^N,\gamma_2^{-N}\}$ for some large~$N$, where $\gamma_1,\gamma_2\in\Gamma_0$ are infinite-order elements such that the attracting fixed points in $\di\Gamma_0$ of $\gamma_1$, $\gamma_1^{-1}$, $\gamma_2$, and $\gamma_2^{-1}$ are pairwise distinct: see \eg \cite[Lem.\,B.2]{zz-rel-Ano-1}.)

Given $\gamma\in\Gamma_0$, consider $s\in S$ as above.
Applying \eqref{eqn:displac} to $s\gamma$ yields
$$\mathrm{transl}_X(\rho(s\gamma)) \geq c\,\mathrm{transl}_{\mathrm{Cay}(\Gamma_0)}(s\gamma) - c'' \geq c \, \mathtt{d}_{\mathrm{Cay}(\Gamma_0)}(e,\gamma) - (c\,C' + c'').$$
To conclude, we observe that
$$\mathrm{transl}_X(g_1g_2) \leq \mathtt{d}_X(x_0,g_1g_2\cdot x_0) \leq \mathtt{d}_X(x_0,g_1\cdot x_0) + \mathtt{d}_X(x_0,g_2\cdot x_0)$$
for all $g_1,g_2\in G$.
Applying this to $(g_1,g_2) = (\rho(s),\rho(\gamma))$, we obtain \eqref{eqn:qi-embed} with $c' = c\,C' + c'' + \max_{s'\in S} \, \mathtt{d}_X(x_0,\rho(s')\cdot x_0)$.
\end{proof}

\begin{proof}[Proof of \eqref{item:rk1-cc}~$\Rightarrow$~\eqref{item:rk1-bound-map-strong-dyn-preserv}:]
We have seen in the proof of \eqref{item:rk1-cc}~$\Rightarrow$~\eqref{item:rk1-qi} that for any $m\in\Ccore_{\rho(\Gamma_0)}$, the orbital map $\gamma \mapsto \rho(\gamma)\cdot m$ is a quasi-isometry from $\Gamma_0$ to $\Ccore_{\rho(\Gamma_0)}$.
It is a classical result in geometric group theory (see \eg \cite[Th.\,11.108]{dk18}) that such a quasi-isometry extends to a $\Gamma_0$-equivariant homeomorphism $\xi$ from $\di\Gamma_0$ to $\di\Ccore_{\rho(\Gamma_0)}$.
Here $\di\Ccore_{\rho(\Gamma_0)}$ is a subset of $\di X$ (namely the intersection of $\di X$ with the closure of $\Ccore_{\rho(\Gamma_0)}$ in~$\overline{X}$).
Thus we can view $\xi$ as a $\rho$-equivariant, continuous, injective map from $\di\Gamma_0$ to $\di X$, such that for any $(\gamma_k)\in\Gamma_0^{\N}$ and any $w^+,w^-\in\di\Gamma_0$, if $\gamma_k\cdot w\to w^+$ for all $w\in\di\Gamma_0\smallsetminus\{w^-\}$, then $\rho(\gamma_k)\cdot z\to\xi(w^+)$ for all $z\in\di\Ccore_{\rho(\Gamma_0)}\smallsetminus\{\xi(w^-)\}$.

In order to see that this last convergence holds for all $z\in\di X\smallsetminus\{\xi(w^-)\}$, one possibility is to use the fact (Cartan decomposition) that if we choose a point $x\in X$ and a geodesic line $\mathcal{G}$ of $X$ through~$x$ with endpoints $z_0^+,z_0^-\in\di X$, then any element $\rho(\gamma_k)\in G$ can be written as $\rho(\gamma_k) = \kappa_k a_k \kappa'_k$ where $\kappa_k,\kappa'_k\in G$ fix~$x$ and $a_k\in G$ is a pure translation along~$\mathcal{G}$ towards~$z_0^+$.
The subgroup of~$G$ fixing~$x$ is compact (it is conjugate to~$K$); therefore, up to passing to a subsequence we may assume that $(\kappa_k)_{k\in\N},(\kappa'_k)_{k\in\N}$ converge respectively to some $\kappa,\kappa'\in G$.
Since $\rho$ has finite kernel and discrete image, we have $a_k\cdot z\to z_0^+$ for all $z\in\di X\smallsetminus\{z_0^-\}$.
Therefore $\rho(\gamma_k)\cdot z\to\kappa\cdot z_0^+$ for all $z\in\di X\smallsetminus\{{\kappa'}^{-1}\cdot z_0^-\}$.
Necessarily $\kappa\cdot z_0^+ = \xi(w^+)$ and ${\kappa'}^{-1}\cdot z_0^- = \xi(w^-)$.
\end{proof}

The implication \eqref{item:rk1-bound-map-strong-dyn-preserv}~$\Rightarrow$~\eqref{item:rk1-bound-map-dyn-preserv} is immediate by considering, for any infinite-order element $\gamma\in\Gamma_0$, the sequence $(\gamma_k) := (\gamma^k) \in \Gamma_0^{\N}$.
The implication \eqref{item:rk1-bound-map-dyn-preserv}~$\Rightarrow$~\eqref{item:rk1-qi} can be proved using flows as in Section~\ref{subsec:Anosov} below (see Remark~\ref{rem:Ano-Zariski-dense} and the implication \eqref{item:Ano}~$\Rightarrow$~\eqref{item:qi-i-direction} in Theorem~\ref{thm:charact-Ano}).

\begin{proof}[Proof of \eqref{item:rk1-cc}~$\Rightarrow$~\eqref{item:rk1-expand}:]
We treat the case that $X$ is $\HH^n$, seen as the open unit ball of~$\R^n$ for a Euclidean norm $\Vert\cdot\Vert$, that $0$ belongs to $\Ccore_{\rho(\Gamma_0)}$, and that $\mathtt{d}_{\di X}$ is the metric induced by $\Vert\cdot\Vert$ on the unit sphere $\di X$ of~$\R^n$.

We first observe that for any element $g\in G$ that does not fix~$0$, the closed subset
$$\mathcal{H}_g := \{ x\in X ~|~ \Vert x\Vert \leq \Vert g\cdot x\Vert\} = \{ x\in X ~|~ \mathtt{d}_X(0,x) \leq \mathtt{d}_X(0,g\cdot x)\}$$
of~$X$ is bounded by the bisector between $0$ and $g^{-1}\cdot 0$.
Moreover, for any neighbourhood $\mathcal{V}$ in~$\overline{X}$ of the closure of $\mathcal{H}_g$ in~$\overline{X}$, the restriction of $g$ to $\di X \smallsetminus \mathcal{V}$ is uniformly expanding in the sense that
$$\inf_{z_1\neq z_2\ \mathrm{in}\ \di X \smallsetminus \mathcal{V}} \frac{\mathtt{d}_{\di X}(g\cdot z_1,g\cdot z_2)}{\mathtt{d}_{\di X}(z_1,z_2)} > 1.$$
Indeed, one can check this when $g$ is a pure translation along a geodesic of~$X$ through~$0$, and then conclude using the fact (Cartan decomposition) that any $g\in G$ can be written as $g = \kappa a\kappa'$ where $a\in G$ is such a pure translation and $\kappa,\kappa'\in G$ fix~$0$ and preserve $\Vert\cdot\Vert$.

Consider the Dirichlet domain of $\Ccore_{\rho(\Gamma_0)}$ centred at~$0$:
$$\mathcal{D} = \bigcap_{\gamma\in\Gamma_0} \mathcal{H}_{\rho(\gamma)} \cap \Ccore_{\rho(\Gamma_0)}.$$
It is compact by~\eqref{item:rk1-cc}.
Since $\Gamma_0$ acts properly discontinuously on $\Ccore_{\rho(\Gamma_0)}$ via~$\rho$, the set $\mathcal{F}$ of elements $\gamma\in\Gamma_0$ such that $\mathcal{D} \cap \rho(\gamma)\cdot\mathcal{D} \neq \emptyset$ and $\rho(\gamma)\cdot 0 \neq 0$ is finite.
One easily checks that $\mathcal{D} = \bigcap_{\gamma\in\mathcal{F}} \mathcal{H}_{\rho(\gamma)} \cap \Ccore_{\rho(\Gamma_0)}$.
For each $\gamma\in\mathcal{F}$, let $\mathcal{V}_{\rho(\gamma)}$ be a closed neighbourhood in~$\overline{X}$ of the closure of $\mathcal{H}_{\rho(\gamma)}$ in~$\overline{X}$.
If we choose these neighbourhoods small enough, then $\mathcal{D}' := \bigcap_{\gamma\in\mathcal{F}} \mathcal{V}_{\rho(\gamma)} \cap \Ccore_{\rho(\Gamma_0)}$ is still a compact subset of~$X$, and so $\Lambda_{\rho(\Gamma_0)} \subset \bigcup_{\gamma\in\mathcal{F}}\ (\di X \smallsetminus \mathcal{V}_{\rho(\gamma)})$.
We conclude using the fact, observed above, that \eqref{eqn:expand} holds for $\mathcal{U} := \di X \smallsetminus \mathcal{V}_{\rho(\gamma)}$ for each $\gamma\in\mathcal{F}$.
\end{proof}

\begin{proof}[Proof of \eqref{item:rk1-expand}~$\Rightarrow$~\eqref{item:rk1-cc}:]
We again treat the case that $X$ is $\HH^n$, seen as the open unit ball of~$\R^n$ for a Euclidean norm $\Vert\cdot\Vert$, and that the metric $\mathtt{d}_{\di X}$ is induced by $\Vert\cdot\Vert$.
We denote by $\mathtt{d}_{\mathrm{Euc}}$ the Euclidean distance on~$\R^n$ associated to $\Vert\cdot\Vert$.

Suppose that \eqref{item:rk1-expand} holds.
Then $\Lambda_{\rho(\Gamma_0)}$ contains at least two points.
(Indeed, by assumption $\rho(\Gamma_0)$ is an infinite discrete subgroup of~$G$, hence $\Lambda_{\rho(\Gamma_0)}$ is nonempty; moreover, the expansion assumption prevents $\Lambda_{\rho(\Gamma_0)}$ from being a singleton, as follows \eg from the classification of elementary discrete subgroups of~$G$: see \cite[Prop.\,3.2.1]{bow95}.)
Therefore $\Ccore_{\rho(\Gamma_0)}$ is nonempty.
Moreover, one can check (\eg using the Cartan decomposition as in the proof of \eqref{item:rk1-cc}~$\Rightarrow$~\eqref{item:rk1-expand} just above) that for any $z\in\Lambda_{\rho(\Gamma_0)}$, there exist a neighbourhood $\mathcal{U}$ of $z$ in~$\R^n$ (rather than just $\di X$) and an element $\gamma\in\Gamma_0$ such that \eqref{eqn:expand} holds for $\mathtt{d}_{\mathrm{Euc}}$ (rather than $\mathtt{d}_{\di X}$).

Suppose by contradiction that the action of $\Gamma_0$ on $\Ccore_{\rho(\Gamma_0)}$ via~$\rho$ is \emph{not} cocompact.
Let $(\varepsilon_m)_{m\in\N}$ be a sequence of positive reals going to~$0$.
For any~$m$, the set $\mathcal{K}_m := \{ x\in\nolinebreak\Ccore_{\rho(\Gamma_0)} \,|\, \mathtt{d}_{\mathrm{Euc}}(x,\Lambda_{\rho(\Gamma_0)}) \geq\nolinebreak \varepsilon_m\}$ is compact, hence there exists a $\rho(\Gamma_0)$-orbit contained in $\Ccore_{\rho(\Gamma_0)}\smallsetminus\mathcal{K}_m$. 
By proper discontinuity of the action on $\Ccore_{\rho(\Gamma_0)}$, the supremum of $\mathtt{d}_{\mathrm{Euc}}(\cdot,\Lambda_{\rho(\Gamma_0)})$ on this orbit is achieved at some point $x_m\in\Ccore_{\rho(\Gamma_0)}$, and by construction we have $0 < \mathtt{d}_{\mathrm{Euc}}(\rho(\gamma)\cdot x_m,\Lambda_{\rho(\Gamma_0)}) \leq \mathtt{d}_{\mathrm{Euc}}(x_m,\Lambda_{\rho(\Gamma_0)}) \leq \varepsilon_m$ for all $\gamma\in\Gamma_0$.
Up to passing to a subsequence, we may assume that $(x_m)_{m\in\N}$ converges to some $z\in\Lambda_{\rho(\Gamma_0)}$.
Consider a neighbourhood $\mathcal{U}$ of $z$ in~$\R^n$~and~an element $\gamma\in\Gamma_0$ such that \eqref{eqn:expand} holds for $\mathtt{d}_{\mathrm{Euc}}$, and let $c>1$ be~the~inf\-imum in \eqref{eqn:expand}.
For any $m\in\N$, there exists $z_m\in\Lambda_{\rho(\Gamma_0)}$ such that $\mathtt{d}_{\mathrm{Euc}}(\rho(\gamma)\cdot\nolinebreak x_m,\Lambda_{\rho(\Gamma_0)}) = \mathtt{d}_{\mathrm{Euc}}(\rho(\gamma)\cdot x_m,\rho(\gamma)\cdot z_m)$.
For large enough~$m$ we have $x_m,z_m\in\mathcal{U}$, and so $\mathtt{d}_{\mathrm{Euc}}(\rho(\gamma)\cdot\nolinebreak x_m,\Lambda_{\rho(\Gamma_0)}) \geq c\, \mathtt{d}_{\mathrm{Euc}}(x_m, z_m) \geq c\, \mathtt{d}_{\mathrm{Euc}}(x_m,\Lambda_{\rho(\Gamma_0)}) \geq c\, \mathtt{d}_{\mathrm{Euc}}(\rho(\gamma)\cdot\nolinebreak x_m,\Lambda_{\rho(\Gamma_0)})>0$.
This is impossible since $c>1$.
\end{proof}

\section{Classes of discrete subgroups in higher real rank} \label{sec:higher-rank}

We have seen in Section~\ref{sec:rank-1} two important classes of discrete subgroups of semisimple Lie groups~$G$ with $\Rrank(G)=1$, namely convex cocompact subgroups and geometrically finite subgroups.
These classes have been much studied, although many interesting questions remain open even in the case of $G = \PO(n,1)$ for $n\geq 4$ (see \eg \cite{kap07}).

We now turn to infinite discrete subgroups of semisimple Lie groups $G$ for $\Rrank(G)\geq 2$.
These discrete subgroups, beyond lattices, remain more mysterious, and very few general results are known (see \cite{fg23} for a notable exception).
Recently, an important class has emerged, namely the class of \emph{Anosov subgroups}, which are by definition the images of the \emph{Anosov representations} of Gromov hyperbolic groups introduced by Labourie \cite{lab06} as part of his study of Hitchin representations (see Section~\ref{subsec:deform-Fuchsian-higher-rank}).
In fact, most examples in Section~\ref{sec:flex-examples} are Anosov subgroups.
We now discuss these subgroups, make the link with convex cocompactness, and mention some generalisations.

\subsection{Anosov subgroups} \label{subsec:Anosov}

Given a noncompact semisimple Lie group~$G$, there are several possible types of Anosov subgroups of~$G$, depending on the choice of one of the (finitely many) \emph{flag varieties} $G/P$ of~$G$, where $P$ is a proper parabolic subgroup of~$G$.
For simplicity, in these notes we consider $G = \PGL(d,\mathbb{K})$ or $\SL^{\pm}(d,\mathbb{K}) = \{ g\in\GL(d,\mathbb{K}) \,|\, \det(g) = \pm 1\}$ where $\mathbb{K}=\R$ or~$\C$; we take $P = P_i$ to be the stabiliser in~$G$ of an $i$-plane of~$\mathbb{K}^d$, for some $1\leq i\linebreak\leq d-1$, so that $G/P_i = \Gr_i(\mathbb{K}^d)$ is the Grassmannian of $i$-planes of~$\mathbb{K}^d$.

\subsubsection{Definition and first observations}

Here is the original definition from Labourie, which appeared in \cite{lab06} for surface groups $\Gamma_0 = \pi_1(S)$ and in \cite{gw12} for general hyperbolic groups.

\begin{defn} \label{def:Ano}
Let $\Gamma_0$ be an infinite Gromov hyperbolic group and $G = \PGL(d,\mathbb{K})$ or $\SL^{\pm}(d,\mathbb{K})$.
For $1\leq i\leq d-1$, a representation $\rho : \Gamma_0\to G$ is \emph{$P_i$-Anosov} if there exist $\rho$-equivariant maps $\xi_i : \di\Gamma_0\to G/P_i = \Gr_i(\mathbb{K}^d)$ and $\xi_{d-i} : \di\Gamma_0\to G/P_{d-i} = \Gr_{d-i} (\mathbb{K}^d)$ which
\begin{itemize}
  \item are continuous,
  \item are \emph{transverse}: $\xi_i(w) \oplus \xi_{d-i}(w') = \mathbb{K}^d$ for all $w\neq w'$ in $\di\Gamma_0$;
  \item satisfy a uniform contraction property (Condition~\ref{cond:Ano} below) which strengthens the dynamics-preserving condition of Theorem~\ref{thm:charact-cc-rank-1}.\eqref{item:rk1-bound-map-dyn-preserv}.
\end{itemize}
\end{defn}

By \emph{the dynamics-preserving condition of Theorem~\ref{thm:charact-cc-rank-1}.\eqref{item:rk1-bound-map-dyn-preserv}} for~$\xi_i$ we mean that for any infinite-order element $\gamma\in\Gamma_0$, the image by~$\xi_i$ of the attracting fixed point of $\gamma$ in $\di\Gamma_0$ (see Section~\ref{subsec:charact-cc-rank-1}) is an attracting fixed point of $\rho(\gamma)$ in $\Gr_i(\mathbb{K}^d)$.

We note that for an element $g\in G$, the property of admitting an attracting fixed point in $G/P_i$ can be characterised in terms of eigenvalues, namely as $(\lambda_i - \lambda_{i+1})(g) > 0$ (Notation~\ref{not:lambda-mu}).
In this case the attracting fixed point is unique and we say that $g$ is \emph{proximal} in $\Gr_i(\mathbb{K}^d)$.

\begin{rem} \label{rem:PGL-SL}
For our purposes, working with $\PGL(d,\mathbb{K})$ or $\SL^{\pm}(d,\mathbb{K})$ is equivalent.
Indeed, a representation $\rho : \Gamma_0\to\SL^{\pm}(d,\mathbb{K})$ is $P_i$-Anosov if and only if its composition with the natural projection $\SL^{\pm}(d,\mathbb{K}) \to \PGL(d,\mathbb{K})$ is $P_i$-Anosov, and up to passing to a finite-index subgroup (which does not change the property of being $P_i$-Anosov) any representation $\rho : \Gamma_0\to\PGL(d,\mathbb{K})$ with $\Gamma_0$ Gromov hyperbolic lifts to $\SL^{\pm}(d,\mathbb{K})$.
\end{rem}

The uniform contraction property in Definition~\ref{def:Ano} is reminiscent of the condition defining Anosov flows in dynamics, which explains the terminology \emph{Anosov representation}.
Before stating it (Condition~\ref{cond:Ano}), let us make a few elementary observations that already follow from the fact that $\xi_i$ and $\xi_{d-i}$ are continuous, transverse, and dynamics-preserving.

\begin{lem} \label{lem:basic-Ano}
If $\rho : \Gamma\to G$ is $P_i$-Anosov, then
\begin{enumerate}[(1)]
  \item\label{item:basic-Ano-1} the boundary maps $\xi_i$ and $\xi_{d-i}$ are unique, and \emph{compatible}:\linebreak $\xi_{\min(i,d-i)}(w) \subset \xi_{\max(i,d-i)}(w)$ for all $w\in\di\Gamma_0$; the image of~$\xi_i$ is the \emph{proximal limit set} of $\rho(\Gamma_0)$ in~$\Gr_i(\mathbb{K}^d)$, \ie the closure in~$\Gr_i(\mathbb{K}^d)$ of the set of attracting fixed points of proximal elements of $\rho(\Gamma_0)$;
  \item\label{item:basic-Ano-2} $\xi_i$ and $\xi_{d-i}$ are injective, hence they are homeomorphisms onto their images;
  \item\label{item:basic-Ano-3} $\rho$ has finite kernel and discrete image.
\end{enumerate}
\end{lem}

By~\eqref{item:basic-Ano-3}, the images of $P_i$-Anosov representations are infinite discrete subgroups of~$G$; we shall call them \emph{$P_i$-Anosov subgroups}.

\begin{proof}
\eqref{item:basic-Ano-1} Recall from Section~\ref{subsec:charact-cc-rank-1} that the subset of $\di\Gamma_0$ consisting of the attracting fixed points of infinite-order elements of~$\Gamma_0$ is dense in $\di\Gamma_0$.
Since $\xi_i$ and~$\xi_{d-i}$ are dynamics-preserving, they are uniquely determined on this subset, and compatible on this subset.
By continuity, they are uniquely determined and compatible on all of $\di\Gamma_0$.
Moreover, the image of~$\xi_i$ is the proximal limit set of $\rho(\Gamma_0)$ in $\Gr_i(\mathbb{K}^d)$.

\eqref{item:basic-Ano-2} For any $w\neq w'$ in $\di\Gamma_0$, the subspaces $\xi_i(w)$ and $\xi_{d-i}(w')$ are transverse by definition, whereas $\xi_i(w)$ and $\xi_{d-i}(w)$ are not by \eqref{item:basic-Ano-1} above.

\eqref{item:basic-Ano-3} Suppose $\Gamma_0$ is nonelementary.
In order to show that $\rho$ has finite kernel and discrete image, it is sufficient to consider an arbitrary sequence $(\gamma_k)_{k\in\N}$ of pairwise distinct points of~$\Gamma_0$ and to check that $(\rho(\gamma_k))_{k\in\N}$ does not converge to the identity of~$G$.
Recall from Section~\ref{subsec:charact-cc-rank-1} that the action of $\Gamma_0$ on $\di\Gamma_0$ is a convergence action.
Therefore, up to passing to a subsequence, there exist $w^+,w^-\in\di\Gamma_0$ such that $\gamma_k\cdot w\to w^+$ for all $w\in\di\Gamma_0\smallsetminus\{w^-\}$.
By $\rho$-equivariance and continuity of~$\xi_i$, we then have $\rho(\gamma_k)\cdot\xi_i(w) = \xi_i(\gamma_k\cdot w) \to \xi_i(w^+)$ for all $w\in\di\Gamma_0\smallsetminus\{w^-\}$.
Since $\di\Gamma_0$ is infinite and $\xi_i$ is injective, there exists $w\in\di\Gamma_0\smallsetminus\{w^-\}$ such that $\xi_i(w) \neq \xi_i(w^+)$.
The convergence $\rho(\gamma_k)\cdot\xi_i(w) \to \xi_i(w^+)$ then implies that $(\rho(\gamma_k))_{n\in\N}$ does not converge to the identity element of~$G$.
This shows that $\rho$ has finite kernel and discrete image.

If $\Gamma_0$ is elementary, then it admits a finite-index subgroup $\Gamma'_0$ which is cyclic.
The fact that $\xi_i$ is dynamics-preserving implies that $\rho$ is injective and discrete in restriction to~$\Gamma'_0$.
From this one easily deduces that $\rho$ has finite kernel and discrete image.
\end{proof}

\subsubsection{The uniform contraction condition}

Let us state this condition in the original case considered by Labourie \cite{lab06}, where $\Gamma_0 = \pi_1(M)$ for some closed negatively curved manifold~$M$.
We denote by $\widetilde{M}$ the universal cover of~$M$, by $T^1$ the unit tangent bundle, and by $(\varphi_t)_{t\in\R}$ the geodesic flow on either $T^1(M)$ or $T^1(\widetilde{M})$.
(For a general Gromov hyperbolic group~$\Gamma_0$, one should replace $T^1(\widetilde{M})$ by a certain \emph{flow space} for~$\Gamma_0$, see \cite{gw12} or \cite[\S\,4.1]{bps19}.)

For simplicity, we take $G = \SL^{\pm}(d,\mathbb{K})$ (see Remark~\ref{rem:PGL-SL}).
Any representation $\rho : \Gamma_0\to G$ then determines a flat vector bundle
$$E^{\rho} = \Gamma_0\backslash (T^1(\widetilde{M})\times\mathbb{K}^d)$$
over $T^1(M)=\Gamma_0\backslash T^1(\widetilde{M})$, where $\Gamma_0$ acts on $T^1(\widetilde{M})\times\mathbb{K}^d$ by $\gamma\cdot (\tilde{x},v) = (\gamma\cdot\tilde{x}, \rho(\gamma)\cdot v)$.
The geodesic flow $(\varphi_t)_{t\in\R}$ on $T^1(M)$ lifts to a flow $(\psi_t)_{t\in\R}$ on $E^{\rho}$, given by $\psi_t\cdot [(\tilde{x},v)] = [(\varphi_t\cdot\tilde{x},v)]$.

Suppose, as in Definition~\ref{def:Ano}, that there exist continuous, transverse, $\rho$-equivariant boundary maps $\xi_i : \di\Gamma_0\to\Gr_i(\mathbb{K}^d)$ and $\xi_{d-i} : \di\Gamma_0\to\Gr_{d-i} (\mathbb{K}^d)$.
By transversality, for each $\tilde{x}\in T^1(\widetilde{M})$ we have a decomposition $\mathbb{K}^d = \xi_i(\tilde{x}^+) \oplus \xi_{d-i}(\tilde{x}^-)$, where $\tilde{x}^{\pm} = \lim_{t\to\pm\infty} \varphi_t\cdot\tilde{x} \in \di\widetilde{M} \simeq \di\Gamma_0$ are the forward and backward endpoints of the geodesic determined by~$\tilde{x}$, and this defines a decomposition of the vector bundle $E^{\rho}$ into the direct sum of two subbundles $E_i^{\rho} = \{ [(\tilde{x},v)] \,|\, v\in\xi_i(\tilde{x}^+)\}$ and $E_{d-i}^{\rho} = \{ [(\tilde{x},v)] \,|\, v\in\nolinebreak\xi_{d-i}(\tilde{x}^-)\}$.
This decomposition is invariant under the flow $(\psi_t)$.
By definition, the representation $\rho$ is $P_i$-Anosov if the following ``dominated splitting'' condition is satisfied.

\begin{condition} \label{cond:Ano}
The flow $(\psi_t)_{t\in\R}$ uniformly contracts $E_i^{\rho}$ with respect to $E_{d-i}^{\rho}$, \ie given a continuous family $(\Vert\cdot\Vert_x)_{x\in T^1(M)}$ of norms on the fibers $E^{\rho}(x)$, there exist $C,C'>0$ such that for any $t\geq 0$, any $x\in T^1(M)$, and any nonzero $\upsilon_i\in E_i^{\rho}(x)$ and $\upsilon_{d-i}\in E_{d-i}^{\rho}(x)$,
$$\frac{\Vert\psi_t\cdot \upsilon_i\Vert_{\varphi_t\cdot x}}{\Vert\psi_t\cdot \upsilon_{d-i}\Vert_{\varphi_t\cdot x}} \leq e^{-Ct+C'} \, \frac{\Vert \upsilon_i\Vert_x}{\Vert \upsilon_{d-i}\Vert_x},$$
\end{condition}

By compactness of $T^1(M)$, this condition does not depend on the choice of continuous family of norms $(\Vert\cdot\Vert_x)_{x\in T^1(M)}$ (changing the norms may change the values of $C,C'$ but not their existence).

\begin{rem} \label{rem:Ano-Zariski-dense}
Guichard and Wienhard \cite{gw12} showed that if there exist $\rho$-equivariant maps $\xi_i$ and~$\xi_{d-i}$ which are continuous, transverse, and dynamics-preserving, and if the group $\rho(\Gamma_0)$ is \emph{Zariski-dense} in~$G$, then Condition~\ref{cond:Ano} is automatically satisfied.
\end{rem}

\subsubsection{Properties}

\begin{itemize}
  \item $P_i$-Anosov is equivalent to $P_{d-i}$-Anosov, as the integers $i$ and $d-i$ play a similar role in Definition~\ref{def:Ano} and Condition~\ref{cond:Ano} (up to reversing the flow, which switches contraction and expansion).
  In particular, we may restrict to $P_i$-Anosov for $1\leq i\leq d/2$.
  \item When $\Rrank(G) = 1$ (\ie $d=2$ for $G=\PGL(d,\mathbb{K})$ or $\SL^{\pm}(d,\mathbb{K})$), there is only one proper parabolic subgroup $P$ of~$G$ up to conjugation (see Remark~\ref{rem:bound-G/K-rank-1}), hence only one notion of Anosov.
  In that case, an infinite discrete subgroup of~$G$ is Anosov if and only if it is convex cocompact in the classical sense of Definition~\ref{def:cc-gf-rank-1}.
  \item When $\Rrank(G)\geq 2$ (\ie $d\geq 3$ for $G=\PGL(d,\mathbb{K})$ or $\SL^{\pm}(d,\mathbb{K})$), Anosov subgroups are \emph{not} lattices of~$G$ (since Anosov subgroups are Gromov hyperbolic unlike lattices, see Section~\ref{subsec:rank}).
  \item Uniform contraction over a compact space as in Condition~\ref{cond:Ano} is stable under small deformations, which implies the following analogue of Fact~\ref{fact:cc-open-rank-1}.
  \begin{fact} \label{fact:Ano-open}
  Let $G$ be a noncompact semisimple Lie group and $P$ a proper parabolic subgroup of~$G$.
  For any infinite Gromov hyperbolic group~$\Gamma_0$, the space of $P$-Anosov representations is open in $\Hom(\Gamma_0,G)$.
  \end{fact}
  \item Anosov representations behave well under Lie group homomorphisms: the following holds similarly to Fact~\ref{fact:cc-embed-rank-1}.
  (We refer to Remark~\ref{rem:charact-cc-rank-1} for the notion of a hyperbolic element of~$G'$.)
  \begin{fact}[{see \cite{gw12}}] \label{fact:cc-embed-Ano}
  Let $G'$ be a semisimple Lie group with $\Rrank(G')\linebreak =1$ and let $\tau : G'\to\PGL(d,\mathbb{K})$ be a Lie group homomorphism with compact kernel.
  For any Gromov hyperbolic group~$\Gamma_0$, any representation $\sigma_0 : \Gamma_0\to G'$, and any $1\leq i\leq d-1$, the following are equivalent:
  \begin{enumerate}[(1)]
    \item the representation $\tau\circ\sigma_0 : \Gamma_0\to\PGL(d,\mathbb{K})$ is $P_i$-Anosov;
    \item $\sigma_0$ is convex cocompact (Definition~\ref{def:cc-gf-rank-1}) and $(\lambda_i - \lambda_{i+1})(\tau(g')) > 0$ for some hyperbolic element $g'\in G'$.
  \end{enumerate}
  \end{fact}
  \noindent
  In this case, $\tau$ induces an embedding $\di\tau_i : G'/P'\hookrightarrow\Gr_i(\mathbb{K}^d)$ (where $G'/P'$ is the visual boundary of the symmetric space of~$G'$, see Remark~\ref{rem:bound-G/K-rank-1}) and the boundary map of~$\rho_0$ is the composition of the boundary map $\di\Gamma_0\to G'/P'$ of~$\sigma_0$ (see Theorem~\ref{thm:charact-cc-rank-1}) with~$\di\tau_i$.
  Moreover, by Fact~\ref{fact:Ano-open} there is in that case a neighbourhood\linebreak of $\tau\circ\sigma_0$ in $\Hom(\Gamma_0,\PGL(d,\mathbb{K}))$ consisting entirely of $P_i$-Anosov representations (hence with finite kernel and discrete image --- see\linebreak Lemma~\ref{lem:basic-Ano}.\eqref{item:basic-Ano-3}).
\end{itemize}

\subsubsection{Examples in higher real rank}

Many of the discrete subgroups in Section~\ref{sec:flex-examples} were Anosov subgroups.

\begin{itemize}
  \item Section~\ref{subsec:ping-pong-higher-rank}: It follows from the work of Benoist \cite{ben97} that the ping pong groups of Claim~\ref{claim:ping-pong-proj} are quasi-isometrically embedded (see Remark~\ref{rem:div-in-G/K}) in $\PGL(d,\R)$.
  They are in fact $P_1$-Anosov: see \cite{clss17,klp14}.
  \item Section~\ref{subsec:ping-pong-higher-rank}: When they are defined by $B_1^{\pm},\dots,B_m^{\pm}$ which have pairwise disjoint closures, the Schottky groups in $\PGL(2n,\mathbb{K})$ of Nori and Seade--Verjovsky are $P_n$-Anosov (see \cite{gw12}) and the crooked Schottky groups in $\Sp(2n,\R) \subset \SL(2n,\R)$ are $P_1$-Anosov (see \cite{bk-Schottky}).
  \item Section~\ref{subsec:deform-Fuchsian-higher-rank}: By Facts \ref{fact:Ano-open} and~\ref{fact:cc-embed-Ano}, the Barbot representations of closed surface groups into $\SL(d,\R)$ are $P_1$-Anosov.
  The Hitchin representations into $\PSL(d,\R)$ are $P_i$-Anosov for all $1\leq i\leq d-1$: this is Labourie's original result from \cite{lab06}, where he introduced Anosov representations.
  The maximal representations of closed surface groups into $\SO(2,n) \subset \SL(n+2,\R)$ are $P_1$-Anosov (see \cite{bilw05,glw}).
  We refer to Figures \ref{fig:Barbot} and~\ref{fig:Hitchin-bound-map} for some illustrations of boundary maps.
  
  \begin{rem}
  Being $P_i$-Anosov for all $1\leq i\leq d-1$ is the strongest possible form of Anosov; in this case, the various boundary maps $\xi_i : \di\Gamma_0\to\Gr_i(\mathbb{K}^d)$ for $1\leq i\leq d-1$ combine into a continuous, injective, $\rho$-equivariant boundary map $\xi : \di\Gamma_0\to\mathrm{Flags}(\R^d)$ as in the proof of Theorem~\ref{thm:Hitchin-comp}.
  \end{rem}
  
  \begin{figure}[ht!]
\vspace{-0.3cm}
\includegraphics[scale=0.3]{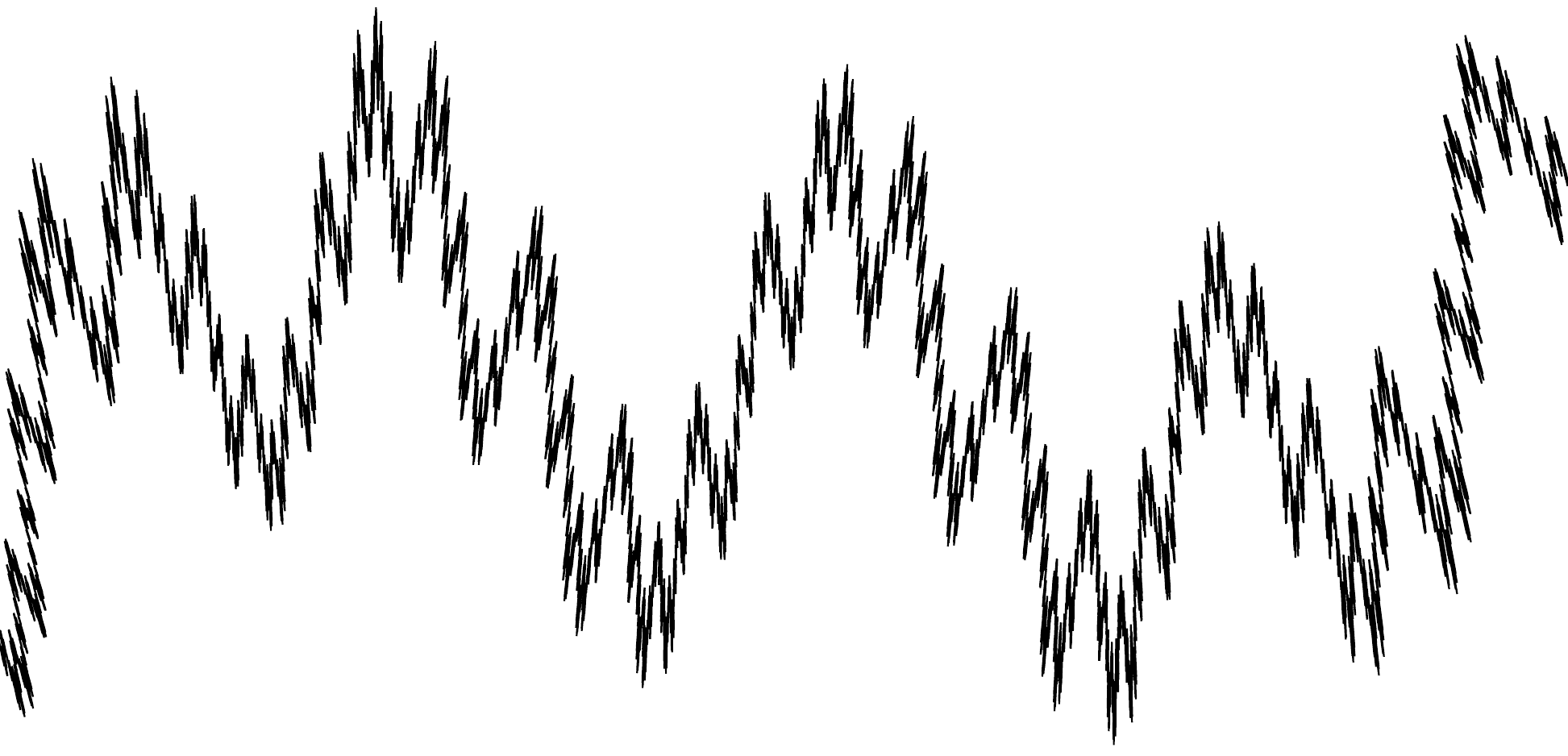}
\vspace{-0.2cm}
\caption{The image of the boundary map $\xi_1 : \di\Gamma_0\to\PP(\R^3)$ of a representation $\rho : \Gamma_0 = \pi_1(S) \to \SL(3,\R)$ which is a small deformation of $\Gamma_0 \overset{\sigma_0}{\longhookrightarrow} \SL(2,\R) \overset{\tau}{\longhookrightarrow} \SL(3,\R)$, where $\sigma_0$ is injective and discrete and $\tau$ is the standard representation. This image is a topological circle in $\PP(\R^3)$ which has H\"older, but not Lipschitz, regularity.}
\label{fig:Barbot}
\end{figure}

\begin{figure}[ht!]
\begin{overpic}[scale=0.56,percent]{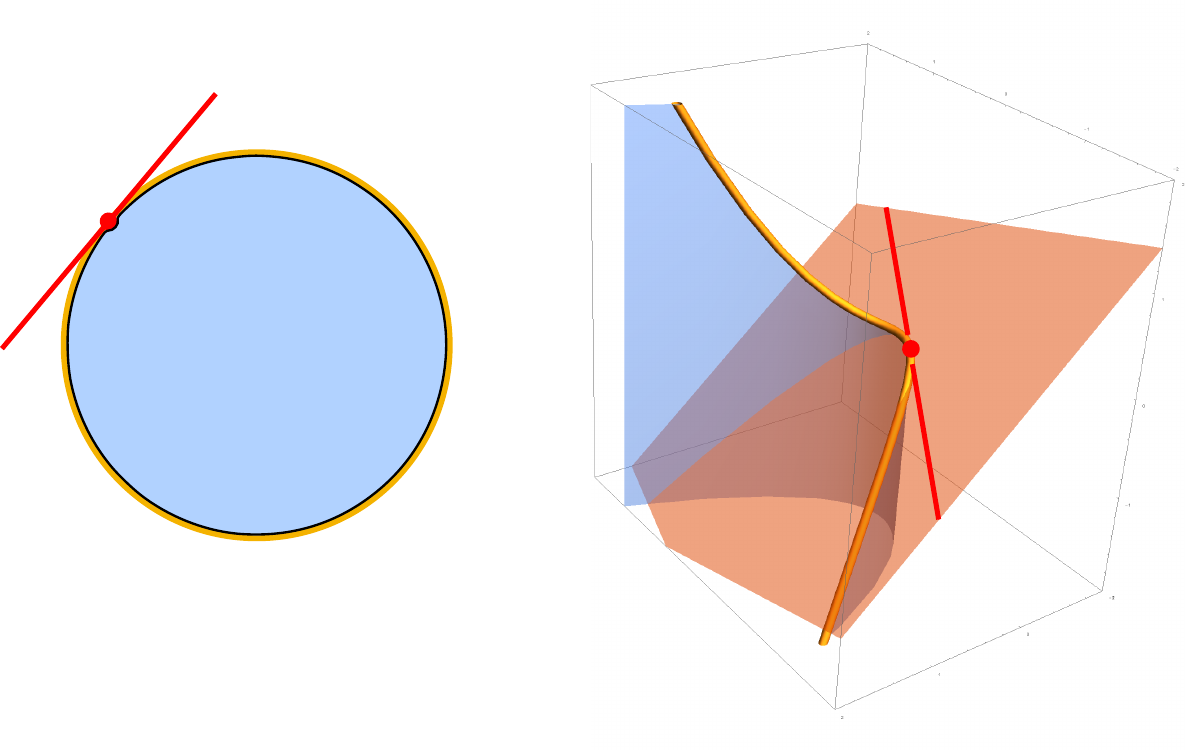}
\put(1,44){\textcolor{red}{$\xi_1(w)$}}
\put(16,57){\textcolor{red}{$\xi_2(w)$}}
\put(78,33){\textcolor{red}{$\xi_1(w)$}}
\put(76,43){\textcolor{red}{$\xi_2(w)$}}
\put(89,40){\textcolor{red}{$\xi_3(w)$}}
\end{overpic}
\vspace{-0.8cm}
\caption{If $\xi = (\xi_1,\dots,\xi_{d-1}) : \di\Gamma_0 \to \mathrm{Flags}(\R^d)$ is the boundary map of a Hitchin representation $\rho : \Gamma_0 = \pi_1(S)\to\PSL(d,\R)$, then the image of~$\xi_1$ is a $C^1$ curve in $\PP(\R^d)$, and $\xi(w)$ is the osculating flag to this curve at the point $\xi_1(w)$ for all $w\in\di\Gamma_0$.
For $d=3$ the curve is the boundary of the properly convex open subset of $\PP(\R^3)$ preserved by~$\rho$, while for $d=4$ the curve is homotopically nontrivial in $\PP(\R^4)$.
This figure shows the curve $\xi_1(\di\Gamma_0)$ and an osculating flag $\xi(w)$ when $\rho : \Gamma_0\to\PSL(2,\R)\hookrightarrow\PSL(d,\R)$ is Fuchsian, for $d=3$ (left) and $d=4$ (right); for $d=4$, the curve is the so-called \emph{twisted cubic} in $\PP(\R^4)$, given by $t\mapsto (t,t^2,t^3)$ in some affine chart.}
\label{fig:Hitchin-bound-map}
\end{figure}
    
  \item Section~\ref{subsec:higher-Teich}: All (known) higher Teichm\"uller spaces consist of Anosov representations (see \cite{bp,bilw05,glw,lab06}).
   
  \begin{rem}
  Not all Anosov representations of closed surface\linebreak groups belong to higher Teichm\"uller spaces.
  For instance, the Barbot representations of $\pi_1(S)$ into $\SL(d,\R)$ from Section~\ref{subsec:deform-Fuchsian-higher-rank} are $P_1$-Anosov, but their connected component in $\Hom(\pi_1(S),\SL(d,\R))$ contains representations that are not injective and discrete.
  \end{rem}
  
  \item Section~\ref{subsec:higher-Teich}: The two known families of higher-dimensional higher-rank Teichm\"uller spaces that we mentioned for Gromov hyperbolic groups $\Gamma_0 = \pi_1(M)$ are $P_1$-Anosov: for holonomies of convex projective structures, see \cite{ben04}, and for $\HH^{p,q}$-convex cocompact representations, see \cite{bm12} (case $q=1$) and \cite{dgk18,dgk-proj-cc} (general case).
\end{itemize}

\subsubsection{Interlude: eigenvalues and singular values}

Before giving (in Theorem~\ref{thm:charact-Ano} below) some characterisations of Anosov representations that generalise Theorem~\ref{thm:charact-cc-rank-1}, we introduce some notation and make a few preliminary observations.

\begin{notation} \label{not:lambda-mu}
For any $g\in\GL(d,\C)$, we denote by $\lambda_1(g) \geq \dots \geq \lambda_d(g)$ the logarithms of the moduli of the complex eigenvalues of~$g$, and by $\mu_1(g) \geq \dots \geq \mu_d(g)$ the logarithms of the \emph{singular values} of~$g$ (\ie of the square roots of the eigenvalues of $\overline{g}^T g$, which are positive numbers).
For any $1\leq i<j\leq d$, this defines functions $\lambda_i - \lambda_j : \GL(d,\C)\to\R_{\geq 0}$ and $\mu_i - \mu_j : \GL(d,\C)\to\R_{\geq 0}$ which factor through $\PGL(d,\C)$.
\end{notation}

As in Section~\ref{subsec:charact-cc-rank-1}, for any finitely generated group~$\Gamma_0$, we choose a finite generating subset of~$\Gamma_0$ and denote by $\mathrm{Cay}(\Gamma_0)$ the corresponding Cayley graph, with its metric $\mathtt{d}_{\mathrm{Cay}(\Gamma_0)}$.
We denote by $X = G/K$ the Riemannian symmetric space of~$G$, with its metric $\mathtt{d}_X$, and fix a basepoint $x_0\in X$.
We denote the translation length as in \eqref{eqn:transl-length}.

Our starting point is the following (see Theorem~\ref{thm:charact-cc-rank-1} and its proof).

\begin{rem} \label{rem:div-in-G/K}
Let $\Gamma_0$ be a finitely generated group and $\rho : \Gamma_0\to G$ a representation.
Then
\begin{itemize}
  \item $\rho$ has finite kernel and discrete image if and only if $\mathtt{d}_X(x_0,\rho(\gamma)\cdot x_0) \to +\infty$ as $\mathtt{d}_{\mathrm{Cay}(\Gamma_0)}(e,\gamma)\to +\infty$;
  \item $\rho$ is called a \emph{quasi-isometric embedding} if there exist $c,c'>0$ such that $\mathtt{d}_X(x_0,\rho(\gamma)\cdot x_0) \geq c\,\mathtt{d}_{\mathrm{Cay}(\Gamma_0)}(e,\gamma) - c'$ for all $\gamma\in\Gamma$;
  \item $\rho$ is called \emph{well-displacing} if there exist $c,c''>0$ such that\linebreak $\mathrm{transl}_X(\rho(\gamma)) \geq c\,\mathrm{transl}_{\mathrm{Cay}(\Gamma_0)}(\gamma) - c''$ for all $\gamma\in\Gamma$.
\end{itemize}
\end{rem}

(As in Remarks~\ref{rem:charact-cc-rank-1}, the inequality $\mathtt{d}_X(x_0,\rho(\gamma)\cdot x_0) \leq C\,\mathtt{d}_{\mathrm{Cay}(\Gamma_0)}(e,\gamma)$ always holds for $C := \max_{f\in F} \mathtt{d}_X(x_0,\rho(f)\cdot x_0)$; it implies that the inequality $\mathrm{transl}_X(\rho(\gamma)) \leq C\,\mathrm{transl}_{\mathrm{Cay}(\Gamma_0)}(\gamma)$ always holds too: see the proof of the implication \eqref{item:rk1-qi}~$\Rightarrow$~\eqref{item:rk1-displac} of Theorem~\ref{thm:charact-cc-rank-1}.)

We now reinterpret Remark~\ref{rem:div-in-G/K} using Notation~\ref{not:lambda-mu}.
Let $\Vert\cdot\Vert_{\mathrm{Euc}}$ be the standard Euclidean norm on~$\R^d$.
For $G = \PGL(d,\mathbb{K})$ with $\mathbb{K} = \R$ or~$\C$, we can take $K = \PO(d)$ or $\mathrm{PU}(d)$ and $x_0 = eK \in G/K = X$, so that for any $g\in G$ lifting to $\hat{g}\in\GL(d,\mathbb{K})$ with $|\det(\hat{g})|=1$,
\begin{align*}
\mathtt{d}_X(x_0,g\cdot x_0) & = \Vert(\mu_1(\hat{g}),\dots,\mu_d(\hat{g}))\Vert_{\mathrm{Euc}},\\
\mathrm{transl}_X(g) & = \Vert(\lambda_1(\hat{g}),\dots,\lambda_d(\hat{g}))\Vert_{\mathrm{Euc}}.
\end{align*}
On the other hand, we have $\sum_{i=1}^d \mu_i(\hat{g}) = \sum_{i=1}^d \lambda_i(\hat{g}) = 0$, and on the linear hyperplane $\{ v\in\R^d \,|\, \sum_{i=1}^d v_i = 0\}$ of~$\R^d$ the Euclidean norm $\Vert\cdot\Vert_{\mathrm{Euc}}$ is equivalent to $\sum_{i=1}^{d-1} |v_i - v_{i+1}|$.
In this setting we can therefore rewrite Remark~\ref{rem:div-in-G/K} as follows.

\begin{rem} \label{rem:div-mu-nu}
Let $\Gamma_0$ be a finitely generated group and $\rho : \Gamma_0\to G = \PGL(d,\mathbb{K})$ a representation.
Then
\begin{itemize}
  \item $\rho$ has finite kernel and discrete image if and only if\linebreak $\sum_{i=1}^{d-1} (\mu_i - \mu_{i+1})(\rho(\gamma)) \to +\infty$ as $\mathtt{d}_{\mathrm{Cay}(\Gamma_0)}(e,\gamma)\to +\infty$;
  \item $\rho$ is a quasi-isometric embedding if and only if there exist $c,c'>0$ such that $\sum_{i=1}^{d-1} (\mu_i - \mu_{i+1})(\rho(\gamma)) \geq c\,\mathtt{d}_{\mathrm{Cay}(\Gamma_0)}(e,\gamma) - c'$ for all $\gamma\in\Gamma$;
  \item $\rho$ is well-displacing if and only if there exist $c,c''>0$ such that $\sum_{i=1}^{d-1} (\lambda_i - \lambda_{i+1})(\rho(\gamma)) \geq c\,\mathrm{transl}_{\mathrm{Cay}(\Gamma_0)}(\gamma) - c''$ for all $\gamma\in\Gamma$.
\end{itemize}
\end{rem}

Remark~\ref{rem:div-mu-nu} should be kept in mind will reading Theorem~\ref{thm:charact-Ano}.\eqref{item:qi-i-direction}--\eqref{item:displac-i-direction} below, as it explains how Anosov representations are refinements of quasi-isometric embeddings and well-displacing representations.

\subsubsection{Characterisations}

The following characterisations of Anosov representations were established by Kapovich--Leeb--Porti, Gu\'eritaud--Guichard--Kassel--Wienhard, Bochi--Potrie--Sambarino, and Kassel--Potrie.
More precisely, \eqref{item:Ano}~$\Rightarrow$~\eqref{item:qi-i-direction} is easy and follows from a property of dominated splittings proved in \cite{bg09}.
The implication \eqref{item:qi-i-direction}~$\Rightarrow$~\eqref{item:Ano} was proved in \cite{klp-morse}, with an alternative proof later given in \cite{bps19}.
The implication \eqref{item:qi-i-direction}~$\Rightarrow$~\eqref{item:displac-i-direction} is easy and similar to the implication \eqref{item:rk1-qi}~$\Rightarrow$~\eqref{item:rk1-displac} of Theorem~\ref{thm:charact-cc-rank-1} (note that $\lambda_i(\hat{g}) = \lim_k \mu_i({\hat{g}}^k)/k$ for all $\hat{g}\in\GL(d,\mathbb{K})$).
The implication \eqref{item:displac-i-direction}~$\Rightarrow$~\eqref{item:qi-i-direction} was proved in \cite{kp22}.
The implications \eqref{item:Ano}~$\Leftrightarrow$~\eqref{item:i-bound-map-dyn-preserv}~$\Leftrightarrow$~\eqref{item:i-bound-map-strong-dyn-preserv} and \eqref{item:Ano}~$\Rightarrow$~\eqref{item:i-expand} were proved in \cite{ggkw17} and \cite{klp14}, and \eqref{item:i-expand}~$\Rightarrow$~\eqref{item:Ano} was proved in \cite{klp14}.
We refer to \cite{kl-msri,klp-survey} for further characterisations (\eg in terms of \emph{conical limit points}).

We fix a basepoint $x_0\in X = G/K$ and a Riemannian metric on $G/P_{i,d-i} = \mathrm{Flags}_{i,d-i}(\R^d) = \{(V_{\min(i,d-i)}\subset V_{\max(i,d-i)}) \,|\, \dim(V_{\bullet}) =\nolinebreak\bullet\}$.
See Definition~\ref{def:Ano} for the notions of transversality and dynamics-preser\-ving, and Theorem~\ref{thm:charact-cc-rank-1}.\eqref{item:rk1-expand} for expansion at the limit set.
The notion of limit set that we use is discussed in the next section.

\begin{thm} \label{thm:charact-Ano}
Let $G = \PGL(d,\mathbb{K})$ or $\SL^{\pm}(d,\mathbb{K})$ where $\mathbb{K} = \R$ or~$\C$, and let $1\leq i\leq d-1$.
Let $\Gamma_0$ be a finitely generated infinite group and $\rho : \Gamma_0\to G$ a representation.
Then the following are equivalent:
\begin{enumerate}[(1)]
  \item\label{item:Ano} $\Gamma_0$ is Gromov hyperbolic and $\rho$ is $P_i$-Anosov,
  \item\label{item:qi-i-direction} $\rho$ is a quasi-isometric embedding ``in the $i$-th direction'': there exist $c,c'>0$ such that for any $\gamma\in\Gamma$,
  $$(\mu_i - \mu_{i+1})(\rho(\gamma)) \geq c\,\mathtt{d}_{\mathrm{Cay}(\Gamma_0)}(e,\gamma) - c' ;$$
  \item\label{item:displac-i-direction} $\Gamma_0$ is Gromov hyperbolic and $\rho$ is well-displacing ``in the $i$-th direction'': there exist $c,c''>0$ such that for any $\gamma\in\Gamma$,
  $$(\lambda_i - \lambda_{i+1})(\rho(\gamma)) \geq c\,\mathrm{transl}_{\mathrm{Cay}(\Gamma_0)}(\gamma) - c'' ;$$
    \item\label{item:i-bound-map-dyn-preserv} $\Gamma_0$ is Gromov hyperbolic, there exist $\rho$-equivariant maps
    $$\xi_{\bullet} : \di\Gamma_0 \longrightarrow G/P_{\bullet} = \Gr_{\bullet}(\mathbb{K}^d),$$
    for $\bullet\in\{ i,d-i\}$, which are continuous, transverse, dynamics-preser\-ving, and $(\mu_i - \mu_{i+1})(\rho(\gamma)) \to +\infty$ as $\mathtt{d}_{\mathrm{Cay}(\Gamma_0)}(e,\gamma)\to +\infty$;
  \smallskip
  \item\label{item:i-bound-map-strong-dyn-preserv} $\Gamma_0$ is Gromov hyperbolic and there exist $\rho$-equivariant maps
  $$\xi_{\bullet} : \di\Gamma_0 \longrightarrow G/P_{\bullet} = \Gr_{\bullet}(\mathbb{K}^d),$$
    for $\bullet\in\{ i,d-i\}$, which are continuous, transverse, and \emph{strongly dynamics-preserving} (\ie for any $(\gamma_k)\in\Gamma_0^{\N}$ and $w^+,w^-\in\di\Gamma_0$, if $\gamma_k\cdot w\to w^+$ for all $w\in\di\Gamma_0\smallsetminus\{w^-\}$, then $\rho(\gamma_k)\cdot z\to\xi_i(w^+)$ for all $z\in G/P_i$ transverse to $\xi_{d-i}(w^-)\in G/P_{d-i}$);
  \smallskip
  \item\label{item:i-expand} $(\mu_i - \mu_{i+1})(\rho(\gamma)) \to +\infty$ as $\mathtt{d}_{\mathrm{Cay}(\Gamma_0)}(e,\gamma)\to +\infty$, any two points of the limit set of $\rho(\Gamma_0)$ in $G/P_{i,d-i}$ are transverse, and the action of $\Gamma_0$ on $G/P_{i,d-i}$ via~$\rho$ is expanding at this limit set.
\end{enumerate}
\end{thm}

\begin{rem} \label{rem:compare-charact-Ano-cc}
Recall that when $\Rrank(G)=1$ (\ie $d=2$), an infinite discrete subgroup of~$G$ is Anosov if and only if it is convex cocompact in the classical sense of Definition~\ref{def:cc-gf-rank-1}.
In that case, the flag variety $G/P_i = G/P_{d-i}$ identifies with the visual boundary $\di X$ of $X = G/K$ (Remark~\ref{rem:bound-G/K-rank-1}) and conditions \eqref{item:qi-i-direction}, \eqref{item:displac-i-direction}, \eqref{item:i-bound-map-dyn-preserv}, \eqref{item:i-bound-map-strong-dyn-preserv}, \eqref{item:i-expand} of Theorem~\ref{thm:charact-Ano} are the same as conditions \eqref{item:rk1-qi}, \eqref{item:rk1-displac}, \eqref{item:rk1-bound-map-dyn-preserv}, \eqref{item:rk1-bound-map-strong-dyn-preserv}, \eqref{item:rk1-expand} of Theorem~\ref{thm:charact-cc-rank-1} (see Remark~\ref{rem:div-mu-nu} and Lemma~\ref{lem:basic-Ano}.\eqref{item:basic-Ano-3}).
On the other hand, when $\Rrank(G)\geq\nolinebreak 2$, conditions \eqref{item:qi-i-direction} and~\eqref{item:displac-i-direction} of  Theorem~\ref{thm:charact-Ano} are strictly stronger than conditions \eqref{item:rk1-qi} and~\eqref{item:rk1-displac} of Theorem~\ref{thm:charact-cc-rank-1} (see Remark~\ref{rem:div-mu-nu}).
\end{rem}

As in Remarks~\ref{rem:charact-cc-rank-1}, in condition~\eqref{item:displac-i-direction} we cannot remove the assumption that $\Gamma_0$ be Gromov hyperbolic.
See \cite[\S\,4.4]{kp22} for further discussion.

\subsubsection{Limit sets}

We now explain the notion of limit set used in Theorem~\ref{thm:charact-Ano}.\eqref{item:i-expand}.
It is based on an important decomposition of the noncompact semisimple Lie group~$G$: the \emph{Cartan decomposition} $G = K\exp(\mathfrak{a}^+)K$.
We refer to \cite{hel01} for the general theory for noncompact semisimple Lie groups~$G$.
For $G = \PGL(d,\mathbb{K})$, as in Remark~\ref{rem:div-mu-nu}, we can take $K = \PO(d)$ or $\mathrm{PU}(d)$, and $\mathfrak{a}^+$ to be the set of diagonal matrices in $\mathfrak{g} = \{ y\in M_d(\mathbb{K}) \,|\, \mathrm{tr}(y)=0\}$ whose entries $t_1,\dots,t_d\in\R$ are in nonincreasing order, with $t_1+\dots+t_d=0$; the Cartan decomposition can then be stated as follows.

\begin{fact} \label{fact:Cartan-decomp}
Any $g\in\PGL(d,\mathbb{K})$ can be written as $g = \kappa\exp(a)\kappa'$ for some $\kappa,\kappa'\in K$ and a unique $a\in\mathfrak{a}^+$; the entries of~$a$ are $\mu_1(\hat{g}),\dots,\mu_d(\hat{g})$ (see Notation~\ref{not:lambda-mu}) where $\hat{g}\in\GL(d,\mathbb{K})$ is any lift of~$g$ with\linebreak $|\det(\hat{g})|= 1$.
\end{fact}

\begin{proof}
By the polar decomposition, any element of $\GL(d,\R)$ (\resp\linebreak $\GL(d,\C)$) can be written as the product of an orthogonal (\resp unitary) matrix and a positive semi-definite real symmetric (\resp Hermitian) matrix; on the other hand, any real symmetric (\resp Hermitian) matrix can be diagonalised by an orthogonal (\resp unitary) matrix.
\end{proof}

Here is a useful consequence of the Cartan decomposition.

\begin{lem} \label{lem:conv-G/P}
For $1\leq i\leq d-1$ and a sequence $(g_m)$ of points of $G = \PGL(d,\mathbb{K})$, consider the following two conditions:
\begin{enumerate}[(a)]
  \item\label{item:conv-G/P-1} $(\mu_i - \mu_{i+1})(g_m)\to +\infty$,
  \item\label{item:conv-G/P-2} there exist $z^+\in G/P_i$ and $z^-\in G/P_{d-i}$ such that $g_m\cdot z\to z^+$ for all $z\in G/P_i$ transverse to~$z^-$.
\end{enumerate}
If $(g_m)$ satisfies (\ref{item:conv-G/P-1}), then some subsequence of $(g_m)$ satisfies (\ref{item:conv-G/P-2}).
Conversely, if $(g_m)$ satisfies (\ref{item:conv-G/P-2}), then it satisfies (\ref{item:conv-G/P-1}).
\end{lem}

\begin{proof}
Let $z_0^+ := \mathrm{span}(e_1,\dots,e_i) \in G/P_i$ and $z_0^- := \mathrm{span}(e_{i+1},\dots,e_d)\linebreak \in G/P_{d-i}$, where $(e_1,\dots,e_d)$ is the canonical basis of~$\mathbb{K}^d$.
By Fact~\ref{fact:Cartan-decomp}, for any~$m$ we can write $g_m = \kappa_m \exp(a_m) \kappa'_m$ where $\kappa_m,\kappa'_m\in K$ and $a_m\in\mathfrak{a}^+$ is diagonal; the entries of~$a_m$ are $\mu_1(\hat{g}_m),\dots,\mu_d(\hat{g}_m)$ where $\hat{g}_m\in\GL(d,\mathbb{K})$ is any lift of~$g_m$ with $|\det(\hat{g}_m)|=1$.

\eqref{item:conv-G/P-1}~$\Rightarrow$~\eqref{item:conv-G/P-2}: If $(\mu_i - \mu_{i+1})(g_m)\to +\infty$, then $a_m\cdot z\to z_0^+$ for all $z\in G/P_i$ transverse to~$z_0^-$.
Since $K$ is compact, up to passing to a subsequence, we may assume that $(\kappa_m),(\kappa'_m)$ converge respectively to some $\kappa,\kappa'\in K$.
Then $g_m\cdot z\to z^+ := \kappa\cdot z_0^+$ for all $z\in G/P_i$ transverse to $z^- := {\kappa'}^{-1}\cdot z_0^-$.

\eqref{item:conv-G/P-2}~$\Rightarrow$~\eqref{item:conv-G/P-1}: If $(\mu_i - \mu_{i+1})(g_m)$ does not tend to $+\infty$, then up to passing to a subsequence it converges to some nonnegative real number, and one easily sees that the image by $a_m$ of any open subset of $G/P_i$ fails to converge to a point.
Up to passing to a subsequence, we may assume that $(\kappa_m),(\kappa'_m)$ converge in~$K$.
Then the image by $g_m$ of any open subset of $G/P_i$ fails to converge to a point.
\end{proof}

For $(g_m)$ and $z^+$ as in condition~\eqref{item:conv-G/P-2} of Lemma~\ref{lem:conv-G/P}, we say that $z^+$ is a \emph{contraction point} for $(g_m)$ in $G/P_i$.
We then define the \emph{limit set in $G/P_i$} of a discrete subgroup $\Gamma$ of~$G$ to be the set of contraction points in $G/P_i$ of sequences of elements of~$\Gamma$.
It is a closed $\Gamma$-invariant subset of $G/P_i$.
When $\Gamma$ is $P_i$-Anosov, it coincides with the proximal limit set of $\Gamma$ in $G/P_i$, which is also the image of the boundary map $\xi_i : \di\Gamma\to G/P_i$.

Similarly, sequences $(g_m)\in G^{\N}$ satisfying both $(\mu_i - \mu_{i+1})(g_m)\to +\infty$ and $(\mu_{d-i} - \mu_{d-i+1})(g_m)\to +\infty$ define contraction points in $G/P_{i,d-i} = \mathrm{Flags}_{i,d-i}(\R^d)$.
This gives a notion of \emph{limit set in $G/P_{i,d-i}$} of a discrete subgroup $\Gamma$ of~$G$, as considered in Theorem~\ref{thm:charact-Ano}.\eqref{item:i-expand}.

We note that in the setting of Theorem~\ref{thm:charact-Ano}.\eqref{item:i-expand}, the limit set of $\rho(\Gamma_0)$ in $G/P_{i,d-i}$ is nonempty.
Indeed, $(\mu_{d-i} - \mu_{d-i+1})(g) = (\mu_i - \mu_{i+1})(g^{-1})$ for all $g\in G$, and so $(\mu_i - \mu_{i+1})(\rho(\gamma)) \to +\infty$ as $\mathtt{d}_{\mathrm{Cay}(\Gamma_0)}(e,\gamma)\to +\infty$ implies $(\mu_{d-i} - \mu_{d-i+1})(\rho(\gamma)) \to +\infty$ as $\mathtt{d}_{\mathrm{Cay}(\Gamma_0)}(e,\gamma)\to +\infty$.

\subsubsection{Cocompact domains of discontinuity}

We end this section by briefly mentioning a generalisation to Anosov representations of a nice feature of rank-one convex cocompact representations.
Namely, we have seen in Section~\ref{subsec:cc-open-rank-1} that for $\Rrank(G) = 1$, if $X = G/K$ denotes the Riemannian symmetric space of~$G$, then any convex cocompact subgroup $\Gamma$ of~$G$ acts properly discontinuously, with compact quotient, on the open subset $\Omega_{\Gamma} := \di X \smallsetminus \Lambda_{\Gamma}$ of $\di X$.
In that case, $\di X$ is the unique flag variety $G/P$ of~$G$ with $P$ a proper parabolic subgroup of~$G$ (Remark~\ref{rem:bound-G/K-rank-1}).

Guichard and Wienhard \cite{gw12}, inspired by work of Frances, generalised this picture to show that in certain situations, for certain proper parabolic subgroups $P$ and $Q$ of~$G$, any $P$-Anosov subgroup $\Gamma$ of~$G$ acts properly discontinuously, with compact quotient, on some open subset $\Omega$ of $G/Q$ which is obtained by removing all points of $G/Q$ that are ``not transverse enough'' (in some precise sense) to the limit set of $\Gamma$ in $G/P$.
This phenomenon was then investigated and described in full generality by Kapovich, Leeb, and Porti \cite{klp18}.
Let us give one concrete example.

\begin{ex} \label{ex:DoD-SO-Sp}
Let $b$ be a nondegenerate symmetric bilinear form on~$\R^d$ with noncompact automorphism group $G := \mathrm{Aut}(b) \subset \SL^{\pm}(d,\R)$.
(If $b$ is symmetric, then $G = \OO(p,q)$ for some $p,q\geq 1$; we require $p$ and~$q$ to be distinct.
If $b$ is skew-symmetric, then $d = 2n$ is even and $G = \Sp(2n,\R)$.)
Let $\Gamma_0$ be an infinite Gromov hyperbolic group, $\rho : \Gamma_0\to G\subset\SL^{\pm}(d,\R)$ a $P_1$-Anosov representation, and $\Lambda_{\rho(\Gamma_0)}$ the limit set of $\rho(\Gamma_0)$ in $\PP(\R^d)$.
Let $\mathcal{L}$ be the space of maximal $b$-isotropic subspaces of~$\R^d$.
(It identifies with $G/Q$ where $Q$ is the stabiliser in~$G$ of a maximal $b$-isotropic subspace of~$\R^d$.)
Then $\Gamma_0$ acts properly discontinuously with compact quotient, via~$\rho$, on
$$\Omega_{\rho(\Gamma_0)} := \mathcal{L} \smallsetminus \bigcup_{z\in\Lambda_{\rho(\Gamma_0)}} \mathcal{L}_z,$$
where $\mathcal{L}_z$ is the set of maximal $b$-isotropic subspaces of~$\R^d$ that contain the line~$z$.
\end{ex}

When $b$ is skew-symmetric, \ie $G = \Sp(2n,\R)$, the set $\mathcal{L}$ is the space $\mathrm{Lag}(\R^{2n})$ of \emph{Lagrangians} of~$\R^{2n}$.
In this setting, if $\rho(\Gamma_0)$ is a ``strong'' crooked Schottky group as in Section~\ref{subsec:ping-pong-higher-rank}, defined by $B_1^{\pm},\dots,B_m^{\pm}$ with pairwise disjoint closures, then the set $\Omega_{\rho(\Gamma_0)}$ of Example~\ref{ex:DoD-SO-Sp} coincides with the set $\Omega = \mathrm{Int}(\bigcup_{\gamma\in\Gamma_0} \rho(\gamma)\cdot\mathcal{D}) \subset \mathrm{Lag}(\R^{2n})$ of Section~\ref{subsec:ping-pong-higher-rank}.

\subsection{Anosov representations and convex cocompactness} \label{subsec:Ano-cc}

Recall that when $\Rrank(G)=1$, Anosov representations coincide with convex cocompact representations in the classical sense of Definition~\ref{def:cc-gf-rank-1}.
When $\Rrank(G)\geq 2$, Theorem~\ref{thm:charact-Ano} shows that Anosov representations have a number of similarities, in terms of their dynamics, with rank-one convex cocompact representations: see Remark~\ref{rem:compare-charact-Ano-cc}.
Another similarity, of a more geometric nature, is the existence of cocompact domains of discontinuity as in Example~\ref{ex:DoD-SO-Sp}: given an Anosov representation $\rho : \Gamma_0\to G$, such a domain of discontinuity $\Omega \subset G/Q$ yields, by taking the quotient, a closed manifold $\rho(\Gamma_0)\backslash\Omega$ locally modeled on $G/Q$, whose geometry can be quite interesting (see \cite[\S\,5]{wie-icm}).
These manifolds $\rho(\Gamma_0)\backslash\Omega$ do not satisfy any kind of convexity properties in general.

Given these similarities, it is natural to wonder if Anosov representations could also be characterised geometrically in terms of some suitable notion of convex cocompactness.
We will see below that this is indeed the case.
This will give more geometric intuition about Anosov representations, and yield new examples constructed geometrically.

\subsubsection{Two attempts}

Our starting point is the following special case of Fact~\ref{fact:cc-embed-Ano}.

\begin{fact} \label{guiding-fact}
Let $\Gamma_0$ be an infinite group and $\rho : \Gamma_0\to\PO(n,1)=\mathrm{Isom}(\HH^n)$ a representation.
Then $\rho$ is convex cocompact (Definition~\ref{def:cc-gf-rank-1}) if and only if $\Gamma_0$ is Gromov hyperbolic and $\rho : \Gamma_0\to\PO(n,1)\hookrightarrow\PGL(n+1,\R)$ is $P_1$-Anosov.
\end{fact}

We would like to generalise this equivalence to higher-rank semisimple Lie groups~$G$.

A natural first attempt would be to replace $\HH^n$ by the Riemannian symmetric space of~$G$.
However, this turns out to be rather restrictive: Kleiner--Leeb~\cite{kl06} and Quint \cite{qui05} proved that if $G$ is a real simple Lie group of real rank $\geq 2$, with Riemannian symmetric space $X = G/K$, then any Zariski-dense discrete subgroup of~$G$, acting with compact quotient on some nonempty convex subset of~$X$, is a cocompact lattice in~$G$; in particular, $\Gamma$ is \emph{not} Gromov hyperbolic and $\rho$ is \emph{not} Anosov.
Thus this approach does not provide a generalisation of Fact~\ref{guiding-fact}.

Instead, we make a second attempt by viewing $\HH^n$ as a properly convex open set in projective space as in \eqref{eqn:proj-model-Hn}; we can then try to generalise Fact \ref{guiding-fact} by replacing $\HH^n$ with any properly convex open subset $\Omega$~of~$\PP(\R^{n+1})$.

Recall from the proof of Theorem~\ref{thm:Hitchin-comp} that $\Omega$ being \emph{properly convex} means that it is convex and bounded in some affine chart of $\PP(\R^{n+1})$.
In this setting $\Omega$ carries a natural proper metric $\mathtt{d}_{\Omega}$, the \emph{Hilbert metric}, which is invariant under $\mathrm{Aut}(\Omega) := \{ g\in\PGL(n+1,\R) \,|\, g\cdot\Omega = \Omega\}$ (see Figure~\ref{fig:Hilbert-dist}).
In particular, any discrete subgroup of $\mathrm{Aut}(\Omega)$ acts properly discontinuously on~$\Omega$.

\begin{figure}[h!]
\centering
\begin{tikzpicture}
\filldraw[draw=black,fill=gray!20]
 plot[smooth,samples=200,domain=0:pi] ({4*cos(\x r)*sin(\x r)},{-4*sin(\x r)});
cycle;
\draw (-1,-2.5) node[anchor=north west] {$x$};
\fill [color=black] (-1,-2.5) circle (2.5pt);
\draw (1,-2) node[anchor=north west] {$y$};
\fill [color=black] (1,-2) circle (2.5pt);
\draw [smooth,samples=200,domain=-4:4] plot ({\x},{0.25*\x-2.25});
\draw (-1.97,-2.75) node[anchor=north west] {$a$};
\fill [color=black] (-1.98,-2.75) circle (2.5pt);
\draw (1.73,-1.75) node[anchor=north west] {$b$};
\fill [color=black] (1.63,-1.85) circle (2.5pt);
\draw (0.7,-3) node[anchor=north west] {$\Omega$};
\end{tikzpicture}
\caption{In a properly convex open subset $\Omega$ of $\PP(\R^d)$, the \emph{Hilbert distance} between two distinct points $x,y\in\Omega$ is given by $\mathtt{d}_{\Omega}(x,y) := \frac{1}{2} \log \, [a,x,y,b]$, where $[\cdot,\cdot,\cdot,\cdot]$ is the cross-ratio on $\PP^1(\R)$, normalised so that $[0,1,y,\infty]=y$, and $a,b$ are the intersection points of $\partial\Omega$ with the projective line through $x$ and~$y$, with $a,x,y,b$ in this order. The Hilbert metric $\mathtt{d}_{\Omega}$ coincides with the hyperbolic metric when $\Omega = \HH^n$ as in \eqref{eqn:proj-model-Hn}, but in general $\mathtt{d}_{\Omega}$ is not Riemannian, only Finsler.}
\label{fig:Hilbert-dist}
\end{figure}
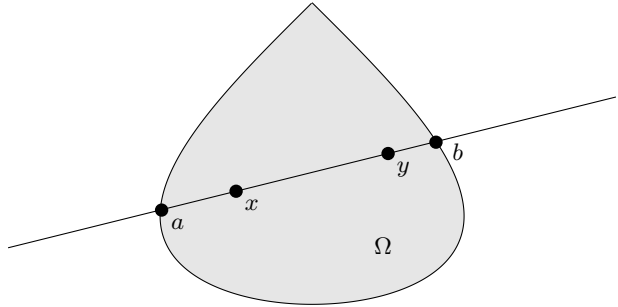

Note that $\HH^n$, viewed as a properly convex open subset of $\PP(\R^{n+1})$, does not contain any nontrivial projective segments in its boundary.
For properly convex open sets $\Omega$ with this property (also known as \emph{strictly convex} open sets), we consider the following analogue of Definition~\ref{def:cc-gf-rank-1}.

\begin{defn} \label{def:strong-proj-cc}
Let $\Omega$ be a properly convex open subset of $\PP(\R^d)$, whose boundary $\partial\Omega$ does not contain any nontrivial projective segments.
Let $\Gamma_0$ be a group and $\rho : \Gamma_0\to\mathrm{Aut}(\Omega) \subset \PGL(d,\R)$ a representation.
We say that the action of $\Gamma_0$ on $\Omega$ via~$\rho$ is \emph{convex cocompact} if it is properly discontinuous and if there exists a nonempty $\rho(\Gamma_0)$-invariant convex subset $\mathcal{C}$ of~$\Omega$ such that $\rho(\Gamma_0)\backslash\mathcal{C}$ is compact.
\end{defn}

In that case, the representation $\rho$ has finite kernel and discrete image and, as in Remark~\ref{rem:cc-gf-finitely-gen}, the group $\Gamma_0$ is finitely generated.

Similarly to Fact~\ref{fact:cc-gf-rk-1-Ccore}, we can rephrase convex cocompactness in terms of some specific convex set in~$\Omega$.
Namely, define the \emph{orbital limit set} $\Lambdao_{\rho(\Gamma_0)}(\Omega)$ of $\rho(\Gamma_0)$ in~$\Omega$ to be the set of accumulation points in $\partial\Omega$ of some $\rho(\Gamma_0)$-orbit of~$\Omega$; one easily checks that $\Lambdao_{\rho(\Gamma_0)}(\Omega)$ does not depend on the choice of $\rho(\Gamma_0)$-orbit, because $\partial\Omega$ does not contain any nontrivial segments.
Define the \emph{convex core} $\Ccore_{\rho(\Gamma_0)}(\Omega) \subset \Omega$ of $\rho(\Gamma_0)$ to be the convex hull of $\Lambdao_{\rho(\Gamma_0)}(\Omega)$ in~$\Omega$ (\ie the smallest closed convex subset of~$\Omega$ whose closure in $\PP(\R^d)$ contains~$\Lambdao_{\rho(\Gamma_0)}$).
Similarly to Fact~\ref{fact:cc-gf-rk-1-Ccore}, for infinite~$\Gamma_0$, the action of $\Gamma_0$ on $\Omega$ via~$\rho$ is then convex cocompact if and only if it is properly discontinuous and $\rho(\Gamma_0)\backslash\Ccore_{\rho(\Gamma_0)}(\Omega)$ is compact and nonempty.

The following result is a generalisation of Fact~\ref{guiding-fact} in this setting.
It was first proved in \cite{dgk18} for representations $\rho$ with values in $\PO(p,q)$, and then in general in \cite{dgk-proj-cc} and independently (in a slightly different form and under some irreducibility assumption) in \cite{zim}.
See also \cite{bm12,ben04,cm14,mes90} for related earlier results.

\begin{thm} \label{thm:Ano-strong-cc}
Let $\Gamma_0$ be an infinite group and $\rho : \Gamma_0\to\PGL(d,\R)$ a representation.
Suppose that $\rho(\Gamma_0)$ preserves a nonempty properly convex open subset of $\PP(\R^d)$.
Then the following are equivalent:
\begin{enumerate}[(1)]
  \item\label{item:Ano-strong-cc-1} $\Gamma_0$ is Gromov hyperbolic and $\rho : \Gamma_0\to\PGL(d,\R)$ is $P_1$-Anosov;
  \item\label{item:Ano-strong-cc-2} $\rho$ is \emph{strongly convex cocompact in $\PP(\R^d)$}: the group $\Gamma_0$ acts convex cocompactly (Definition~\ref{def:strong-proj-cc}) via~$\rho$ on some properly convex open subset $\Omega$ of $\PP(\R^d)$ such that $\partial\Omega$ is $C^1$ and contains no segments.
\end{enumerate}
\end{thm}

Here we say that $\partial\Omega$ is $C^1$ if every point of $\partial\Omega$ has a unique supporting hyperplane.
The phrase \emph{strongly convex cocompact} is meant to reflect the strong regularity imposed on $\partial\Omega$ (namely, $C^1$ and no segments).

\subsubsection{A few comments on Theorem~\ref{thm:Ano-strong-cc}}

In certain situations, the assumption in Theorem~\ref{thm:Ano-strong-cc} that $\rho(\Gamma_0)$ preserve a properly convex open subset of $\PP(\R^d)$ is automatically satisfied for $P_1$-Anosov representations~$\rho$.
For instance, this is the case when $\di\Gamma_0$ is connected and $\rho$ takes values in $\PO(p,q)\subset\PGL(d,\R)$ for some $p,q\geq 1$ with $p+q=d$, by \cite{dgk18}.
In this case, the $\rho(\Gamma_0)$-invariant properly convex open set $\Omega$ given by Theorem~\ref{thm:Ano-strong-cc}.\eqref{item:Ano-strong-cc-2} can be taken in
$$\{ [v]\in\nolinebreak\PP(\R^{p+q}) \,|\, \langle v,v\rangle_{p,q} <\nolinebreak 0\} = \HH^{p,q-1}$$
(we then say that $\rho$ is \emph{$\HH^{p,q-1}$-convex cocompact}) or in
$$\{ [v]\in\nolinebreak\PP(\R^{p+q}) \,|\, -\nolinebreak\langle v,v\rangle_{p,q}<\nolinebreak 0\} \simeq \HH^{q,p-1}$$
(we then say that $\rho$ is \emph{$\HH^{q,p-1}$-convex cocompact}), where $\langle\cdot,\cdot\rangle_{p,q}$ is the symmetric bilinear form of signature $(p,q)$ on~$\R^{p+q}$ defining $\PO(p,q)$.

On the other hand, there exist $P_1$-Anosov representations that do not preserve any properly convex open subset of $\PP(\R^d)$: \eg Hitchin representations (see Sections \ref{subsec:deform-Fuchsian-higher-rank} and~\ref{subsec:Anosov}) into $\PSL(d,\R)$ for even~$d$.

However, one can always reduce to preserving a properly convex open set by considering a larger projective space.
Indeed, consider the natural action of $\GL(d,\R)$ on the vector space $\mathrm{Sym}_d(\R)$ of symmetric $(d\times d)$ real matrices by $g\cdot A = gAg^T$.
It induces a representation $\tau : \PGL(d,\R)\to\PGL(\mathrm{Sym}_d(\R))$, which preserves the open subset $\Omega_{\mathrm{sym}}$ of $\PP(\mathrm{Sym}_d(\R))$ corresponding to positive definite symmetric matrices.
The set $\Omega_{\mathrm{sym}}$ is properly convex.
One can check (see \cite{gw12}, or use one of the characterisations of Theorem~\ref{thm:charact-Ano}) that a representation $\rho : \Gamma_0\to\PGL(d,\R)$ is $P_1$-Anosov if and only if $\tau\circ\rho : \Gamma_0\to\PGL(\mathrm{Sym}_d(\R))$ is $P_1$-Anosov.
Theorem~\ref{thm:Ano-strong-cc} then implies the following.
  
\begin{cor} \label{cor:Ano-strong-cc-no-hyp}
For any infinite group $\Gamma_0$ and any representation $\rho : \Gamma_0\to\PGL(d,\R)$, the following are equivalent:
\begin{enumerate}[(1)]
  \item $\Gamma_0$ is Gromov hyperbolic and $\rho : \Gamma_0\to\PGL(d,\R)$ is $P_1$-Anosov;
  \item $\tau\circ\rho$ is strongly convex cocompact in $\PP(\mathrm{Sym}_d(\R))$.
\end{enumerate}
\end{cor}

This actually yields a characterisation of $P$-Anosov representations into~$G$ for any proper parabolic subgroup $P$ of any noncompact semisimple Lie group~$G$, by considering an appropriate representation of $G$ to some large projective linear group.
For instance, for $G = \PGL(d,\R)$ and $P = P_i$ with $1\leq i\leq d-1$ as in Section~\ref{subsec:Anosov}, we can consider the natural representation $\tau_i : \PGL(d,\R)\to\PGL(S^2(\Lambda^i\R^d))$ where $S^2(\Lambda^i\R^d)$ is the second symmetric power of the $i$-th exterior power of the standard representation of $\GL(d,\R)$ on~$\R^d$.
(For $i=1$, this identifies with $\tau : \PGL(d,\R)\to\PGL(\mathrm{Sym}_d(\R))$ above.)
Again, one can check that $\rho : \Gamma_0\to\PGL(d,\R)$ is $P_i$-Anosov if and only if $\tau_i\circ\rho : \Gamma_0\to\PGL(S^2(\Lambda^i\R^d))$ is $P_1$-Anosov.
Theorem~\ref{thm:Ano-strong-cc} then implies the following.

\begin{cor} \label{cor:i-Ano-strong-cc-no-hyp}
For any infinite group $\Gamma_0$, any representation $\rho : \Gamma_0\to\PGL(d,\R)$, and any $1\leq i\leq d-1$, the following are equivalent:
\begin{enumerate}[(1)]
  \item $\Gamma_0$ is Gromov hyperbolic and $\rho : \Gamma_0\to\PGL(d,\R)$ is $P_i$-Anosov;
  \item $\tau_i\circ\rho$ is strongly convex cocompact in $\PP(S^2(\Lambda^i\R^d))$.
\end{enumerate}
\end{cor}

\subsubsection{Sketch of proof of Theorem~\ref{thm:Ano-strong-cc}}

\begin{proof}[Proof of \eqref{item:Ano-strong-cc-1}~$\Rightarrow$~\eqref{item:Ano-strong-cc-2}:]
Suppose that $\Gamma_0$ is Gromov hyperbolic, that $\rho$ is $P_1$-Anosov with boundary maps $\xi_1 : \di\Gamma_0\to\mathrm{Gr}_1(\R^d)=\PP(\R^d)$ and $\xi_{d-1} : \di\Gamma_0\to\mathrm{Gr}_{d-1}(\R^d)=\PP((\R^d)^*)$, and that $\rho(\Gamma_0)$ preserves a nonempty properly convex open subset $\Omega$ of $\PP(\R^d)$.
Since $\Omega$ was chosen without care, it is possible that $\partial\Omega$ contains segments or that the action of $\Gamma_0$ on $\Omega$ via~$\rho$ is not convex cocompact.
Therefore, we do not work with~$\Omega$ itself, but consider instead the connected component $\Omega_{\max}$ of $\PP(\R^d)\smallsetminus\bigcup_{w\in\di\Gamma_0} \xi_{d-1}(w)$ containing~$\Omega$ (where we view each $\xi_{d-1}(w)$ as a projective hyperplane in $\PP(\R^d)$); it is $\rho(\Gamma_0)$-invariant, open, and convex (not necessarily bounded) in some affine chart of $\PP(\R^d)$.
Using Lemma~\ref{lem:conv-G/P}, one can show that the action of $\Gamma_0$ on $\Omega_{\max}$ via~$\rho$ is properly discontinuous, and that the set of accumulation points of any $\rho(\Gamma_0)$-orbit of $\Omega_{\max}$ is $\xi_1(\di\Gamma_0)$.

Consider the convex hull $\mathcal{C}$ of $\xi_1(\di\Gamma_0)$ in $\Omega_{\max}$.
One easily checks, using the transversality of $\xi_1$ and~$\xi_{d-1}$, that $\xi_1(\di\Gamma_0)$ is not contained in a single supporting hyperplane to $\Omega_{\max}$ in $\PP(\R^d)$, and therefore that $\mathcal{C}$ is nonempty.
Using the expansion property \eqref{item:i-expand} of Theorem~\ref{thm:charact-Ano} for Anosov representations, a similar reasoning to the proof of \eqref{item:rk1-expand}~$\Rightarrow$~\eqref{item:rk1-cc} in Section~\ref{subsec:charact-cc-rank-1} then shows that $\rho(\Gamma_0)\backslash\mathcal{C}$ is compact: see \cite[\S\,8]{dgk-proj-cc}.

By transversality of $\xi_1$ and $\xi_{d-1}$, there are no nontrivial segments in $\partial\Omega_{\max}$ between points of $\xi_1(\di\Gamma_0)$.
This makes it possible to ``smooth out'' $\Omega_{\max}$ to obtain a $\rho(\Gamma_0)$-invariant properly convex open subset $\Omega' \subset \Omega_{\max}$ containing~$\mathcal{C}$ such that $\partial\Omega'$ is $C^1$ and contains no segments: see \cite[\S\,9]{dgk-proj-cc}.
The action of $\Gamma_0$ on $\Omega'$ via~$\rho$ is convex cocompact as desired.
\end{proof}

\begin{proof}[Proof of \eqref{item:Ano-strong-cc-2}~$\Rightarrow$~\eqref{item:Ano-strong-cc-1}:]
Suppose that $\Gamma_0$ acts convex cocompactly via~$\rho$ on some properly convex open subset $\Omega$ of $\PP(\R^d)$ such that $\partial\Omega$ is $C^1$ and contains no segments.
Because $\partial\Omega$ contains no segments, the geodesic rays of~$\Omega$ for the Hilbert metric $\mathtt{d}_{\Omega}$ (Figure~\ref{fig:Hilbert-dist}) are exactly the projective segments between a point of~$\Omega$ and a point of $\partial\Omega$, and two such rays remain at bounded Hausdorff distance for $\mathtt{d}_{\Omega}$ if and only if their endpoints in $\partial\Omega$ are the same.
Therefore the convex core $\Ccore_{\rho(\Gamma_0)}$, endowed with the restriction of $\mathtt{d}_{\Omega}$, is a geodesic metric space whose visual boundary $\di\Ccore_{\rho(\Gamma_0)}$ identifies with its ideal boundary $\overline{\Ccore_{\rho(\Gamma_0)}} \cap \partial\Omega$ in $\partial\Omega$.

Using the fact that $\partial\Omega$ contains no segments, one can check by a limiting argument that all triangles in $\Ccore_{\rho(\Gamma_0)}$ must be uniformly thin, \ie the metric space $(\Ccore_{\rho(\Gamma_0)},\mathtt{d}_{\Omega})$ is Gromov hyperbolic: see \cite[Lem.\,6.3]{dgk-proj-cc}.
Since the action of $\Gamma_0$ on $(\Ccore_{\rho(\Gamma_0)},\mathtt{d}_{\Omega})$ via~$\rho$ is properly discontinuous, by isometries, with compact quotient, we deduce that $\Gamma_0$ is Gromov hyperbolic and (as in the proof of \eqref{item:rk1-cc}~$\Rightarrow$~\eqref{item:rk1-bound-map-strong-dyn-preserv} in Section~\ref{subsec:charact-cc-rank-1}) that any orbital map $\Gamma_0\to\Ccore_{\rho(\Gamma_0)}$ extends to a continuous $\rho$-equivariant boundary map $\xi_1 : \di\Gamma_0\to\di\Ccore_{\rho(\Gamma_0)} \subset \PP(\R^d)$.

Consider the dual $\Omega^* = \{H\in\PP((\R^d)^*) \,|\, H\cap\overline{\Omega}=\emptyset\}$ of~$\Omega$ (where we view $\PP((\R^d)^*)$ as the set of projective hyperplanes in $\PP(\R^d)$).
It is a properly convex open subset of $\PP((\R^d)^*)$.
The boundary $\partial\Omega^*$ of~$\Omega^*$ is $C^1$ (because $\partial\Omega$ contains no segments), and it contains no segments (because $\partial\Omega$ is $C^1$).
One can show that the dual action of $\Gamma_0$ on~$\Omega^*$ via~$\rho$ is still convex cocompact: see \cite[\S\,5]{dgk-proj-cc}.
Then the same reasoning as above yields a continuous $\rho$-equivariant boundary map $\xi_{d-1} : \di\Gamma_0\to\PP((\R^d)^*)$.

By construction, $\xi_1$ and~$\xi_{d-1}$ are transverse: indeed, $\xi_{d-1}(w)$ is a supporting hyperplane to~$\Omega$ at $\xi_1(w)$ for any~$w$, and $\partial\Omega$ contains no segments.
One checks that $\xi_1$ and $\xi_{d-1}$ are dynamics-preserving and (using Lemma~\ref{lem:conv-G/P}) that $(\mu_1-\mu_2)(\rho(\gamma))\to +\infty$ as $\mathtt{d}_{\mathrm{Cay}(\Gamma_0)}(e,\gamma)\to +\infty$: see \cite[\S\,7]{dgk-proj-cc}.
We then apply the implication \eqref{item:i-bound-map-dyn-preserv}~$\Rightarrow$~\eqref{item:Ano} of Theorem~\ref{thm:charact-Ano}.
\end{proof}

\subsubsection{Applications}

Theorem~\ref{thm:Ano-strong-cc} and Corollaries \ref{cor:Ano-strong-cc-no-hyp}--\ref{cor:i-Ano-strong-cc-no-hyp} give geometric interpretations for Anosov representations.

\begin{ex}
For odd $d$, any Hitchin representation $\rho : \pi_1(S)\to\PSL(d,\R)$ as in Section~\ref{subsec:deform-Fuchsian-higher-rank} preserves a nonempty properly convex open subset of $\PP(\R^d)$ (see \cite{dgk18,dgk-proj-cc,zim}).
Therefore these representations are strongly convex cocompact in $\PP(\R^d)$ by Theorem~\ref{thm:Ano-strong-cc}.
This extends the case $d=3$ due to Choi and Goldman (see the proof of Theorem~\ref{thm:Hitchin-comp}).
\end{ex}

\begin{ex}
For $n\geq 2$, any maximal representation $\rho : \pi_1(S)\to\SO(2,n)$ as in Section~\ref{subsec:deform-Fuchsian-higher-rank} preserves a nonempty properly convex open subset of $\PP(\R^{n+2})$, contained in $\HH^{2,n-1} = \{ [v]\in\PP(\R^{n+2}) \,|\, \langle v,v\rangle_{2,n} <\nolinebreak 0\}$ (see \cite{ctt19,dgk18}).
Therefore these representations are strongly convex cocompact in $\PP(\R^{n+2})$ by Theorem~\ref{thm:Ano-strong-cc}, and in fact \emph{$\HH^{2,n-1}$-convex cocompact} as in Section~\ref{subsec:higher-Teich} (see the comments after Theorem~\ref{thm:Ano-strong-cc}).
\end{ex}

Theorem~\ref{thm:Ano-strong-cc} can also be used to construct new examples of Anosov representations.
One source of examples comes from representations of Coxeter groups as linear reflection groups.
Recall that a Coxeter group is a group with a presentation by generators and relations of the form
\begin{equation} \label{eqn:Coxeter-group}
W = \langle s_1,\dots,s_N ~|~ (s_i s_j)^{m_{i,j}} = e \quad\forall 1\leq i,j\leq N\rangle
\end{equation}
where $m_{i,i} = 1$ (\ie $s_i$ is an involution) and $m_{i,j}\in\{2,3,4,\dots\}\cup\{\infty\}$ for all $i\neq j$.
(By convention, $(s_i s_j)^{\infty}=e$ means that $s_is_j$ has infinite order in the group~$W$.)
Vinberg \cite{vin71} developed a theory of representations of~$W$ \emph{as a reflection group} in a finite-dimensional real vector space~$V$: these are by definition representations $\rho : W\to\GL(V)$ such that each $\rho(s_i)$ is a linear reflection in a hyperplane of~$V$ and the configuration of these reflections is such that $\rho$ is injective, discrete, and the associated fundamental polytope has nonempty interior.
These representations may preserve a nondegenerate quadratic form on~$V$ (\eg the image of~$\rho$ could be a discrete subgroup of $\OO(n,1)$ generated by orthogonal reflections in the faces of a right-angled polyhedron of~$\HH^n$ as in Section~\ref{subsec:ex-rk1-cc-gf}), but in general they need not preserve any nonzero quadratic form.
Representations of $W$ as a reflection group constitute a subset $\Hom_{\mathrm{refl}}(W,\GL(V))$ of $\Hom(W,\GL(V))$ which is semialgebraic (defined by finitely many equalities and inequalities).

\begin{ex}[{\cite{dgk18,dgklm,lm19}}] \label{ex:hyp-Cox-cc}
Let $W$ be a Coxeter group in $N$ generators as in \eqref{eqn:Coxeter-group}.
Suppose $W$ is infinite and Gromov hyperbolic.
Then for any $d\geq N$ there exist representations $\rho : W\to\SL^{\pm}(d,\R)$ of $W$ as a reflection group which are strongly convex cocompact in $\PP(\R^d)$; for $d\geq 2N-2$, they constitute the full interior of $\Hom_{\mathrm{refl}}(W,\GL(d,\R))$.
By Theorem~\ref{thm:Ano-strong-cc}, these representations are $P_1$-Anosov.
\end{ex}

By \cite{dfwz}, a conclusion similar to that of Example~\ref{ex:hyp-Cox-cc} holds if $W$ is an infinite Gromov hyperbolic group which is not necessarily a Coxeter group, but which embeds into a right-angled Coxeter group as a so-called \emph{quasiconvex subgroup}.
Using celebrated work of Agol and Haglund--Wise, this provides Anosov representations for a large class of infinite Gromov hyperbolic groups, namely all those which admit a properly discontinuous and cocompact action on a CAT(0) cube complex.

One can also use the geometric interpretation of Anosov representations from Theorem~\ref{thm:Ano-strong-cc} to prove that free products $\Gamma_1 * \Gamma_2$ of Anosov subgroups $\Gamma_1,\Gamma_2$ are Anosov \cite{dgk-ex-cc}, using a generalisation of the ping pong arguments of Sections \ref{subsec:ex-rank-1}--\ref{subsec:ping-pong-higher-rank}.
For instance, for $1\leq i\leq d-1$, let $\tau_i : \SL(d,\R)\to\SL(S^2(\Lambda^i\R^d))$ be the second symmetric power of the $i$-th exterior power of the standard representation as in Corollary~\ref{cor:i-Ano-strong-cc-no-hyp}, let $V'_i := S^2(\Lambda^i\R^d)\oplus\R$, and let $\tau'_i : \SL(d,\R) \to \SL(V'_i)$ be the direct sum of $\tau_i$ and of the trivial representation.
Then the following holds.

\begin{ex}[\cite{dgk-ex-cc}]
Let $1\leq i\leq d-1$ and let $\Gamma_1,\Gamma_2$ be any discrete subgroups of $\SL(d,\R)$.
Then there exists $g\in\SL(V'_i)$ such that the representation $\rho : \Gamma_1 * g\Gamma_2g^{-1} \to \SL(V'_i)$ induced by the restrictions of $\tau'_i$ to $\Gamma_1$ and $g\Gamma_2g^{-1}$ has finite kernel and discrete image.
If moreover $\Gamma_1$ and~$\Gamma_2$ are $P_i$-Anosov, then we can choose $g$ so that $\rho$ is $P_1$-Anosov.
\end{ex}

(Note that beyond Anosov representations, this construction can be used to prove that the free product of two $\Z$-linear groups is $\Z$-linear, and that there exist Zariski-dense discrete subgroups of $\SL(V'_i)$ which are not lattices but contain cocompact lattices of $\tau'_i(\SL(d-1,\R))$: see~\cite{dgk-ex-cc}.)

We refer to \cite{dk-maskit,dk-maskit-bis} for other combination theorems for Anosov representations which do not use Theorem~\ref{thm:Ano-strong-cc}.

Finally we note that, although we have seen many constructions of Anosov representations above, it is expected that not every linear Gromov hyperbolic group admits an Anosov representation into some noncompact semisimple Lie group; a concrete example remains to be found.

\subsection{Generalisations of Anosov subgroups}

In the past few years, several fruitful generalisations of Anosov subgroups have appeared, which are currently being actively investigated.
These generalisations exploit both the dynamical definition of Anosov subgroups from Section~\ref{subsec:Anosov} and their geometric characterisation from Section~\ref{subsec:Ano-cc}.
Let us briefly mention three of these generalisations.

\subsubsection{More general convex cocompact subgroups}

We just saw in Theorem~\ref{thm:Ano-strong-cc} and Corollaries \ref{cor:Ano-strong-cc-no-hyp} and~\ref{cor:i-Ano-strong-cc-no-hyp} that Anosov representations can be characterised geometrically by a strong convex cocompactness condition in projective space.
Here \emph{strong} refers to the regularity imposed on the properly convex open set~$\Omega$ (its boundary $\partial\Omega$ should be $C^1$ and contain no segments).

It is natural to try to generalise Anosov representations by relaxing this strong regularity requirement.
Removing it altogether in Definition~\ref{def:strong-proj-cc} leads to a notion which is not stable under small deformations (see \cite{dgk-proj-cc,dgk-ex-cc}).
Instead, we impose the following mild condition, which relies on the notions of \emph{full orbital limit set} and \emph{convex core}.

\begin{defn}[{\cite{dgk-proj-cc}}] \label{def:proj-cc}
Let $\Omega$ be a properly convex open subset of~$\PP(\R^d)$.
Let $\Gamma_0$ be a group and $\rho : \Gamma_0\to\mathrm{Aut}(\Omega)\subset\PGL(d,\R)$ a representation.
\begin{itemize}
  \item The \emph{full orbital limit set} $\Lambdao_{\rho(\Gamma_0)}(\Omega)$ of $\rho(\Gamma_0)$ in~$\Omega$ is the set of all accumulation points in $\partial\Omega$ of all possible $\rho(\Gamma_0)$-orbits of~$\Omega$.
  \item The \emph{convex core} $\Ccore_{\rho(\Gamma_0)}(\Omega) \subset \Omega$ of $\rho(\Gamma_0)$ is the convex hull of $\Lambdao_{\rho(\Gamma_0)}(\Omega)$ in~$\Omega$.
  \item The action of $\Gamma_0$ on $\Omega$ via~$\rho$ is \emph{convex cocompact} if it is properly discontinuous and if there exists a nonempty $\rho(\Gamma_0)$-invariant convex subset $\mathcal{C}$ of~$\Omega$ such that $\rho(\Gamma_0)\backslash\mathcal{C}$ is compact and $\mathcal{C}$ is ``large enough'' in the sense that it contains the convex core $\Ccore_{\rho(\Gamma_0)}(\Omega)$.
\end{itemize}
\end{defn}

Note that Definition~\ref{def:proj-cc} coincides with Definition~\ref{def:strong-proj-cc} when $\partial\Omega$ does not contain any nontrivial projective segments.
Indeed, in that case the full orbital limit set $\Lambdao_{\rho(\Gamma_0)}(\Omega)$ is the set of accumulation points of any single $\rho(\Gamma_0)$-orbit of~$\Omega$, hence any nonempty $\rho(\Gamma_0)$-invariant convex subset $\mathcal{C}$ of~$\Omega$ contains the convex core $\Ccore_{\rho(\Gamma_0)}(\Omega)$ (see the comments after Definition~\ref{def:strong-proj-cc}).

\begin{defn} \label{def:proj-cc-rep}
Given a group~$\Gamma_0$, we say that a representation $\rho : \Gamma_0\to\PGL(d,\R)$ is \emph{convex cocompact in $\PP(\R^d)$} if $\Gamma_0$ acts convex cocompactly via~$\rho$ on some properly convex open subset $\Omega$ of $\PP(\R^d)$.
In that case, we also say that the image $\rho(\Gamma_0)$ is \emph{convex cocompact in $\PP(\R^d)$}.
\end{defn}

As above, if $\rho$ is convex cocompact in $\PP(\R^d)$, then it has finite kernel and discrete image, and the group $\Gamma_0$ is finitely generated.

This notion turns out to be quite fruitful: by \cite{dgk-proj-cc}, the set of convex cocompact representations is open in $\Hom(\Gamma_0,\PGL(d,\R))$, and it is stable under duality and under embedding into a larger projective space; moreover, a representation $\rho : \Gamma_0\to\PGL(d,\R)$ is strongly convex cocompact in $\PP(\R^d)$ in the sense of Theorem~\ref{thm:Ano-strong-cc}.\eqref{item:Ano-strong-cc-2} if and only if it is convex cocompact in $\PP(\R^d)$ and $\Gamma_0$ is Gromov hyperbolic.
Theorem~\ref{thm:Ano-strong-cc} then shows that convex cocompact representations are generalisations of $P_1$-Anosov representations, for finitely generated infinite groups $\Gamma_0$ that are not necessarily Gromov hyperbolic, and that may therefore contain subgroups isomorphic to $\Z^2$ (see Remark~\ref{rem:hyp-gp-no-Z2}).

In fact, Weisman \cite{wei23} has recently given a dynamical characterisation of convex cocompact representations of~$\Gamma_0$ that extends the characterisation of Anosov representations of Theorem~\ref{thm:charact-Ano}.\eqref{item:i-expand}.
The expansion now takes place in various Grassmannians (not only projective space): namely, at each face of the full orbital limit set in $\partial\Omega$, there is expansion in the Grassmannian of $i$-planes of~$\R^d$ where $i-1$ is the dimension of the face.

We conclude this section by mentioning a few examples of convex cocompact groups that are not necessarily Gromov hyperbolic (\ie that are not necessarily Anosov subgroups).

\begin{ex} \label{ex:conv-div}
Let $\Gamma$ be a discrete subgroup of $\PGL(d,\R)$ \emph{dividing} (\ie acting properly discontinuously with compact quotient on) some properly convex open subset $\Omega$ of $\PP(\R^d)$.
Then $\Lambdao_{\rho(\Gamma_0)}(\Omega) = \partial\Omega$ and the action of $\Gamma$ on~$\Omega$ is convex cocompact.
By \cite{ben04}, the group $\Gamma$ is Gromov hyperbolic if and only if $\partial\Omega$ contains no segments.
Examples where $\partial\Omega$ contains segments include the symmetric divisible convex sets $\Omega_{\mathrm{sym}} \subset \PP(\mathrm{Sym}_{d'}(\R)) \simeq \PP(\R^d)$ with $d=d'(d'+1)/2\geq 6$ discussed before Corollary~\ref{cor:Ano-strong-cc-no-hyp}.
The first nonsymmetric irreducible examples were constructed in small dimensions ($4\leq d\leq 7$) by Benoist \cite{ben06}; examples in all dimensions $d\geq 4$ were recently constructed by Blayac and Viaggi~\cite{bv}.
\end{ex}

\begin{ex} \label{ex:conv-div-embed}
For $\Gamma$ dividing~$\Omega$ as in Example~\ref{ex:conv-div}, we can lift $\Gamma$ to a subgroup $\hat{\Gamma}$ of $\SL^{\pm}(d,\R)$ preserving a properly convex cone of $\R^d$ lifting~$\Omega$, and then embed $\hat{\Gamma}$ into $\PGL(D,\R)$ for some $D\geq d$.
By the result of \cite{dgk-proj-cc} mentioned above, the discrete subgroup of $\PGL(D,\R)$ obtained in this way will be convex cocompact in $\PP(\R^D)$; moreover, it will remain convex cocompact in $\PP(\R^D)$ after any small deformation in $\PGL(D,\R)$.
\end{ex}

Recall that, given a Coxeter group $W$ as in \eqref{eqn:Coxeter-group}, a subgroup of~$W$ is called \emph{standard} if it is generated by a subset of the generating set $\{ s_1,\dots,s_N\}$.
The Coxeter group $W$ is called \emph{affine} if it is irreducible (\ie it cannot be written as a direct product of two nontrivial Coxeter groups) and if it is virtually (\ie it admits a finite-index subgroup which is) isomorphic to $\Z^k$ for some $k\geq\nolinebreak 1$.
Affine Coxeter groups have been completely classified; they include the Coxeter groups of type $\tilde{A}_k$ (which are virtually isomorphic to~$\Z^k$), where we say that $W$ is of type $\tilde{A}_1$ if $N=2$ and $m_{1,2} = \infty$, and $W$ is of type $\tilde{A}_{N-1}$ for $N\geq 3$ if $m_{i,j} = 3$ for all $i\neq j$ with $|i-j| = 1$ mod~$N$ and $m_{i,j} = 2$ for all other $i\neq j$.

\begin{ex}[{\cite{dgklm}}] \label{ex:cg-cc}
As a generalisation of Example~\ref{ex:hyp-Cox-cc}, let $W$ be a Coxeter group in $N$ generators as in \eqref{eqn:Coxeter-group}.
Suppose $W$ is infinite.
Then there exists a representation $\rho\in\Hom_{\mathrm{refl}}(W,\GL(d,\R))$ which is convex cocompact in $\PP(\R^d)$ for some~$d$ if and only if any affine standard subgroup of~$W$ is of type $\tilde{A}_k$ for some $k\geq 1$ and $W$ does not contain a direct product of two infinite standard subgroups.
If this holds, then we can take any $d\geq N$ and the convex cocompact representations then constitute a large open subset of $\Hom_{\mathrm{refl}}(W,\GL(d,\R))$: see \cite[\S\,1.5]{dgklm}.
\end{ex}

Examples \ref{ex:conv-div}, \ref{ex:conv-div-embed}, and \ref{ex:cg-cc} provide many convex cocompact groups which are not Gromov hyperbolic.
(In Example~\ref{ex:cg-cc}, the group $W$ is nonhyperbolic as soon as it contains an affine standard subgroup of type $\tilde{A}_k$ with $k\geq 2$, see Remark~\ref{rem:hyp-gp-no-Z2}.)

Some of these groups are still relatively hyperbolic: \eg in Example~\ref{ex:cg-cc}, the group $W$ is relatively hyperbolic with respect to a collection of virtually abelian subgroups of rank $\geq 2$ (see \cite[Cor.\,1.7]{dgklm}).
We refer to \cite{iz23} for general results about the structure of relatively hyperbolic groups which are convex cocompact in $\PP(\R^d)$ and about the geometry of the associated convex sets.
On the other hand, Example~\ref{ex:conv-div} includes, for $d=d'(d'+1)/2\geq 6$, discrete subgroups of $\PGL(d,\R)$ which divide a symmetric properly convex open set $\Omega_{\mathrm{sym}} \subset \PP(\R^d) \simeq \PP(\mathrm{Sym}_{d'}(\R))$ and which are isomorphic to cocompact lattices of $\PGL(d',\R)$, hence \emph{not} relatively hyperbolic (see Section~\ref{subsec:rank}).
Further examples of convex cocompact groups which are not relatively hyperbolic can be constructed \eg using free products inside larger projective spaces: see \cite{dgk-ex-cc}.

\subsubsection{Relatively Anosov subgroups}

Kapovich--Leeb \cite{kl-rel-Ano} and Zhu \cite{zhu21,zhu23} have developed notions of a \emph{relatively Anosov representation} of a relatively hyperbolic group into a noncompact semisimple Lie group~$G$, which generalise the notion of an Anosov representation of a hyperbolic group into~$G$ from Section~\ref{subsec:Anosov}.
They obtain various characterisations similar to those of Theorem~\ref{thm:charact-Ano}.
The original definition of Anosov representations using flows (Definition~\ref{def:Ano} and Condition~\ref{cond:Ano}) is recovered in this more general setting by recent work of Zhu and Zimmer~\cite{zz-rel-Ano-1}.

Extending Fact~\ref{fact:Ano-open}, if $\Gamma_0$ is relatively hyperbolic with respect to a collection of subgroups (called \emph{peripheral subgroups}), then relatively Anosov representations of $\Gamma_0$ into a given~$G$ are stable under small deformations that preserve the conjugacy class of the image of each peripheral subgroup \cite{kl-rel-Ano,zz-rel-Ano-1}.

Any relatively Anosov representation $\rho : \Gamma_0\to G$ has finite kernel and discrete image $\rho(\Gamma_0)$, called a \emph{relatively Anosov subgroup} of~$G$.
There are many examples of relatively Anosov subgroups (see \cite{kl-rel-Ano,zz-rel-Ano-2}), including:
\begin{itemize}
  \item geometrically finite subgroups of~$G$ for $\Rrank(G) = 1$ (Definition~\ref{def:cc-gf-rank-1}),
  \item some of the Schottky groups of Section~\ref{subsec:ping-pong-higher-rank},
  \item the images of certain compositions $\tau\circ\sigma_0 : \Gamma_0\to G$ where $\sigma_0 : \Gamma_0\to G'$ is a geometrically finite representation into a semisimple Lie group~$G'$ with $\Rrank(G') = 1$ and $\tau : G'\to G$ is a representation with compact kernel (\eg Fact~\ref{fact:cc-embed-Ano} generalises to the relative setting);
  \item similarly to Section~\ref{subsec:deform-Fuchsian-higher-rank}, small deformations in~$G$ of such $\tau\circ\sigma_0(\Gamma_0)$, preserving the conjugacy class of the image of each peripheral subgroup;
  \item certain representations of $\PSL(2,\Z)$ into $\PGL(3,\R)$ constructed by Schwartz \cite{sch93} by iterating Pappus's theorem (see \cite[\S\,13]{zz-rel-Ano-2});
  \item for a finite-volume hyperbolic surface~$S$, the images of positive (in the sense of Fock--Goncharov \cite{fg06}) type-preserving representations of $\Gamma_0 = \pi_1(S)$ into a real split simple Lie group~$G$, see \cite{czz22} (for closed~$S$, these coincide with the Hitchin representations of Sections \ref{subsec:deform-Fuchsian-higher-rank}--\ref{subsec:higher-Teich});
  \item discrete subgroups of $\PGL(d,\R)$ preserving a properly convex open subset $\Omega$ of $\PP(\R^d)$ with strong regularity ($\partial\Omega$ is $C^1$ with no segments), and whose action on~$\Omega$ is geometrically finite in the sense of \cite{cm14}.
\end{itemize}

It would be interesting to determine whether relatively Anosov representations of relatively hyperbolic groups can also be fully characterised geometrically similarly to Theorem~\ref{thm:Ano-strong-cc} and Corollaries \ref{cor:Ano-strong-cc-no-hyp}--\ref{cor:i-Ano-strong-cc-no-hyp}.

\subsubsection{Extended geometrically finite subgroups}

Recently, Weisman \cite{wei-egf-def,wei-egf-ex} has introduced a notion of \emph{extended geometrically finite} (or \emph{EGF} for short) representation of a relatively hyperbolic group.
This is a dynamical notion, which extends a dynamical characterisation of Anosov representations in terms of multicones \cite{bps19}.
EGF representations include all Anosov or relatively Anosov representations, all representations of relatively hyperbolic groups which are convex cocompact in the sense of Definition~\ref{def:proj-cc-rep}, as well as other examples (see \cite[Th.\,1.5--1.7]{wei-egf-ex} and \cite[Prop.\,6.5 \& Rem.\,6.2]{bv}).
They are stable under certain small deformations, called \emph{peripherally stable}, for which the dynamics of the peripheral subgroups does not degenerate too much.

On the other hand, it would be interesting to define a general notion of geometric finiteness in convex projective geometry (involving properly convex open subsets $\Omega$ of $\PP(\R^d)$ where $\partial\Omega$ may contain segments or not be $C^1$), and to make the link with Weisman's EGF representations.
A good notion of geometric finiteness should contain as a particular case the notion of convex cocompactness from Definition~\ref{def:proj-cc}.
More precisely, a convex projective manifold $M = \Gamma\backslash\Omega$ should be geometrically finite if its convex core $\Gamma\backslash\Ccore_{\Gamma}(\Omega)$ (see Definition~\ref{def:proj-cc}) is covered by a compact piece and finitely many ends of~$M$, called \emph{cusps}, with a controlled geometry.
It is not completely clear what the right definition of a cusp should be.
Following Cooper, Long, and Tillmann \cite{clt15}, one could define a (full) cusp to be the image in~$M$ of some convex open subset of~$\Omega$ whose stabiliser in~$\Gamma$ is infinite and does not contain any hyperbolic element (\ie any element of this stabiliser has all its complex eigenvalues of the same modulus); in that case, the cusp is diffeomorphic to the direct product of $\R$ with an affine $(d-2)$-dimensional manifold called the \emph{cusp cross-section}, and the stabiliser of the cusp is virtually nilpotent \cite[Th.\,5.3]{clt15}.
The cusp is said to have \emph{maximal rank} if the cross-section is compact.
A more general notion of cusp of maximal rank, where the stabiliser may contain hyperbolic elements but is still assumed to be virtually nilpotent, was studied in \cite{bcl20}.
A notion of geometric finiteness involving only such generalised cusps of maximal rank was introduced and studied in \cite{wol-PhD}, where it was characterised in dynamical terms.
Examples (both of finite and infinite volume) were constructed in \cite{bm20,bob19} as small deformations of finite-volume real hyperbolic manifolds, using a stability result from \cite{clt18}; the corresponding representations are EGF by \cite{wei-egf-ex}.
On the other hand, the study of convex projective cusps of nonmaximal rank, possibly allowing for hyperbolic elements, is still at its infancy, and a good general notion of geometric finiteness in this setting still remains to be found, together with appropriate dynamical characterisations.

\vspace{1cm}

\noindent\textbf{Acknowledgements.}
These notes grew out of a minicourse given at Groups St Andrews 2022 in Newcastle.
I would like to warmly thank the organisers (Colin Campbell, Martyn Quick, Edmund Robertson, Colva Roney-Dou\-gal, and David Stewart) for a very stimulating and enjoyable meeting.

I am grateful to Pierre-Louis Blayac, Jean-Philippe Burelle, Jeffrey Danciger, Sami Douba, Balthazar Fl\'echelles, Ilia Smilga, J\'er\'emy Toulisse, and the referee for many useful comments on a preliminary version of this text, and to Florian Stecker for providing Figure~\ref{fig:Barbot}.

This project received funding from the European Research Council (ERC) under the European Union's Horizon 2020 research and innovation programme (ERC starting grant DiGGeS, grant agreement No. 715982).
The final proofreading was done at the Institute for Advanced Study in Princeton, supported by the National Science Foundation under Grant No. DMS-1926686.


\end{document}